\documentclass[11pt]{article}

\usepackage[utf8]{inputenc}
\usepackage{amsmath}
\usepackage[english]{babel}
\usepackage{amssymb}
\usepackage{amsthm}
\usepackage{mathtools}
\usepackage{esdiff}
\usepackage{microtype}
\usepackage{paralist}
\usepackage{skak}
\usepackage{bbm}
\usepackage[paper=a4paper,left=25mm,right=25mm,top=25mm,bottom=25mm]{geometry}
\usepackage{stmaryrd}
\usepackage{hyperref}
\usepackage{color} 
\usepackage{fancyhdr}
\usepackage{nomencl}
\usepackage{cite}
\usepackage{longtable}

\usepackage{thmtools}
\declaretheoremstyle[headfont=\normalfont\bfseries]{bfthmstyle}
\declaretheorem[numberwithin=section,style=bfthmstyle]{Theorem}
\declaretheorem[sharenumber=Theorem,style=bfthmstyle]{Lemma}
\declaretheorem[sharenumber=Theorem,style=bfthmstyle]{Proposition}
\declaretheorem[sharenumber=Theorem,style=bfthmstyle]{Corollary}

\newtheorem{prop}[Theorem]{Proposition} 

\declaretheoremstyle[headfont=\normalfont\bfseries,qed={$\diamond$}]{bfdefstyle}
\declaretheorem[sharenumber=Theorem,style=bfdefstyle]{Definition}
\declaretheorem[sharenumber=Theorem,style=bfdefstyle]{Definition-Proposition}
\declaretheorem[sharenumber=Theorem,style=bfdefstyle]{Example}

\newtheorem{definition-proposition}[Theorem]{Definition-Proposition} 

\declaretheoremstyle[headfont=\normalfont\bfseries,qed={$\diamond$}]{bfremstyle}
\declaretheorem[sharenumber=Theorem,style=bfdefstyle]{Remark}

\newcommand{\dom}{\mathrm{dom}\,}
\newcommand{\cl}{\mathrm{cl}\,}
\newcommand{\rk}{\mathrm{rk}\,}
\newcommand{\im}{\mathrm{im}\,}
\newcommand{\id}{\mathrm{id}\,}
\newcommand{\Res}{\mathrm{Res}}

\newcommand{\norm}[1]{\left\Vert #1 \right\Vert}
\newcommand{\abs}[1]{\left\vert #1 \right\vert}

\newcommand{\set}[1]{\left\lbrace #1 \right\rbrace}

\newcommand{\sign}{\text{sign }}
\newcommand{\scprod}[2]{ \langle #1, #2 \rangle }
\newcommand{\setdef}[2]{\left\{\ #1\ \left|\ \vphantom{#1} #2 \right.\right\}}

\renewcommand{\exp}[1]{e^{#1}}

\renewcommand{\phi}{\varphi}
\renewcommand{\Re}{\mathrm{Re}\,}
\renewcommand{\Im}{\mathrm{Im}\,}

 \numberwithin{equation}{section}

\makenomenclature

\author{Jonas Kirchhoff}
\title{Linear differential-algebraic systems are generically controllable}
\date{\today}

\begin{document}
\setlength{\parindent}{0em}

\begin{titlepage}
	\begin{flushright}
	\includegraphics[scale=1]{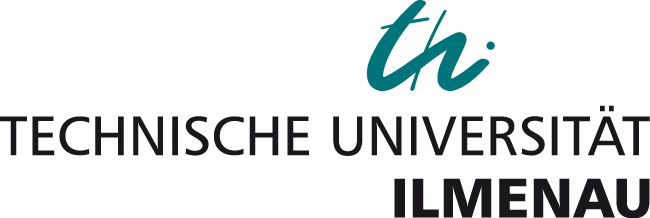}
	\end{flushright}
	\vspace{2cm}
	\centering
	{ Bachelor thesis\par}
	\vspace{1.5cm}
	{\huge\bfseries Linear differential-algebraic systems\\ are generically controllable\par}
	\vspace{1cm}
	{\Large revised edition\par}
	\vspace{2cm}
	{\Large Jonas Kirchhoff\par}
	\vspace{2cm}
	Advisor:\par
	{\Large Prof. Dr. Achim Ilchmann\par}
	\vspace{3.5cm}
	{\Large Fakultät für Mathematik und Naturwissenschaften\\ Fachgebiet für Analysis und Systemtheorie\par}

	\vfill

	{\large \today\par}
\end{titlepage}

\begin{titlepage}
\textcolor{white}{blindtext}
\end{titlepage}

\newpage

\begin{titlepage}
\begin{Large}
\textbf{Erklärung}
\end{Large}
\newline
\vspace*{1.5cm}
\newline
{Ich versichere hiermit, dass die vorliegende Bachelorarbeit selbständig verfasst wurde und keine weiteren als die angegebenen Hilfsmittel und Quellen verwendet wurden. Alle Stellen der Arbeit, die anderen Werken dem Sinn oder dem Wortlaut nach entnommen sind, wurden durch Angabe der Quellen sichtbar gemacht.}
\newline
\vspace*{1.5cm}{}
\newline
\begin{tabular}{c}\hline
\hspace{0.5cm}Ort, Datum\hspace{0.5cm}
\end{tabular}
\hspace{\fill}
\begin{tabular}{c}\hline
\hspace{0.5cm}Jonas Kirchhoff\hspace{0.5cm}
\end{tabular}
\end{titlepage}

\begin{titlepage}
\textcolor{white}{blindtext}
\end{titlepage}

\begin{titlepage}
\begin{Large}
\textbf{Acknowledgements}
\end{Large}
\newline
\vspace*{1.5cm}
\newline
{First of all I would like to thank my advisor, Prof. Achim Ilchmann, for his patience in the countless iterations of this thesis. I would also like to thank Prof. Thomas Hotz, whose valuable advice has been of enormous benefit to the work.}
\newline
\vspace*{1.5cm}{}
\end{titlepage}

\begin{titlepage}
\textcolor{white}{blindtext}
\end{titlepage}

\newpage

\tableofcontents

\newpage

\textbf{\Large{Nomenclature}}
\newline\newline

\begin{longtable}{lp{13cm}}
{$0_n$} & {$:= (\underbrace{0,\ldots,0}_{n~\text{zeros}})\in\mathbb{R}^n.$ If the dimension is clear from context, we will write~$0.$ In any case~$0 = 0_1$}\\[0.2\normalbaselineskip]
{$[[a,b]]$} & {$:= \setdef{a+t(b-a)}{t\in [0,1]}$, the \textit{straight line} between~$a,b\in X$ for a real vector space~$X$}\\[0.2\normalbaselineskip]
{$\norm{x}_2$} & {$:= \sqrt{x_1^2+\cdots +x_n^2}$, the \textit{Euclidean norm} of~$x = (x_1,\ldots,x_n)\in\mathbb{R}^n$}\\[0.2\normalbaselineskip]
{$\norm{x}_\infty$} & {$:= \max\set{\abs{x_1},\ldots,\abs{x_n}}$, the maximum norm of~$x = (x_1,\ldots,x_n)\in\mathbb{R}^n$}\\[0.2\normalbaselineskip]
{$\abs{A}$} & {the cardinality of a set~$A$}\\[0.2\normalbaselineskip]
{$\mathcal{AC}(I,\mathbb{R}^n)$} & {the set of absolutely continuous functions on the interval~$I\subseteq\mathbb{R}$ and values in~$\mathbb{R}^n$}\\[0.2\normalbaselineskip]
{$\lambda^d$} & {the~$d$-dimensional \textit{Lebesgue measure}}\\[0.2\normalbaselineskip]
{$\chi_A$} & {the \textit{characteristic function} of a subset~$A$ of a set~$X$; if~$x\in A$, then~$\chi_A(x) = 1$, and~$\chi_A(x) = 0$ else}\\[0.2\normalbaselineskip]
{$\mathbb{B}_d(x_0,\varepsilon)$} & {$:= \setdef{x\in X}{ d(x,x_0)<\varepsilon}$, the \textit{open ball} with center~$x_0\in X$ and radius~$\varepsilon\in~[0,\infty]$ for a metric space~$(X,d).$ If the metric is clear, we write~$\mathbb{B}(x_0,\varepsilon)$}\\[0.2\normalbaselineskip]
{$\mathbb{B}_\infty(x_0,\varepsilon)$} & {$:= \set{x\in \mathbb{R}^n: \norm{x-x_0}_\infty<\varepsilon}$, the open ball with respect to the maximum norm}\\[0.2\normalbaselineskip]
{$\overline{\mathbb{B}_d}(x_0,\varepsilon)$} & {$:= \set{x\in X: d(x,x_0)\leq\varepsilon}$, the \textit{closed ball} with center~$x_0\in X$ and radius~$\varepsilon\in~[0,\infty]$ for a metric space~$(X,d)$}\\[0.2\normalbaselineskip]
{$\mathbb{N}$} & {$:= \set{0,1,2,\ldots}$}\\[0.2\normalbaselineskip]
{$\mathbb{N}^*$} & {$:= \set{1,2,\ldots} = \mathbb{N}\setminus\set{0}$}\\[0.2\normalbaselineskip]
{$\underline j$} & {$:= \set{1,\ldots,j},~j\in\mathbb{N}^*.$ We set~$\underline 0 = \emptyset$}\\[0.2\normalbaselineskip]
{$\mathbb{R}$} & {the field of the real numbers}\\[0.2\normalbaselineskip]
{$\mathbb{C}$} & {the field of the complex numbers}\\[0.2\normalbaselineskip]
{$\mathbb{F}$} & {either~$\mathbb{R}$ or~$\mathbb{C}$}\\[0.2\normalbaselineskip]
{$\text{Im}\,z$} & {the imaginary part of the complex number~$z\in\mathbb{C}$}\\[0.2\normalbaselineskip]
{$\Re z$} & {the real part of the complex number~$z\in\mathbb{C}$}\\[0.2\normalbaselineskip]
{$\overset{\circ}{\mathbb{C}}_-$} & {$:=\set{z\in\mathbb{C}\,\big\vert \Re z < 0}$, the open left halfplane}\\[0.2\normalbaselineskip]
{$\overline{\mathbb{C}}_+$} & {$:=\set{z\in\mathbb{C}\,\big\vert\,\Re z\geq 0}$, the closed right halfplane}\\[0.2\normalbaselineskip]
{$\cl(A)$} & {$:= \bigcap\set{O^c: O\in\mathcal{O}, O\subseteq X\setminus A}$, the \textit{closure} of a subset~$A$ of a topological space~$(X,\mathcal{O})$}\\[0.2\normalbaselineskip]
{$\exists !$} & {this is an abbreviation for ``there is some unique''}\\[0.2\normalbaselineskip]
{$e_i$} & {the~$i$-th standard unit vector in~$\mathbb{R}^n$}\\[0.2\normalbaselineskip]
{$f^{-1}(A)$} & {$:= \set{x\in X: f(x)\in A}$, the \textit{preimage} of the set~$A\subseteq Y$ under the function~$f: X\to Y$}\\[0.2\normalbaselineskip]
{$\mathbb{F}[x_1,\ldots,x_n]$} & {$:= \set{\sum_{k = 0}^\ell a_k x_1^{\nu_{k,1}}\cdots x_n^{\nu_{k,n}}\,\Big\vert\, \ell\in\mathbb{N},a_k\in\mathbb{F},\nu_{k,j}\in\mathbb{N}}$, the ring of \textit{(real or complex) polynomials} in~$n$ indeterminants}\\[0.2\normalbaselineskip]
{$\mathbb{F}(x_1,\ldots,x_n)$} & {$:=\set{\frac{p}{q}\,\Big\vert\,p = p(x_1,\ldots,x_n),q = q(x_1,\ldots,x_n)\in\mathbb{F}[x_1,\ldots,x_n],q\neq 0}$}\\[0.2\normalbaselineskip]
{$R^{k\times\ell}$} & {the vector space of all~$k\times\ell$ matrices with entries in a ring~$R.$}\\[0.2\normalbaselineskip]
{$\rk_F$} & {the rank of a matrix with entries in the field $F$.}\\[0.2\normalbaselineskip]

{$M_{i,\cdot\,(}M_{\cdot,j})$} & {the~$i$th row (column) for a matrix~$M\in\mathbb{F}^{k\times\ell}$ and~$i\in\underline{k}$ ($j\in\underline{\ell}$)}\\[0.2\normalbaselineskip]
{$M_{i,j}$} & {the entry~$(i,j)$ of a matrix~$M\in\mathbb{F}^{k\times \ell}$ and~$i\in\underline{k},\,j\in\underline{\ell}$}\\[0.2\normalbaselineskip]
{$\mathcal{GL}(\mathbb{F}^n)$} & {the group of all invertible matrices~$M\in\mathbb{F}^{n\times n}$}\\[0.2\normalbaselineskip]
{$I_n$} & {the \textit{identity matrix} in~$\mathbb{R}^{n\times n}$}\\[0.2\normalbaselineskip]
{$m_{\sigma,\pi}$} & {the submatrix induced by the mappings~$\sigma$ and~$\pi$, see p.\,15}\\[0.2\normalbaselineskip]
{$M_{\sigma,\pi}$} & {the minor induced by the mappings~$\sigma$ and~$\pi$, see p.\,15}\\[0.2\normalbaselineskip]
{$\mathcal{L}^1_{\text{loc}}(I,\mathbb{R}^m)$} & {the set of locally integrable functions over the interval~$I\subseteq\mathbb{R}$}\\[0.2\normalbaselineskip]
{$\mathfrak{P}(M)$} & {$:=\set{A\,\big\vert\,\subseteq M}$, the power set of a set~$M$}\\[0.2\normalbaselineskip]
{$\mathfrak{P}_n(M)$} & {$:=\set{A\in\mathfrak{P}(M)\,\big\vert\abs{A} = n}$}\\[0.2\normalbaselineskip]
{$S_d$} & {the set of all permutations of a~$d$-element family}\\[0.2\normalbaselineskip]
{$\Sigma_{\ell,n,m}$} & {$:= \mathbb{R}^{\ell\times n}\times\mathbb{R}^{\ell\times n}\times\mathbb{R}^{\ell\times m}$}\\[0.2\normalbaselineskip]
{$\Sigma_{n,m}$} & {$:= \mathbb{R}^{n\times n}\times\mathbb{R}^{n\times m}$}\\[0.2\normalbaselineskip]
{$\text{span}(S)$} & {$:= \bigcap\set{U\subseteq V\,\big\vert\,U~\text{linear~subspace~with}~S\subseteq U}$, the linear span of the subset~$S$ of the vector space~$V$; if~$S = \set{s}$, then we write~$\text{span}(s)$ instead}\\[0.2\normalbaselineskip]
{$U^\bot$} & {$:=\set{v\in\mathbb{R}^n\,\big\vert\,\forall\, u\in U: \scprod{u}{v} = 0}$, the orthogonal complement of a subspace~$U$ of a Hilbert space~$(H,\scprod{\cdot}{\cdot})$}\\[0.2\normalbaselineskip]
{$\mathcal{V}_n(\mathbb{F})$} & {$:= \set{\mathbb{V}\subseteq\mathbb{F}^n\,\Big\vert\,\exists\, q_1(\cdot),\ldots,q_k(\cdot)\in\mathbb{F}[x_1,\ldots, x_n]: \bigcap_{i=1}^k q^{-1}(\set{0}) = \mathbb{V}}$, the set of algebraic varieties in~$\mathbb{F}^n$, see p.\,4}\\[0.2\normalbaselineskip]
{$\mathcal{V}_n^{\text{prop}}(\mathbb{F})$} & {$:= \mathcal{V}_n(\mathbb{F})\setminus\set{\mathbb{F}^n}$ see p.\,4}\\[0.2\normalbaselineskip]
{$\mathcal{W}^{1,1}_{\text{loc}}(I,\mathbb{R}^n)$} & {$:=\set{\varphi\in\mathcal{L}^1_\text{loc}(I,\mathbb{R}^n)\,\big\vert\,\varphi~\text{weakly differentiable}},$ for some interval~$I\subseteq\mathbb{R}$}\\[0.2\normalbaselineskip]
{$\text{ess\,sup} f$} & {$:=\inf\set{\sup f(\dom f\setminus N)\,\big\vert\,N~\text{Lebesgue~nullset}},$ the essential supremum of a function~$f$}\\[0.2\normalbaselineskip]
\end{longtable}

\newpage\textcolor{white}{blindtext}\newpage
\section{Introduction}

In this thesis we investigate linear differential-algebraic equations (DAEs) with constant real coefficients of the form
\begin{align}\label{eq:intro_1}
\tfrac{\mathrm{d}}{\mathrm{d}t}(Ex) = Ax + Bu
\end{align}
with constant real coefficients $(E,A,B)\in\Sigma_{\ell,n,m} = \mathbb{R}^{\ell\times n}\times\mathbb{R}^{\ell\times n}\times\mathbb{R}^{\ell\times m}$. A locally integrable control $u\in L^1_{\text{loc}}(\mathbb{R}_{\geq 0},\mathbb{R}^m)$ and trajectory $x\in L^1_{\text{loc}}(\mathbb{R}_{\geq 0},\mathbb{R}^n)$ so that $Ex\in AC(\mathbb{R}_{\geq 0},\mathbb{R}^{\ell})$ is almost everywhere differentiable is called a solution of~\eqref{eq:intro_1}, if~\eqref{eq:intro_1} is almost everywhere fullfilled. By studying the set of all solutions $(x,u)$ of~\eqref{eq:intro_1} it is possible to formulate some controllability concepts for DAEs of the form~\eqref{eq:intro_1}. We are interested in topological properties of the set
\begin{align*}
S := \set{(E,A,B)\in\Sigma_{\ell,n,m}\,\big\vert\,(E,A,B)~\text{``controllable''}},
\end{align*}
where ``controllable'' stands for these different controllability concepts. \cite{DAEs} have shown that~$S$ can be described by certain algebraic conditions. Wonham proved in~\cite[Theorem 1.3, p.\,44]{Wonham} that the set of controllable linear ordinary differential equations (ODEs) of the form
\begin{align}\label{eq:intro_ODE}
\tfrac{\mathrm{d}}{\mathrm{d}t} x = Ax + Bu,\quad (A,B)\in\mathbb{R}^{n\times n}\times\mathbb{R}^{n\times m},~u\in L^1_{\text{loc}}(\mathbb{R}_{\geq 0},\mathbb{R}^m)
\end{align}
is generic using an algebraic condition, namely the well-known Kalman criterion. We take this result as a motivation and show that~$S$ is generic under certain conditions on~$\ell,n,m.$

At first we consider genericity as defined in~\cite[p.\,28]{Wonham} in the second section. Unlike Wonham we consider the space~$\mathbb{F}^n$ with~$\mathbb{F} = \mathbb{R}$ or~$\mathbb{F} = \mathbb{C}$. We collect the properties of generic sets, i.e. they are closed under~$\cap$ and~$\cup$ and they are dense. Further we briefly discuss the advantage of the usage of the Zariski-topology in the context of generic sets.

In the third section we study polynomial block matrices of the form
\begin{align*}
P(x) = \begin{bmatrix}P^{1,1}(x) & \cdots & P^{1,q}(x)\\\vdots & \ddots & \vdots\\ P^{p,1}(x) & \cdots & P^{p,q}(x)\end{bmatrix},\quad\begin{array}{l}
P^{i,j}(x)\in\mathbb{F}[x]^{n_{i,j}\times m_{i,j}},\\
\mathbb{F} = \mathbb{R}~\text{or}~\mathbb{F} = \mathbb{C},\\
n_{i,j} = n_{k,j},\\
m_{i,j} = m_{i,k}
\end{array}
\end{align*}
whose degree is bounded from above by some constant. The vector space of these polynomial matrices can be identified with~$\mathbb{F}^n$ for some~$n\in\mathbb{N}$ and hence we can apply the concept of generic sets. We look at the rank of these matrices, consider the three conditions
\begin{align}
\rk_{\mathbb{F}[x]} P(x)  & \geq d \label{eq:no_1}\\
\forall \lambda\in\mathbb{C}: \rk_{\mathbb{C}} P(\lambda) & \geq d\label{eq:no_2}\\
\forall \lambda\in\mathbb{C}~\text{with}~\Re\lambda\geq 0: \rk_{\mathbb{C}} P(\lambda) & \geq d\label{eq:no_3}
\end{align}
and prove that there are conditions on~$d$ such that the set of polynomial matrices which fullfill~\eqref{eq:no_1}, or~\eqref{eq:no_2} or~\eqref{eq:no_3} resp., is generic.

In the fourth section we write down the proof that the set of controllable ODEs of the form~\eqref{eq:intro_ODE} is generic as an example for the application of the propositions from the third section to differential equations.

In the fifth and sixth section we turn our attention to DAEs and consider the concepts of freely initializable, impulse controllable, completely controllable (stabilizable), strongly controllable (stabilizable) and in the behavioural sense controllable (stabilizable) systems. Using the results of the third section we derive necessary and sufficient conditions on~$\ell,n,m$ so that~$S$ is generic.

To the best of my knowledge, the paper~\cite{Belur} by Belur und Shankar is the only one which treats genericity of controllability of DAEs, namely the case of impulse controllable systems. We will later discuss the differences between their result and the result of this thesis.

\newpage
\section{Genericity: definitions and elementary properties}

For the following observations let~$n\in\mathbb{N}^*$,~$\mathbb{F} = \mathbb{R}$ or~$\mathbb{C}$ and provide~$\mathbb{F}^n$ with the Euclidean norm~$\norm{\cdot}_2.$ Further, we identify any polynomial
\begin{align}
    p(x) = p(x_1,\ldots,x_n) = \sum_{k = 0}^\ell a_k x_1^{\nu_{k,1}}\cdots x_n^{\nu_{k,n}}\in \mathbb{F}[x_1,\ldots,x_n]
\end{align}
with its polynomial function
\begin{align}
    p(\cdot): \mathbb{F}^n\to\mathbb{F},\quad x = (x_1,\ldots,x_n) \mapsto p(x) =  \sum_{k = 0}^\ell a_k x_1^{\nu_{k,1}}\cdots x_n^{\nu_{k,n}}.
\end{align}

\begin{Definition}[{Algebraic variety, genericity, see~\cite[p.\,28]{Wonham} and~\cite[p.\,50]{UndAlgGeom}}]~\label{def:gen_set}
A set~$ \mathbb{V} \subseteq \mathbb{F}^n$ is called an \emph{algebraic variety}\footnote{The notion of \textit{algebraic varieties} differs from source to source. What~\cite{Wonham} calls algebraic variety (in fact he only calls them variety, but as suggested in~\cite[p.\,240]{GeomMeasTheory} there is an analytic analogon) is called \textit{algebraic set} or \textit{affine variety} in algebraic geometry books (e.g.~\cite{AlgVar,UndAlgGeom}). Since we are interested in applications in mathematical systems theory, we will stick to Wonhams nomenclature.}, if there exist finitely many polynomials
\begin{align*}
    p_1(x_1,\ldots, x_n),\ldots,p_k(x_1,\ldots, x_n) \in \mathbb{F}[x_1,\ldots, x_n]
\end{align*}{}
such that~$\mathbb{V}$ is the locus of their zeros, i.e.
\begin{align}
\mathbb{V} = \set{x\in \mathbb{F}^n\,\big\vert\, \forall\, i\in\underline{k}: p_i(x)=0} = \bigcap_{i=1}^k p_i^{-1}(\set{0}).
\end{align}
An algebraic variety~$ \mathbb{V}$ is called \emph{proper} if ~$ \mathbb{V} \subsetneq \mathbb{F}^n$, and \emph{nontrivial} if~$\mathbb{V} \neq \emptyset.$ The set of all algebraic varieties in~$\mathbb{F}^n$ is denoted as
\begin{align}
    \mathcal{V}_n(\mathbb{F}) := \set{\mathbb{V}\subseteq\mathbb{F}^n\,\Big\vert\,\exists\, q_1(\cdot),\ldots,q_k(\cdot)\in\mathbb{F}[x_1,\ldots, x_n]: \bigcap_{i=1}^k q^{-1}(\set{0}) = \mathbb{V}}
\end{align}
and the set of all proper algebraic varieties as
\begin{align}\label{eq:def_prop_var}
    \mathcal{V}_n^{\text{prop}}(\mathbb{F}) := \mathcal{V}_n(\mathbb{F})\setminus\set{\mathbb{F}^n}.
\end{align}

A set~$S\subseteq\mathbb{F}^n$ is called \textit{generic}, if there exist a proper algebaric variety $\mathbb{V}\in\mathcal{V}_n^{\text{prop}}(\mathbb{F})$ so that $S^c\subseteq\mathbb{V}$. If the algebraic variety $\mathbb{V}$ is known, then we call $S$ \textit{generic with respect to (w.r.t.)~$\mathbb{V}$.}
\end{Definition}

Wonham does not talk about generic sets in~\cite{Wonham} but about \textit{properties}, which are functions of the form~$\Pi : \mathbb{F}^n \to \set{0,1}.$ Such a property~$\Pi$ is called generic if 
\begin{align*}
\exists\,\mathbb{V}\in\mathcal{V}_n^{\text{prop}}(\mathbb{F})~\forall x\in\mathbb{V}^c: \Pi(x) = 1.
\end{align*}
It is evident that  there is this 1-1-connection between generic sets and generic properties, the latter being characteristic functions of the former.

We will use the term of generic sets and not of generic properties throughout this thesis. 

In a next step we briefly discuss the correspondence between algebraic varieties and ideals of $\mathbb{F}[x_1,\ldots,x_n]$. First we recall the definition of an ideal in a commutative ring.

\begin{Definition}[{Ideal, see~\cite[p.\,86]{Alg}}]
Let~$R$ be a commutative ring with 1. A set~$I\subseteq R$ is called an \textit{ideal} of~$R$, if
\begin{enumerate}[(i)]
\item~$\forall x,y\in I: x+y\in I$ and
\item~$\forall\,x\in I~\forall\,r\in R: rx \in I$.
\end{enumerate}
For any set~$S\subseteq\mathbb{R}$, the set
\begin{align*}
(S)_R := \bigcap\set{I\subseteq R\,\big\vert\, S\subseteq I, I~\text{an~ideal}}
\end{align*}
is called the \textit{ideal generated by~$S$}; $S$ is the \textit{generator} of $(S)_R$.
\end{Definition}

\begin{Remark}[{Correspondence between algebraic varieties and ideals, see~\cite[p.\,50]{UndAlgGeom}}]\label{Rem:Corr_alg_var_ideal}
Let
\begin{align*}
I := (\set{p_1,\ldots,p_n})_{\mathbb{F}[x_1,\ldots,x_n]}
\end{align*}
be the ideal generated by~$p_1,\ldots,p_k\in\mathbb{F}[x_1,\ldots,x_n].$ Then by~\cite[p.\,86]{Alg} and commutativity of~$\mathbb{F}[x_1,\ldots,x_n]$ we have
\begin{align*}
I = \set{\left.\sum_{j = 1}^k\alpha_jp_j\,\right\vert\,\alpha_1,\ldots,\alpha_k\in\mathbb{F}[x_1,\ldots,x_n]}
\end{align*}
and hence
\begin{align*}
\mathbb{V}(I) & := \set{x\in\mathbb{F}^n\,\big\vert\,\forall p\in I: p(x) = 0}\\
& = \set{x\in\mathbb{F}^n\,\big\vert\,\forall i\in\underline{k}: p_i(x) = 0}\\
& = \bigcap_{i = 1}^kp_i^{-1}(\set{0})\in\mathcal{V}_n(\mathbb{F}).
\end{align*}
This yields the inclusion 
\begin{align}\label{eq:corr_alg_var_ideal}
\mathcal{V}_n(\mathbb{F})\subseteq\set{\mathbb{V}(I)\,\big\vert\, I\subseteq\mathbb{F}[x_1,\ldots,x_n]~\text{an~ideal}}.
\end{align}

We recall that a commutative ring with 1 is called \textit{Noetherian}, if any ideal has a finite generator (see~\cite[Proposition-Definition 3.1(i),~p.\,48]{UndAlgGeom}). Since any field is a Noetherian ring (indeed, the only ideals are the whole field, $\set{0}$ and the empty set, which are generated by 1, 0 and $\emptyset$, resp.) and the isomorphy
\begin{align*}
\mathbb{F}[x_1,\ldots,x_m,x_{m+1}] \cong \big(\mathbb{F}[x_1,\ldots,x_m]\big)[x_{m+1}]
\end{align*}
holds true, the Hilbert Basis Theorem (see~\cite[Theorem 3.3,~p.\,49]{UndAlgGeom}) yields that~$\mathbb{F}[x_1,\ldots,x_n]$ is a Noetherian ring. This yields that for any ideal~$I\subseteq\mathbb{F}[x_1,\ldots,x_n]$ there are~$p_1,\ldots,p_k\in\mathbb{F}[x_1,\ldots,x_n]$ such that~$I = (\set{p_1,\ldots,p_k})_{\mathbb{F}[x_1,\ldots,x_n]}$. Then we find
\begin{align*}
\mathbb{V}(I) = \bigcap_{i = 1}^k p_i^{-1}(\set{0})\in\mathcal{V}_n(\mathbb{F})
\end{align*}
and hence we have equality in~\eqref{eq:corr_alg_var_ideal}. With the Hilbert basis theorem it is evident that
\begin{align*}
\mathcal{V}^{\text{prop}}_n(\mathbb{F}) = \set{\mathbb{V}(I)\,\big\vert\, I\subseteq\mathbb{F}[x_1,\ldots,x_n]~\text{an~ideal},~I\neq\set{0}}.
\end{align*}
\end{Remark}

Although~$\mathbb{F}[x_1,\ldots,x_n]$ is not a principal ring for~$n\geq 2$ (see~\cite[p.\,113]{Alg}), it is well-known that any algebraic variety~$\mathbb{V}\subseteq\mathbb{R}^n$ is generated by a principal ideal, i.e. an ideal that is generated by a singleton set.

\begin{Lemma}\label{lem:MostBoringResult}
Any set~$\mathbb{V}\subseteq\mathbb{R}^n$ is an algebraic variety if, and only if, there is a polynomial~$p\in\mathbb{R}[x_1,\ldots,x_n]$ so that~$\mathbb{V} = p^{-1}(\set{0}).$
\end{Lemma}
\begin{proof}
\begin{enumerate}
    \item[$\implies$] Let~$\mathbb{V} = \bigcap_{i = 1}^kp_i^{-1}(\set{0}).$ Then
    \begin{align*}
        \mathbb{V} = \Bigg(\underbrace{\sum_{i = 1}^kp_i^2}_{=:p(\cdot)}\Bigg)^{-1}(\set{0}).
    \end{align*}
    \item[$\impliedby$] This is trivial.
\end{enumerate}
\end{proof}

It is notable that this proof does not work for $\mathbb{F} = \mathbb{C}$. As an example consider the constant polynomials
\begin{align*}
p_1 & : \mathbb{C}^n\to\mathbb{C}, z\mapsto 1,\\
p_2 & : \mathbb{C}^n\to\mathbb{C}, z\mapsto i.
\end{align*}
Then $p_1^{-1}(\set{0}) = p_2^{-1}(\set{0}) = \emptyset$ and 
\begin{align*}
(p_1^{2}+p_2^2)^{-1}(\set{0}) = \mathbb{C}^n.
\end{align*}
Of course, in this case $p_1^{-1}(\set{0})\cap p_2^{-1}(\set{0})$ is generated by a principal ideal and hence this is not a counterexample against Lemma~\ref{lem:MostBoringResult} for $\mathbb{C}$ instead of $\mathbb{R}$. However, there is the following well-known counterexample.

\begin{Lemma}
For $n\geq 2$ there exists a proper algebraic variety $\mathbb{V}\in\mathcal{V}_n^{\text{prop}}(\mathbb{C})$ so that
\begin{align*}
\forall p\in\mathbb{C}[x_1,\ldots,x_n]: \mathbb{V}\neq p^{-1}(\set{0}).
\end{align*}
\end{Lemma}
\begin{proof}
Consider first $n = 2$ and the set
\begin{align*}
\mathbb{V} := \set{(0,0)}\subseteq\mathbb{C}^2.
\end{align*}
With
\begin{align*}
p:\mathbb{C}^2\to\mathbb{C}, (x,y)\mapsto x
\end{align*}
and
\begin{align*}
q:\mathbb{C}^2\to\mathbb{C}, (x,y)\mapsto y
\end{align*}
we see that $\mathbb{V} = q^{-1}(\set{0})\cap p^{-1}(\set{0})\in\mathcal{V}_2^{\text{prop}}(\mathbb{C}).$ Assume, there exists a polynomial 
\begin{align*}
\widehat{p}(x,y) = \sum_{i,j=0}^k a_{i,j}x^iy^j = \sum_{i = 0}^k\underbrace{\left(\sum_{j=0}^{k_i}a_{i,j}x^i\right)}_{=:p_i(x)}y^j\in\mathbb{C}[x,y]
\end{align*}
with $\mathbb{V} = \widehat{p}^{-1}(\set{0}).$ If $k = 0$, then we would find either $\widehat{p}^{-1}(\set{0}) = \emptyset$ or $\abs{\widehat{p}^{-1}(\set{0})} = \infty$, which contradicts our assumption. Hence we can w.l.o.g. assume $k\geq 1$ and $p_k\neq 0_{\mathbb{C}[x]}$. Then we find
\begin{align*}
\forall x\in\mathbb{C}\setminus p_k^{-1}(\set{0}):\quad \widehat{p}(x,\cdot)\in\mathbb{C}[y]\quad\wedge\quad\deg\widehat{p}(x,\cdot) = k\geq 1.
\end{align*}
Since $\mathbb{C}$ is algebraically closed, we find further
\begin{align*}
\forall x\in\mathbb{C}\setminus p_k^{-1}(\set{0})\exists y\in\mathbb{C}: \widehat{p}(x,y) = 0.
\end{align*}
By our assumption we have $\widehat{p}^{-1}(\set{0}) = \set{(0,0)}$ and thus we find
\begin{align*}
\mathbb{C}\setminus p_k^{-1}(\set{0}) = \set{0}
\end{align*}
which contradicts the fact, that any complex polynomial has at most finitely many zeros. Hence there is no polynomial $\widehat{p}\in\mathbb{C}[x,y]$ so that $\mathbb{V} = \widehat{p}^{-1}(\set{0}).$

Let $n>2$. Then $\mathbb{V}\times\mathbb{C}^{n-2}$ is an algebraic variety. Assume that there exists a polynomial $\widehat{q}\in\mathbb{C}[x_1,\ldots,x_n]$ so that $\mathbb{V}\times\mathbb{C}^{n-2} = \widehat{q}^{-1}(\set{0})$. Then
\begin{align*}
\forall (y_3,\ldots,y_n)\in\mathbb{C}^{n-2}:\quad \widehat{q}(\cdot,\cdot,y_3,\ldots,y_n)\in\mathbb{C}[x_1,x_2]\quad\wedge\quad\big(\widehat{q}(\cdot,\cdot,y_3,\ldots,y_n)\big)^{-1}(\set{0}) = \mathbb{V},
\end{align*}
in contrary to our result that $\mathbb{V}$ is not the set of zeros of a single polynomial.
\end{proof}

A very useful property of generic sets for later is the following.

\begin{Remark}\label{Rem:InclusionOfGenericSets}
If~$S\subseteq\mathbb{F}^n$ is generic w.r.t.~$\mathbb{V}\in\mathcal{V}_n^{\text{prop}}$, then any set~$\widetilde{S}\supseteq S$ is generic w.r.t.~$\mathbb{V}$. This follows from~$\widetilde{S}^c\subseteq S^c\subseteq\mathbb{V}.$
\end{Remark}

We have given the definition of generic sets and algebraic varieties. Now we shall give a small example and draw a picture.

\begin{Example}
The algebraic variety induced by the polynomial~$$p:\mathbb{R}^2\to\mathbb{R},\quad (x,y)\mapsto x^3-xy^2+5$$ is the blue and red line depicted in Figure~\ref{fig:fig_1}.
\begin{figure}[h]
    \centering
    \includegraphics[scale = 0.45]{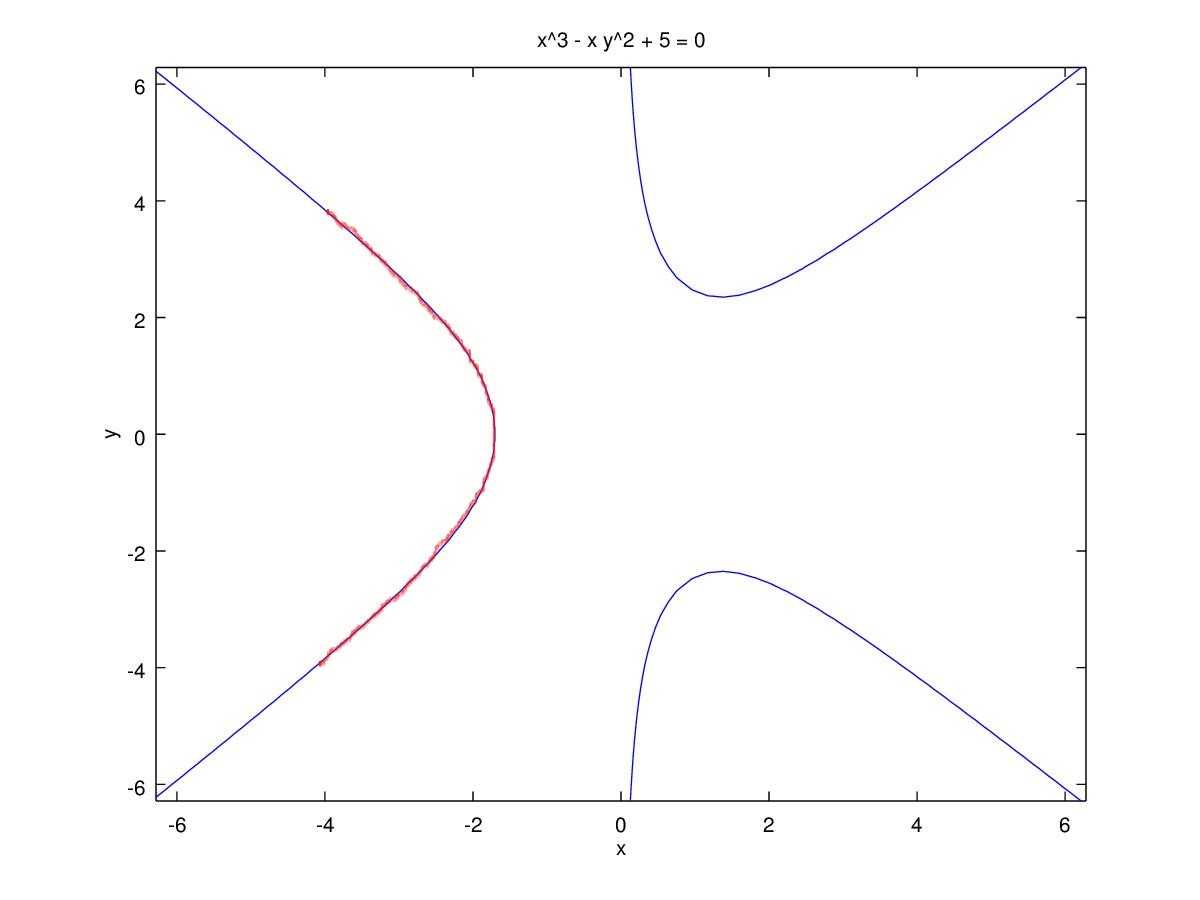}
    \caption{Algebraic variety induced by~$(x,y)\mapsto p(x,y) = x^3-xy^2+5$; the red line being the complement of the generic set $S$.}
    \label{fig:fig_1}
\end{figure}\newline
Thus the set
\begin{align*}
S := \set{(x,y)\in\mathbb{R}^2\,\big\vert\,x\leq-4\vee x\geq 0\vee p(x,y)\neq 0}
\end{align*}
is generic w.r.t. to~$p^{-1}(\set{0}).$ The complement of $S$ is the red line depicted in Figure~\ref{fig:fig_1}.
\end{Example}

Reid proves in~\cite[Proposition-Definition 3.5, p.\,50]{UndAlgGeom} some properties of~$\mathcal{V}_n(\mathbb{F})$. We recall the properties which we will use later.

\begin{Lemma}[{Properties of algebraic varieties, see~\cite[Proposition-Definition 3.5, p.\,50]{UndAlgGeom}}]~\label{Rem:Zariski}
\begin{enumerate}[(i)]
    \item Any algebraic variety~$\mathbb{V} = \bigcap_{i=0}^k p_i^{-1}(\set{0})\subseteq\mathbb{F}^n$ is closed.
    \item Let~$(\mathbb{V}_i)_{i\in I}\in\big(\mathcal{V}_n(\mathbb{F})\big)^I$ a family of algebraic varieties for an arbitrary index set $I$. Then~$\bigcap_{i\in I} \mathbb{V}_i\in\mathcal{V}_n(\mathbb{F}).$ If~$\mathbb{V}_i$ is proper for some~$i\in I,$ then~$\bigcap_{i\in I} \mathbb{V}_i$ is proper. This implies that $\mathcal{V}_n^{\text{prop}}(\mathbb{F})$ and $\mathcal{V}_n(\mathbb{F})$ are closed under~$\cap$.
    \item~$\mathcal{V}_n(\mathbb{F})$ is closed under~$\cup$.
\end{enumerate}
\end{Lemma}
\begin{proof}
\begin{enumerate}[(i)]
    \item Since any multivariate polynomial is continuous and each singleton set is closed, the statement follows from continuity .
    \item Let~$\mathbb{V}_i := \bigcap_{\ell = 1}^{k_i} \big(p_\ell^i\big)^{-1}(\set{0})$ for any~$i\in I$ and define the ideal
    \begin{align*}
    J := \left(\set{p_\ell^i\,\big\vert\,i\in I, \ell\in\underline{k_i}}\right)_{\mathbb{F}[x_1,\ldots,x_n]}.
    \end{align*}
    By Remark~\ref{Rem:Corr_alg_var_ideal}, we find~$\bigcap_{i \in I}\mathbb{V}_i = \mathbb{V}(J)$. Since~$\mathbb{F}[x_1,\ldots,x_n]$ is a Noetherian ring, we find~$q_1,\ldots,q_k\in\mathbb{F}[x_1,\ldots,x_n]$ so that~$J = \left(\set{q_1,\ldots,q_k}\right)_{\mathbb{F}[x_1,\ldots,x_n]}$ and hence
    \begin{align*}
    \mathbb{V}(J) = \bigcap_{i = 1}^kq_i^{-1}(\set{0})\in\mathcal{V}_n(\mathbb{F}).
    \end{align*}
    If~$\mathbb{V}_i$ is proper, then
    \begin{align*}
    \bigcap_{j\in I}\mathbb{V}_j\subseteq\mathbb{V}_i\subsetneq\mathbb{F}
    \end{align*}
    and hence~$\bigcap_{i\in I} \mathbb{V}_i$ is proper.
    \item Let~$\mathbb{V}^1,\ldots,\mathbb{V}^\ell\in\mathcal{V}_n(\mathbb{F})$ be finitely many algebraic varieties with
    \begin{align*}
        \forall\, i\in\underline{\ell}: \mathbb{V}^i = \bigcap_{j=1}^{k_i}(p_j^i)^{-1}(\set{0}).
    \end{align*}
    It is evident that
    \begin{align*}
        \bigcup_{i=1}^\ell \mathbb{V}^i = \bigcap_{m_1 = 1}^{k_1}\cdots\bigcap_{m_\ell = 1}^{k_\ell}\left(\prod_{i=1}^\ell p^i_{m_i}\right)^{-1}(\set{0})\in\mathcal{V}_n(\mathbb{F}).
    \end{align*}
\end{enumerate}
\end{proof}

\begin{Remark}[{Zariski topology, see~\cite[p.\,50]{UndAlgGeom}}]\label{rem:Zariski_topology}
The \textit{Zariski topology} on $\mathbb{F}^n$ is defined as
\begin{align*}
\mathcal{ZT} := \set{\mathbb{F}^n\setminus\mathbb{V}\,\big\vert\,\mathbb{V}\in\mathcal{V}_n(\mathbb{F})}.
\end{align*}
Since, by Lemma~\ref{Rem:Zariski}\,(ii) and (iii)
\begin{itemize}
\item $\mathbb{F}^n = \mathbb{F}^n\setminus (1_{\mathbb{F}[x_1,\ldots,x_n]})^{-1}(\set{0})$, where $1_{\mathbb{F}[x_1,\ldots,x_n]}$ is the constant mapping $x\mapsto 1$, and $\emptyset = \mathbb{F}^n\setminus (0_{\mathbb{F}[x_1,\ldots,x_n]})^{-1}(\set{0})$,\\
\item if $(\mathbb{F}^n\setminus\mathbb{V}_i)_{i\in I}\in\mathcal{ZT}^I,$ then $\bigcup_{i\in I}(\mathbb{F}^n\setminus\mathbb{V}_i) = \mathbb{F}^n\setminus\bigcap_{i\in I}\mathbb{V}_i\in\mathcal{ZT},$
\item if $\mathbb{F}^n\setminus\mathbb{V}_1,\mathbb{F}^n\setminus\mathbb{V}_2\in\mathcal{ZT}$, then $\mathbb{F}^n\setminus\mathbb{V}_1\cap\mathbb{F}^n\setminus\mathbb{V}_2 = \mathbb{F}^n\setminus(\mathbb{V}_1\cup\mathbb{V}_2) \in\mathcal{ZT}$,
\end{itemize}
$\mathcal{ZT}$ is indeed a topology. For shortness, we call $O\in\mathcal{ZT}$ \textit{Zariski-open}. Then we find for any $O\subseteq\mathbb{F}^n$
\begin{align*}
O~\text{is~Zariski-open}~\begin{array}{c}
\Rightarrow\\
\overset{\text{i.g.}}{\nLeftarrow}
\end{array}~O~\text{is~open~in}~\mathbb{F}^n.
\end{align*}
The direction ``$\implies$'' follows from the definition of $\mathcal{ZT}$ and Lemma~\ref{Rem:Zariski}\,(i). To see that the converse is in general not true, consider the upper halfplane~$\set{z\in\mathbb{C}\,\big\vert\,\Im z > 0}$ which is open in the Euclidean topology, but by the identity theorem for analytic functions (see~\cite[Theorem V.3.13, p.\,406]{AE_I}) not open in the Zariski topology. Hence the Zariski topology is strictly coarser than the Euclidean topology. We will later see that the Zariski open sets are either empty or dense, which illustrates this point.

Although Reid writes ``The Zariski topology may cause trouble to some students; since it is only being used as a language, and has almost no content [\ldots]''\cite[p.\,51]{UndAlgGeom}, the Zariski topology is an advantageous choice for the investigation of genericity. Indeed, we can rewrite Definition~\ref{def:gen_set} as:~$S\subseteq\mathbb{F}^n$ is generic, if~$S$ contains a nonempty Zariski-open set.

However, we won't use the notion of Zariski-open sets any further in this thesis.
\end{Remark}

With the properties of algebraic varieties proven in Lemma~\ref{Rem:Zariski} we can prove the well-known fact that the locus of zeros of a finite or infinite polynomial sequence is an algebraic variety.

\begin{Lemma}\label{Cor:some_more_varieties}
For any polynomial sequence~$P(x) = (p_i(x))_{i\in K}\in\left(\mathbb{F}[x_1,\ldots,x_n]\right)^K$ with $K\in\set{\underline{k}\,\big\vert\,k\in\mathbb{N}^*}\cup\set{\mathbb{N}}$, the set~$$P^{-1}(\set{0_{\mathbb{F}^K}})\subseteq\mathbb{F}^n$$ is an algebraic variety, and it is proper if, and only if, at least one~$p_i^{-1}(\set{0})$ is proper.
\end{Lemma}
\begin{proof}
Since
\begin{align*}
	P^{-1}(\set{0_{\mathbb{F}^K}}) & = \set{x\in\mathbb{F}^n\,\left\vert\,\forall\,{i\in K}:p_i(x) = 0\right.} = \bigcap_{i\in K}p_i^{-1}(\set{0}),
\end{align*}
$P^{-1}(\set{0_{\mathbb{F}^K}})$ is an algebraic variety by Lemma~\ref{Rem:Zariski}\,(ii). It is evident that~$P^{-1}(\set{0_{\mathbb{F}^K}})$ is proper if, and only if,
\begin{align*}                    
	\exists\,i\in K: p_i^{-1}(\set{0})~\text{is~proper}.
\end{align*}
\end{proof}

A useful property of algebraic varieties is the well-known fact that we can move an algebraic variety by some constant vector and the resulting set is also an algebraic variety.

\begin{Lemma}\label{lem:affine_variety}
For any~$\mathbb{V}\in\mathcal{V}_n^{\text{prop}}(\mathbb{F})$ and~$v\in\mathbb{F}^n$, the set~$v+\mathbb{V} := \set{v+u\,\big\vert\, u\in\mathbb{V}}$ is a proper algebraic variety.
\end{Lemma}
\begin{proof}
    Let~$\mathbb{V}\in\mathcal{V}_n^{\text{prop}}(\mathbb{F})$ with~$\mathbb{V} = \bigcap_{i=1}^kp_i^{-1}(\set{0})$ for some~$p_1,\ldots,p_k\in\mathbb{F}[x_1,\ldots,x_n]$ and~$v\in\mathbb{F}^n$ be arbitrary. For any~$x\in\mathbb{V}^c$ we find~$v+x\in\big(v+\mathbb{V}\big)^c.$ Thus~$v+\mathbb{V}$ is proper and it suffices to show that~$v+\mathbb{V}$ is an algebraic variety. It is well-known that, for any~$p,q_1,\ldots,q_n\in\mathbb{F}[x_1,\ldots,x_n]$, the mapping
    \begin{align*}
        P:\mathbb{F}^n\to\mathbb{F},\quad x\mapsto p(q_1(x),\ldots,q_n(x))
    \end{align*}
    is also a polynomial in~$n$ indeterminants. Since, for any~$i\in\underline{n}$, the mapping
    \begin{align*}
        r_i:\mathbb{F}^n\to\mathbb{F},\quad x\mapsto x_i-v_i
    \end{align*}
    is a polynomial, we get
    \begin{align*}
        v+\mathbb{V} = \bigcap_{i = 1}^k\big(p_i(r_1(\cdot),\ldots,r_n(\cdot))\big)^{-1}(\set{0}),
    \end{align*}
    which is an algebraic variety.
\end{proof}

We give a well-known example of an algebraic variety, which we do not immediately identify as the graph of some polynomial, but as the image of a polynomial vector. 

\begin{Proposition}\label{ex:LinSubspAlgVar}
Any affine linear subspace~$U\subseteq\mathbb{F}^n$ with dimension~$\dim U\leq n-1$ is a proper algebraic variety.
\end{Proposition}
\begin{proof}
Let~$U = u+V$ for some linear subspace~$V\subseteq\mathbb{F}^n$ and~$u\in\mathbb{F}^n.$ In view of Lemma~\ref{lem:affine_variety} it is sufficient to prove that~$V$ is a proper algebraic variety. 

Since~$\dim V = \dim U\leq n-1$ and hence~$V\neq\mathbb{F}^n,$ it suffices to show that~$V$ is an algebraic variety. Note that, for any~$c\in\mathbb{F}^n$, the mapping
\begin{align*}
    p_c: \mathbb{F}^n\to\mathbb{F}, x\mapsto c^\top x,
\end{align*}
is a polynomial,~$p_c\in\mathbb{F}[x_1,\ldots,x_n].$ Let~$m:=n-\dim V$ and~$W\in\mathbb{F}^{n\times m}$ so that~$\im W = V^\bot.$ As indicated in the nomenclature,~$W_{\cdot,j}\in\mathbb{F}^n$ denotes, for any~$j\in\underline{m}$, the~$j$-th column of the matrix~$W.$ Then 
\begin{align*}
    v\in V\iff W^Tv = 0\iff v\in \bigcap_{j=1}^{m} p_{W_{\cdot,j}}^{-1}(\set{0})\in\mathcal{V}_n(\mathbb{F}).
\end{align*}
\end{proof}

\begin{Remark}
At this point we need the fact that the dualspace of $\mathbb{F}^n$ can be embedded into $\mathbb{F}[x_1,\ldots,x_n]$. This holds true for arbitrary finite dimensional vectorspaces $V$ over a field $F$. In this case the dualspace of $V$, $V'$, has the same dimension as $V$ (see~\cite[Satz II.2.6, p.\,69]{FunkAny}) and hence there exists a basis $(\beta_1,\ldots,\beta_n)$ of $V'$, $n$ being the dimension of $V$. Let $(b_1,\ldots,b_n)$ be a basis of $V$ and $\gamma\in V'$ be an arbitrary mapping from the dual space. Then we find
\begin{align*}
\exists!\, g_1,\ldots,g_n\in K: \gamma = \sum_{i = 1}^ng_i\beta_i.
\end{align*}
Since $(b_1,\ldots,b_n)$ is a basis of $V$ we find for any $x\in V$ unique $x_1,\ldots,x_n\in K$ with 
\begin{align*}
x = \sum_{i= 1}^n x_ib_i.
\end{align*}
Hence
\begin{align*}
\gamma(x) = \sum_{i = 1}^ng_i\beta_i(b_i) x_i.
\end{align*}
If we call $\gamma_i := g_i\beta_i(b_i)$, then we find
\begin{align*}
\gamma(x) = \sum_{i = 1}^n \gamma_i x_i.
\end{align*}
Since the mappings
\begin{align*}
E_1: V\to K^n, \quad x\mapsto (x_1,\ldots,x_n)
\end{align*}
and
\begin{align*}
E_2: V'\to K^n, \quad \gamma\mapsto (g_1,\ldots,g_n) = (E_2(\gamma)_1,\ldots,E_2(\gamma)_n)
\end{align*}
are isomorphisms we find the embedding
\begin{align*}
E: V'\to K[x_1,\ldots,x_n], \quad \gamma\mapsto \Big((x_1,\ldots,x_n)\mapsto \sum_{i = 1}^n E_2(\gamma)_i\beta_i(b_i)x_i\Big)
\end{align*}
with
\begin{align*}
\forall\,\gamma\in V'~\forall\,x\in V: \gamma(x) = E(\gamma)(E_1(x)).
\end{align*}
If $V = F^n$, then $E_1 = I$, $I$ being the identity mapping. If we extend Definition~\ref{def:gen_set} to arbitrary field, then we find that Proposition~\ref{ex:LinSubspAlgVar} holds true for affine linear subspaces of $F^n$ with dimension at most $n-1$. 
\end{Remark}

We show that proper real algebraic varieties are Lebesgue nullsets. Federer proves in~\cite[p.\,240]{GeomMeasTheory} that for any analytic function~$f:A\to\mathbb{R}$ with an open and connected domain~$A\subseteq\mathbb{R}^n$ the preimage of zero under~$f$ is either~$A$ or has Lebesgue measure zero. For this he proves that the set
\begin{align*}
W := \set{x\in A\,\Bigg\vert\,\forall i\in\mathbb{N}: \diff{^i}{x^i} f(x) = 0}
\end{align*}
contains~$f^{-1}(\set{0})\setminus N$ for some Lebesgue nullset~$N$ and that~$W$ is either~$A$ or a Lebesgue nullset. Since any polynomial is an analytic funcion with open and connected domain~$\mathbb{R}$, this proves the statement. 

However, we do not use Federer's method. Instead we prove Proposition~\ref{alg_var_nullset} by induction on~$n.$ This proof does not work for all analytic functions, but it is good enough for our purposes.

\begin{Proposition}\label{alg_var_nullset}
Let~$\mathbb{V}\subseteq\mathbb{R}^n$ be an algebraic variety. Then~$\mathbb{V}$ is proper if, and only if,~$\mathbb{V}$ is a closed Lebesgue nullset.
\end{Proposition}
\begin{proof}
Since~$\mathbb{R}^n$ is not a Lebesgue nullset, it suffices to show that any proper algebraic variety is a Lebesgue nullset. By Lemma~\ref{lem:MostBoringResult}, it suffices to consider algebraic varieties which are the preimage of zeros under one polynomial. Let~$\mathbb{V}:=p^{-1}(\set{0})$ for some~$$p(x_1,\ldots,x_n)\in\mathbb{R}[x_1,\ldots,x_n]\setminus\set{0}.$$ Since~$\mathbb{V}$ is closed, it is Lebesgue-Borel-measureable. It remains to prove that~$\lambda^n(\mathbb{V}) = 0$, which is shown by induction on~$n.$ If~$n = 1$, then~$\mathbb{V}\subseteq\mathbb{R}$ is finite and thus~$\lambda^1(\mathbb{V}) = 0.$ Suppose that
\begin{align*}\label{eq:proof_alg_var_nullset_2}
    \forall\, q\in\mathbb{R}[x_1,\ldots,x_{n-1}]\setminus\set{0}: \lambda^{n-1}(q^{-1}(\set{0})) = 0.\tag{$\star$}
\end{align*}
Since~$\mathbb{R}[x_1,\ldots,x_{n}]\cong\mathbb{R}[x_1,\ldots,x_{n-1}][x_n]$, 
we may consider
\begin{align*}
    p_n(x_1,\ldots,x_n) = \sum_{i=0}^\alpha q_i(x_1,\ldots,x_{n-1}) x_n^i
\end{align*}
for some~$\alpha\in\mathbb{N}$ and~$q_i\in\mathbb{R}[x_1,\ldots,x_{n-1}]$ for~$i\in\underline{\alpha}.$
Hence, for any fixed~$(z_1,\ldots,z_{n-1})\in\mathbb{R}^{n-1}$,~$p_n(z_1,\ldots,z_{n-1},x_n)$ is a polynomial in one variable, namely~$x_n$, and thus the induction assumption for~$n=1$ yields
\begin{align*}\label{eq:proof_alg_var_nullset_1}
    \lambda^1\left(p_n(z_1,\ldots,z_{n-1},\cdot)^{-1}(\set{0})\right) = 0\iff \exists\, j\in\set{0,\ldots,\alpha}: q_j(z_1,\ldots,z_{n-1})\neq 0.\tag{$\star_2$}
\end{align*}
Since~$p_n\not\equiv 0$ by assumption, we obtain
\begin{align*}
    \exists\,\ell\in\set{0,\ldots,\alpha}: q_\ell\in\mathbb{R}[x_1,\ldots,x_{n-1}]\setminus\set{0_{\mathbb{R}[x_1,\ldots,x_{n-1}]}}.
\end{align*}
Let~$\mathbb{V}_{n-1} := q_\ell^{-1}(\set{0})$ be the algebraic variety induced by this~$q_\ell.$ Then
\begin{align*}
    \widetilde{\mathbb{V}}_{n-1}:=\set{z=(z_1,\ldots,z_{n-1})\in\mathbb{R}^{n-1}\,\big\vert\,\forall\, i\in\set{1,\ldots,\alpha}: q_i(z) = 0} = \bigcap_{i=1}^\alpha q_i^{-1}(\set{0}) \subseteq \mathbb{V}_{n-1},
\end{align*}
and \eqref{eq:proof_alg_var_nullset_2} yields that~$\lambda^{n-1}\left(\mathbb{V}_{n-1}\right) = 0$, and thus~$\lambda^{n-1}\left(\widetilde{\mathbb{V}}_{n-1}\right)= 0.$ Hence by~\eqref{eq:proof_alg_var_nullset_1} the mapping
\begin{align*}
    \varphi: \mathbb{R}^{n-1}\to\mathbb{R}_{\geq 0}\cup\set{\infty},~(z_1,\ldots,z_{n-1})\mapsto \lambda^1(p_n(z_1,\ldots,z_{n-1},\cdot)^{-1}(\set{0}))
\end{align*}
vanishes on~$\widetilde{\mathbb{V}}_{n-1}^c$ and therefore we find~$\varphi = 0$ almost everywhere and especially~$\varphi(\cdot)\in\mathcal{L}^1(\mathbb{R}^{n-1},\mathbb{R}).$
Now we can use Tonelli's Theorem and arrive at
\begin{align*}
\lambda(\mathbb{V}_n) = \int_{\mathbb{R}^{n-1}\times\mathbb{R}}\chi_{\set{p_n(z,y) = 0}}(z,y)\,\mathrm{d}\lambda^n(z,y) & = \int_{\mathbb{R}^{n-1}}\underbrace{\int_\mathbb{R}\chi_{\set{p_n(z,y) = 0}}(z,y)\,\mathrm{d}\lambda^1(y)}_{=\varphi(z)}\,\mathrm{d}\lambda^{n-1}(z)\\
& = 
\int_{\mathbb{R}^{n-1}}\varphi(z)\,\mathrm{d}\lambda^{n-1}(z) = 0.
\end{align*}
\end{proof}

In Lemma~\ref{Rem:Zariski} we have shown that~$\mathcal{V}_n(\mathbb{F})$ is~closed under~$\cup$. With help of Proposition~\ref{alg_var_nullset} we can now show that~$\mathcal{V}_n^{\text{prop}}(\mathbb{F})$ is~closed under~$\cup$ if~$\mathbb{F} = \mathbb{R}.$ With some work we will prove in Lemma~\ref{cor:intersection_union} that this holds true for $\mathbb{F} = \mathbb{C}.$

\begin{Corollary}\label{cor:intersection_union1}
Let~$S_1,S_2\subseteq\mathbb{R}^n$ be generic sets. Then~$S_1^c$ is a Lebesgue nullset and~$S_1\cap S_2$,~$S_1\cup S_2$ are generic sets.
\end{Corollary}
\begin{proof}
By Definition~\ref{def:gen_set} of genericity, we find
\begin{align*}
    \forall\,{i\in\set{1,2}}~\exists\,\mathbb{V}_i\in\mathcal{V}_n^{\text{prop}}(\mathbb{R}): S_i^c\subseteq\mathbb{V}_i.
\end{align*}
Hence
\begin{align*}
    (S_1\cup S_2)^c = S_1^c\cap S_2^c\subseteq \mathbb{V}_1\cap\mathbb{V}_2\qquad\wedge\qquad(S_1\cap S_2)^c = S_1^c\cup S_2^c\subseteq \mathbb{V}_1\cup\mathbb{V}_2.
\end{align*}
In Lemma~\ref{Rem:Zariski} we have seen that~$\mathbb{V}_1\cap\mathbb{V}_2,\mathbb{V}_1\cup\mathbb{V}_2\in\mathcal{V}_n(\mathbb{R}).$ Since the union as well as the intersection of Lebesgue nullsets are Lebesgue nullsets, Proposition~\ref{alg_var_nullset} implies that the sets~$\mathbb{V}_1\cap\mathbb{V}_2,\mathbb{V}_1\cup\mathbb{V}_2\in\mathcal{V}^{\text{prop}}_n(\mathbb{R}).$ Hence the first part of the corollary is proved. Also Proposition~\ref{alg_var_nullset} yields that~$S_1^c$ is a Lebesgue nullset since~$\mathbb{V}_1$ is a Lebesgue nullset and the Lebesgue measure is complete (see e.g.~\cite[p.\,54]{MassIntTheorie}).
\end{proof}

Another nice property of generic real sets is that they have ``full measure'', i.e. their Lebesgue measure is infinite.

\begin{Corollary}\label{Cor:GenSetsInfiniteMeasure}
If~$\lambda^n(S)<\infty$ for some set~$S\subseteq\mathbb{R}^n$, then~$S$ is not generic.
\end{Corollary}
\begin{proof}
    Seeking a contradiction, assume that~$S\subseteq\mathbb{R}^n$ is a generic set with~$\lambda^n(S)<\infty.$ Then~$S^c$ is by Corollary~\ref{cor:intersection_union1} a Lebesgue nullset. The additivity of the Lebesgue measure implies
    \begin{align*}
        \lambda^n(\mathbb{R}^n) = \lambda^n(S\cup S^c) = \lambda^n(S) + \lambda^n(S^c) = \lambda^n(S)<\infty,
    \end{align*}{}
    which is a contradiction, since~$\lambda^n(\mathbb{R}^n) = \infty.$
\end{proof}

In the next lemma, we show the well-known fact that algebraic varieties can simply be embedded into higher dimensional spaces (see~\cite[Exercise 4.11, p.\,78]{UndAlgGeom}).

\begin{Lemma}\label{lem:Product_of_varieties}
If~$\mathbb{V} = \bigcap_{i=1}^kp_i^{-1}(\set{0})\in\mathcal{V}_n^{\text{prop}}(\mathbb{F})$ and~$m\in\mathbb{N}^*$, then~$\mathbb{V}\times\mathbb{F}^m\in\mathcal{V}_{n+m}^{\text{prop}}(\mathbb{F}).$
\end{Lemma}
\begin{proof}
Since~$\left(\mathbb{V}\times\mathbb{F}^m\right)^c = \mathbb{V}^c\times\mathbb{F}\neq\emptyset$, we have that~$\mathbb{V}\times\mathbb{F}^m$ is proper. We identify~$\mathbb{F}^{n+m}$ and~$\mathbb{F}^{n}\times\mathbb{F}^{m}$ and define the mappings
\begin{align*}
\forall i\in\underline{k}: P_i:\mathbb{F}^{n+m}\to\mathbb{F},\qquad\mathbb{F}^{n}\times\mathbb{F}^{m}\ni(x,y)\mapsto p_i(x).
\end{align*}
It is evident that for any~$i\in\underline{k}$ we have~$P_i\in\mathbb{F}[x_1,\ldots,x_{n+m}]$ and~$\bigcap_{i=1}^kP_i^{-1}(\set{0}) = \mathbb{V}\times\mathbb{F}^m.$
\end{proof}

We see immideately that the same statement holds true for generic sets.

\begin{Corollary}\label{lem:product_of_generic_sets}
Let~$m\in\mathbb{N}^*$ and~$S'\subseteq\mathbb{F}^n$ be a generic set. Then the set~$S = S'\times \mathbb{F}^m\subseteq\mathbb{F}^{n+m}$ is generic.
\end{Corollary}
\begin{proof}
    Let~$\mathbb{V}'\in\mathcal{V}_n^{\text{prop}}(\mathbb{F})$ be a proper algbraic variety with~$(S')^c\subseteq\mathbb{V}'.$
    The inclusion
    \begin{align*}
        S^c = (S')^c\times \mathbb{F}^m\subseteq\mathbb{V}'\times \mathbb{F}^m
    \end{align*}
    and Lemma~\ref{lem:Product_of_varieties} yield that~$S$ is generic.
\end{proof}{}

In the following lemma we show that the complement of any proper real algebraic variety is dense and open. However, the converse is in general not true and therefore the notions of a generic set and an open and dense set differ considerably.

\begin{Lemma}\label{lem:some_uninteresting_lemma}
\begin{enumerate}[(i)]
    \item If~$\mathbb{V}\subseteq\mathbb{R}^n$ is a proper algebraic variety, then~$\mathbb{V}^c$ is dense and open.
    \item If~$V^c\subseteq\mathbb{R}^n$ is a dense and open set, then~$V$ is not necessary a proper algebraic variety.
\end{enumerate}
\end{Lemma}
\begin{proof}
\begin{enumerate}[(i)]
    \item Openess follows from Lemma~\ref{Rem:Zariski}\,(i). Seeking a contradiction, suppose that~$\mathbb{V}^c$ is not dense or, equivalently,~$\mathbb{V}$ has at least one inner point. Then~$\lambda^n(\mathbb{V})>0$ and~$\mathbb{V}$ cannot be a proper algebraic variety by Proposition~\ref{alg_var_nullset}.
    \item It is well known (see e.g. \cite[Bemerkung 2.47]{Hot}) that there exists a set~$D\subseteq\mathbb{R}^n$ which is dense, open and~$\lambda^n(D^c)>0.$ Consider for example
    \begin{align*}
        D := \bigcup_{i\in\mathbb{N}}\mathbb{B}_\infty\left(\varphi(i),\frac{1}{42^i}\right)\supseteq\mathbb{Q}^n
    \end{align*}
    where~$\varphi: \mathbb{N}\to\mathbb{Q}^n$ is a bijection. Then~$D$ is open, dense and
    \begin{align*}
        \lambda^n(D) \leq\sum_{i\in\mathbb{N}}\frac{2^n}{42^{in}}\leq 2^n\sum_{i\in\mathbb{N}}\frac{1}{42^i} = 2^n \frac{42}{41}<\infty.
    \end{align*}
    By Corollary~\ref{Cor:GenSetsInfiniteMeasure}, $D^c$ is not an algebraic variety.
\end{enumerate}
\end{proof}

In Corollary~\ref{cor:intersection_union1} we have seen that~$\mathcal{V}_n(\mathbb{R})$ is~closed under~$\cup$. We want to show that this holds true for~$\mathcal{V}_n(\mathbb{C})$ and start with the fact that for any complex algebraic variety~$\mathbb{V}\in\mathcal{V}_n(\mathbb{C})$ the set~$\mathbb{V}\cap\mathbb{R}^n$ is a real algebraic variety. 

\begin{Lemma}\label{lem:ComplexPolynomialsInduceRealVarieties}
Let~$p\in\mathbb{C}[x_1,\ldots,x_n]$ be a complex polynomial. Then the set
\begin{align*}
    \mathbb{V} = \set{x\in\mathbb{R}^n\,\big\vert\,p(x) = 0}\subseteq\mathbb{R}^n
\end{align*}
is a proper algebraic variety, i.e. in $\mathcal{V}_n^{\text{prop}}(\mathbb{R})$ if, and only if,~$p\neq 0.$
\end{Lemma}
\begin{proof}
For any~$y\in\mathbb{R}^n$ we have~$p(y) = 0$ if, and only if,~$\Re p(y) = 0$ and~$\text{Im}\, p(y) = 0.$ If we put
\begin{align*}
    p(x) & = \sum_{i=1}^k p_i\prod_{j=1}^nx_j^{\nu_{i,j}},\\
    (\Re p)(x) & = \sum_{i=1}^k \Re p_i\prod_{j=1}^nx_j^{\nu_{i,j}}\in\mathbb{R}[x_1,\ldots,x_n]~\text{and}\\
    (\text{Im}\,p)(x) & = \sum_{i=1}^k \text{Im}\, p_i\prod_{j=1}^nx_j^{\nu_{i,j}}\in\mathbb{R}[x_1,\ldots,x_n],
\end{align*}
then~$\text{Im}\, p(y) = (\text{Im}\, p)(y)$ and~$\Re p(y) = (\Re p)(y)$, and so
\begin{align*}
    \mathbb{V} = (\text{Im}\, p)^{-1}(\set{0})\cap (\Re p)^{-1}(\set{0})\in\mathcal{V}_n(\mathbb{R}).
\end{align*}
Clearly, if $\mathbb{V}$ is proper, then $p\neq 0$. Conversely, if $p\neq 0_{\mathbb{C}[x_1,\ldots,x_n]}$, then $\Re p\neq 0_{\mathbb{R}[x_1,\ldots,x_n]}$ or $\Im p\neq 0_{\mathbb{R}[x_1,\ldots,x_n]}$. Hence $(\Re p)^{-1}(\set{0})\in\mathcal{V}_n(\mathbb{R})$ or $(\Im p)^{-1}(\set{0})\in\mathcal{V}_n(\mathbb{R})$ is a proper algebraic variety. By Lemma~\ref{Rem:Zariski}(ii), we get that $\mathbb{V}$ is proper. Thsi proves the second part of the Lemma.
\end{proof}

As an application of this lemma we conclude from Proposition~\ref{ex:LinSubspAlgVar} the following.

\begin{Corollary}\label{Cor:LinSubspaceReal}
Let~$V\subseteq\mathbb{F}^n$ be a linear subspace with~$\dim V\leq n-1.$ Then~$V\cap\mathbb{R}^n$ is a proper algebraic variety.
\end{Corollary}
\begin{proof}
The polynomials constructed in the proof of Proposition~\ref{ex:LinSubspAlgVar} are complex polynomials. Now the statement follows from Lemma~\ref{lem:ComplexPolynomialsInduceRealVarieties}.
\end{proof}

We conclude from Lemma~\ref{lem:ComplexPolynomialsInduceRealVarieties} another property of complex generic sets.

\begin{Lemma}\label{Cor:GenSetReal}
Let~$S\subseteq\mathbb{F}^n$ be generic w.r.t.~$\mathbb{V} = \bigcap_{i = 1}^kp_i^{-1}(\set{0})\in\mathcal{V}_n^{\text{prop}}(\mathbb{F}).$ Then~$S\cap \mathbb{R}^n$ is generic.
\end{Lemma}
\begin{proof}
    It is an elementary property that
    \begin{align*}
        \mathbb{R}^n\setminus(S\cap\mathbb{R}^n) = S^c\cap\mathbb{R}^n\subseteq\mathbb{V}\cap\mathbb{R}^n
    \end{align*}
    and it remains to investigate, whether~$\mathbb{V}\cap\mathbb{R}^n$ is a proper algebraic variety. From Lemma~\ref{lem:ComplexPolynomialsInduceRealVarieties} we obtain that~$\mathbb{V}\cap\mathbb{R}^n = \set{x\in\mathbb{R}^n\,\big\vert\,\forall i\in\underline{k}:p_i(x) = 0}$ is an algebraic variety. Since~$\mathbb{V}$ is proper there exists some~$i\in\underline{k}$ so that~$p_i\neq 0_{\mathbb{F}[x_1,\ldots,x_n]}.$ Hence~$\mathbb{V}\cap\mathbb{R}^n$ is proper.
\end{proof}

Another property for later use is that $\mathcal{V}_n^{\text{prop}}(\mathbb{F})$ is invariant under regular transformations.

\begin{Lemma}\label{Lem:transformgenericsets}
Let~$S\subseteq\mathbb{F}^n$,~$T\in\mathcal{GL}(\mathbb{F}^n)$,~$p\in\mathbb{F}[x_1,\ldots,x_n]$, and~$\mathbb{V}\in\mathcal{V}_n^{\text{prop}}(\mathbb{F}).$ Then
\begin{enumerate}[(i)]
    \item~$(TS)^c = TS^c$,
    \item there exists a unique~$q\in\mathbb{F}[x_1,\ldots,x_n]$ such that
    \begin{align*}
        \forall x\in\mathbb{F}^n: p(Tx) = q(x),
    \end{align*}
\item~$T\mathbb{V}\in\mathcal{V}_n^{\text{prop}}(\mathbb{F})$,
\item~$S$ is generic w.r.t.~$\mathbb{V}$ if, and only if,~$TS$ is generic w.r.t.~$T\mathbb{V}.$
\end{enumerate}
\end{Lemma}
\begin{proof}
\begin{enumerate}[(i)]
\item This is an elementary property of a bijective mapping.
\item Let~$p(x)$ be given as in (1.1). Then
\begin{align*}
p(Tx) = \sum_{k=0}^\ell a_k(Tx)_1^{\nu_{k,1}}\cdots(Tx)_n^{\nu_{k,n}}
\end{align*}
and since
\begin{align*}
\forall i\in\underline{n}: (Tx)_i = \sum_{j=1}^nT_{i,j}x_j
\end{align*}
it follows that~$q$ exists.
\item Let~$\mathbb{V} = \bigcap_{i=1}^kp_i^{-1}(\set{0}).$ Then the equivalence
\begin{align*}
x\in\mathbb{V} & \iff \forall i\in\underline{k}: 0 = p_i(x) = p_i(T^{-1}Tx)
\end{align*}
and (ii) imply that~$T\mathbb{V} = \bigcap_{i=1}^k\left(p_i(T^{-1})\right)^{-1}(\set{0})$ is an algebraic variety and proper since~$T$ is bijective.
\item From (i) and~$T\in\mathcal{GL}(\mathbb{F}^n)$ we conclude
\begin{align*}
S^c\subseteq\mathbb{V} \iff (TS)^c = TS^c\subseteq T\mathbb{V}.
\end{align*}
(iii) implies that~$T\mathbb{V}$ is a proper algebraic variety if, and only if~$\mathbb{V}$ is an algebraic variety. 
\end{enumerate}
This completes the proof of the lemma.
\end{proof}

A next step to prove that $\mathcal{V}_n^{\text{prop}}(\mathbb{C})$ is closed under~$\cup$ is the embedding into $\mathcal{V}_{2n}^{\text{prop}}(\mathbb{R}).$

\begin{Lemma}\label{lem:complexvarietieshochziehen}
Consider the $\mathbb{R}$-vector space isomorphism
\begin{align*}
R: \mathbb{C}^n\to\mathbb{R}^{2n}, z = \Re z + i\,\Im z\mapsto Rz := \begin{pmatrix}\Re z\\\Im z\end{pmatrix}
\end{align*}
and let~$\mathbb{V}\in\mathcal{V}_n(\mathbb{C}).$
Then~$R\mathbb V$ is a proper algebraic variety if, and only if,~$\mathbb{V}$ is proper.
\end{Lemma}
\begin{proof}
Since~$R$ is an isomorphism, the equivalence 
\begin{align*}
\forall\,\mathbb{V}\subseteq\mathbb{C}^n: R\mathbb{V} = \mathbb{R}^{2n}\iff\mathbb{V} = \mathbb{C}^n
\end{align*}
holds true. Let~$\mathbb{V} = \bigcap_{i=1}^k p_i^{-1}(\set{0})\in\mathcal{V}_n(\mathbb{C})$ be an algebraic variety with
\begin{align*}
p_i(x) = \sum_{j=1}^{k_i}p_i^jx_1^{\nu_{j,1}}\cdots x_n^{\nu_{j,n}}.
\end{align*}
Then we find
\begin{align*}
\forall\,z\in\mathbb{C}~\forall\,\ell\in\underline{k}: p_\ell(z) & = \sum_{j=1}^{k_\ell}p_\ell^jz_1^{\nu_{j,1}}\cdots z_n^{\nu_{j,n}}\\
& = \sum_{j=1}^{k_\ell}p_\ell^j((Rz)_1+i(Rz)_{n+1})^{\nu_{j,1}}\cdots ((Rz)_n+i(Rz)_{2n})^{\nu_{j,n}}\\
& =: \sum_{r = 1}^{m_\ell}q_\ell^r(Rz)_1^{\mu_{r,1}}\cdots (Rz)_{2n}^{\mu_{r,2n}} =: q_\ell(Rz).
\end{align*}
Thus
\begin{align*}
z\in\mathbb{V}\iff Rz\in\bigcap_{i=1}^k q_i^{-1}(\set{0})\cap\mathbb{R}^{2n}.
\end{align*}
We conlcude with Lemma~\ref{lem:ComplexPolynomialsInduceRealVarieties} that~$R\mathbb{V}$ is an algebraic variety.
\end{proof}

Finally we prove that $\mathcal{V}_n^{\text{prop}}(\mathbb{F})$ is closed under~$\cap$ and $\cup$ not only for $\mathbb{F} = \mathbb{R}$, which was proven in Corollary~\ref{cor:intersection_union1}, but also for $\mathbb{F} = \mathbb{C}.$

\begin{Lemma}\label{cor:intersection_union}
Let~$S_1,S_2\subseteq\mathbb{C}^n$ be generic sets. Then~$S_1\cap S_2$ and~$S_1\cup S_2$ are generic sets.
\end{Lemma}
\begin{proof}
In the proof of Corollary~\ref{cor:intersection_union1} we have seen that this statement is equivalent to
\begin{align*}
\mathcal{V}_n^{\text{prop}}(\mathbb{C})~\text{is}~\text{closed under~$\cap$~and~$\cup$}.
    \end{align*}
With Lemma~\ref{Rem:Zariski}(ii) it remains to show that~$\mathcal{V}_n^{\text{prop}}(\mathbb{C})$ is~closed under~$\cup$. Since~$\mathcal{V}_n(\mathbb{C})$ is, by Lemma~\ref{Rem:Zariski},~closed under~$\cup$, we have to show that~$\mathbb{C}^n$ can not be partitioned into two proper algebraic varieties. Seeking a contradiction assume
\begin{align*}
\exists\,\mathbb{V}_1,\mathbb{V}_2\in \mathcal{V}_n^{\text{prop}}(\mathbb{C}): \mathbb{V}_1\cup\mathbb{V}_2 = \mathbb{C}^n.
\end{align*}
Let~$R$ be the mapping from Lemma~\ref{lem:complexvarietieshochziehen}. Then we find with
\begin{align*}
R\mathbb{V}_1\cup R\mathbb{V}_2 = R(\mathbb{V}_1\cup\mathbb{V}_2) = R\mathbb{C}^n = \mathbb{R}^{2n}
\end{align*}
and Lemma~\ref{lem:complexvarietieshochziehen} a partition of~$\mathbb{R}^{2n}$ into two proper algebraic varieties, namely~$R\mathbb{V}_1$ and~$R\mathbb{V}_2.$ With Proposition~\ref{alg_var_nullset} this means that~$\mathbb{R}^{2n}$ can be partitionated into finitely many Lebesgue nullsets, which is a contradiction to the fact that~$\mathbb{R}^{2n}$ is not a Lebesgue nullset.
\end{proof}

As a corollary we see that the preimage of any finite set under a polynomial is an algebraic variety.

\begin{Corollary}
For any non-constant~$p\in\mathbb{F}[x_1,\ldots,x_n]$ and~$z_0,\ldots,z_k\in\mathbb{F},~k\in\mathbb{N}$, the sets
\begin{itemize}
\item~$p^{-1}(\set{z_0})$
and
\item~$p^{-1}(\set{z_0,\ldots,z_k}).$
\end{itemize}
are proper algebraic varieties. If~$\mathbb{F} = \mathbb{R}$, then the sets
\begin{itemize}
\item~$\set{x\in\mathbb{R}^n\,\big\vert\,\abs{p(x)} = \abs{z_1}}$ and
\item~$\set{x\in\mathbb{R}^n\,\big\vert\,\abs{p(x)}\in\set{\abs{z_1},\ldots,\abs{z_k}}}$
\end{itemize}
are also proper algebraic varieties.
\end{Corollary}
\begin{proof}
This follows from Lemma~\ref{cor:intersection_union} and the following calculations:
\begin{itemize}
\item~$p^{-1}(\set{z_0}) = (p-z_0)^{-1}(\set{0})$,
\item~$p^{-1}(\set{z_0,\ldots,z_k}) = \bigcup_{i=0}^np^{-1}(\set{z_i})$
\end{itemize}
for $\mathbb{F} = \mathbb{R}$ or $\mathbb{F} = \mathbb{C}$ and
\begin{itemize}
\item~$\set{x\in\mathbb{R}^n\,\big\vert\,\abs{p(x)}\in\set{\abs{z_1}}} = (p^2-z_0^2)^{-1}(\set{0})$ and
\item~$\set{x\in\mathbb{R}^n\,\big\vert\,\abs{p(x)}\in\set{\abs{z_1},\ldots,\abs{z_k}}} = \bigcup_{i=0}^np^{-1}(\set{\abs{z_i}})$
\end{itemize}
for $\mathbb{F} = \mathbb{R}$.
\end{proof}

In a next step we want to talk about rational mappings. The ring of rational functions over $\mathbb{F}$ in $n$ indeterminants is the set
\begin{align*}
\mathbb{F}(x_1,\ldots,x_n) := \set{\frac{p}{q}\,\Bigg\vert\,p,q\in\mathbb{F}[x_1,\ldots,x_n],~q\neq 0_{\mathbb{F}[x_1,\ldots,x_n]}}.
\end{align*}
If we want to consider a rational function $\frac{p}{q}\in\mathbb{F}(x_1,\ldots,x_n)$ as a mapping on a subset of $\mathbb{F}^n$, then there are two possibilities. At first we can put
\begin{align*}
\frac{p}{q}: \mathbb{F}^n\setminus q^{-1}(\set{0}), \quad x\mapsto\frac{p(x)}{q(x)}.
\end{align*}
The second possibility is using reduction. If $p$ and $q$ have the common zeros $y_1,\ldots,y_k\in\mathbb{F}$, then we find $\widehat{p}\in\mathbb{F}[x_1,\ldots,x_n]$ and $\widehat{q}\in\mathbb{F}[x_1,\ldots,x_n]$ so that $\widehat{p}$ and $\widehat{q}$ have no common zeros and
\begin{align*}
p(x) = \widehat{p}(x)\prod_{j = 1}^k(x-y_j) \qquad\wedge\qquad q(x) = \widehat{q}(x)\prod_{j = 1}^k(x-y_j).
\end{align*}
We can then put
\begin{align*}
\frac{p}{q}: \mathbb{F}^n\setminus \widehat{q}^{-1}(\set{0}), \quad x\mapsto\frac{\widehat{p}(x)}{\widehat{q}(x)}.
\end{align*}
It is notable that the second option is a continuous extension of the first option.
 
However, for simplicity we will use the first option and take rational functions as mappings with domain
\begin{align*}
\dom \frac{p}{q} := \mathbb{F}^n\setminus q^{-1}(\set{0}).
\end{align*}

With this agreement we see immideately that, in addition to polynomials and polynomial vectors, the preimage of zero under rational functions is an algebraic variety.

\begin{Corollary}~\label{cor:rational_functions}
For~$q(x) \in\mathbb{F}(x_1,\ldots,x_n)$ the following statements hold:
\begin{enumerate}[(i)]
    \item~$(\dom q)^c\subseteq\mathbb{F}^n$ is a proper algebraic variety and
    \item~$q^{-1}(\set{0})\cup(\dom q)^c$ is an algebraic variety and proper if, and only if,~$q\neq 0_{\mathbb{F}(x_1,\ldots,x_n)}.$
\end{enumerate}
\end{Corollary}
\begin{proof} Let~$q = \frac{p_1}{p_2}$ with~$p_1,p_2\in\mathbb{F}[x_1,\ldots,x_n]$ and~$p_2\neq 0.$ Then we find:
\begin{enumerate}[(i)]
    \item We use the definition $\dom q = \mathbb{F}^n\setminus p_2^{-1}(\set{0})$ and hence it is evident that $(\dom q)^c = p_2^{-1}(\set{0})$ is a proper algebraic variety.
    \item By definition we have~$q^{-1}(\set{0}) = p_1^{-1}(\set{0})\cap\dom q.$ Hence the set
    \begin{align*}
        q^{-1}(\set{0})\cup(\dom~q)^c & = (p_1^{-1}(\set{0})\cap\dom~q)\cup(\dom~q)^c\\
        & = p_1^{-1}(\set{0})\cup(\dom~q)^c\\
        & = p_1^{-1}(\set{0})\cup p_2^{-1}(\set{0})
    \end{align*}
    is an algebraic variety. Since $p_2^{-1}(\set{0})$ is a proper algebraic variety, $q^{-1}(\set{0})$ is by Lemma~\ref{cor:intersection_union} proper if, and only if, $p_1^{-1}(\set{0})$ is proper. This is the case if, and only if, $p_1\neq 0$ or, equivalentely, $q\neq 0$.
\end{enumerate}
\end{proof}

With the help of Lemma~\ref{cor:intersection_union} we conclude from Proposition~\ref{alg_var_nullset} that $\mathbb{F}^n$ can not be partitionated into more than one generic set.

\begin{Corollary}\label{Cor:PartitionInGenericSets}
$\mathbb{F}^n$ can not be partitioned into two generic sets.
\end{Corollary}
\begin{proof}
Seeking a contradiction assume that~$S_1,S_2\subseteq\mathbb{F}^n$ is a partition of~$\mathbb{F}^n$, i.e.~$S_1,S_2$ are nonempty disjoint sets with~$S_1\cup S_2 = \mathbb{F}^n.$ Let~$S_i$ be generic w.r.t.~$\mathbb{V}_i\in\mathcal{V}_n^{\text{prop}}(\mathbb{F}).$ Then~$S_2 = S_1^c\subseteq\mathbb{V}_1$ and~$S_1 = S_2^c\subseteq\mathbb{V}_2.$ This yields~$\mathbb{V}_1\cup\mathbb{V}_2 = \mathbb{F}^n$, which is a contradiction to Lemma~\ref{cor:intersection_union}.
\end{proof}

Another implication from Lemma~\ref{cor:intersection_union} is that the Cartesian product of finitely many generic sets is generic.

\begin{Lemma}
Let~$m\in\mathbb{N}^*$ and~$S'\subseteq\mathbb{F}^n,\widehat{S}\subseteq\mathbb{F}^{m}$ be generic sets. Then the set~$S = S'\times \widehat{S}\subseteq\mathbb{R}^{n+m}$ is generic.
\end{Lemma}
\begin{proof}
This follows from Corollary~\ref{lem:product_of_generic_sets}, Lemma~\ref{cor:intersection_union} and~$S = \big(S'\times\mathbb{F}^m\big)\cap\big(\mathbb{F}^n\times\widehat{S}\big).$
\end{proof}

In Proposition~\ref{alg_var_nullset} we have seen that the complement of a real generic set is a Lebesgue nullset. Now we state that generic sets are dense.

\begin{Corollary}\label{cor:some_even_more_boring_corollary}
Any generic set~$S\subseteq\mathbb{F}^n$ is dense.
\end{Corollary}
\begin{proof}
If~$\mathbb{F} = \mathbb{C}$, then we use the mapping~$R$ from Lemma~\ref{lem:complexvarietieshochziehen}, which is bounded and has bounded inverse. Thus any set~$S\subseteq\mathbb{C}^n$ is dense if, and only if,~$RS\subseteq\mathbb{R}^{2n}$ is dense. Hence it suffices to show the statement for the case~$\mathbb{F}=\mathbb{R}.$

Let~$S\subseteq\mathbb{R}^n$ be a generic set. Then there exists some proper algebraic variety~$\mathbb{V}\in\mathcal{V}_n^{\text{prop}}(\mathbb{F})$ with~$\mathbb{V}^c\subseteq S.$ In Lemma~\ref{lem:some_uninteresting_lemma} we have shown that~$\mathbb{V}^c$ is dense and hence~$S$ is also dense.
\end{proof}

A useful property of generic sets is the next Lemma.

\begin{Lemma}\label{lem:generic_simplification}
Let~$S_1\subseteq\mathbb{F}^n$ be generic and~$S_2,S\subseteq\mathbb{F}^n$ with~$S\cap S_1 = S_2\cap S_1$. Then~$S$ is generic if, and only if,~$S_2$ is generic.
\end{Lemma}
\begin{proof}
Since the statement is symmetric, only one direction has to be shown. Suppose~$S$ is generic. Then Lemma~\ref{cor:intersection_union} yields that~$S\cap S_1 = S_2\cap S_1$ is generic. Since~$S_2\cap S_1\subseteq S_2$, genericity of~$S_2$ follows from Remark~\ref{Rem:InclusionOfGenericSets}.
\end{proof}

\newpage

\section{Rank properties}

In the remainder of this section and any further section, let 
\[
\ell,n,m,p,q,n_{1,1},\ldots,n_{p,q},m_{1,1},\ldots,m_{p,q}\in\mathbb{N}^*\quad\text{and}\quad d,g,g_{1,1},\ldots,g_{p,q}\in\mathbb{N},
\]
if not said otherwise

It is well-known that various controllability concepts for linear ordinary differential equations (ODEs) and differential-algebraic equations (DAEs) can be characterized by rank properties of certain matrices. Thus if we show genericity of those rank properties, we have genericity of the controllability concepts. These rank properties are the properties of block matrices. Hence we want to study block matrices of the form
\begin{align*}
M = \begin{bmatrix}M_{1,1} & \cdots & M_{1,q}\\\vdots & \ddots & \vdots\\ M_{p,1} & \cdots & M_{p,q}\end{bmatrix}
\end{align*}
with
\begin{align*}
\forall\,i\in\underline{p}~\forall\,j\in\underline{q}: M_{i,j}\in\mathbb{F}^{n_{i,j}\times m_{i,j}}.
\end{align*}
We need some conditions on $m_{i,j}$ and $n_{i,j}$ to guarantee that $M\in\mathbb{F}^{n\times m}.$ At first we need to guarantee that the block matrices of the rows,~$M_{i,1},\ldots,M_{i,q}, i\in\underline{p}$, have a consistent format of rows. This will be done by
\begin{align*}
\forall\, i\in\underline{p}~\forall\, j,k\in\underline{q}: \big(n_{i,j} = n_{i,k}\big)\,\wedge\,\left(\sum_{\alpha=1}^pn_{\alpha,k} = n\right).\tag{3.0a}
\end{align*}
Next we must guarantee that the block matrices of the columns,~$M_{1,j},\ldots,M_{q,j}, j\in\underline{q}$ have a consistent format of columns, which we do with
\begin{align*}
\forall\, i,j\in\underline{p}~\forall\, k\in\underline{q}: \big(m_{i,k} = m_{j,k}\big)\,\wedge\,\left(\sum_{\beta=1}^qm_{j,\beta} = m\right),\tag{3.0b}
\end{align*}
Finally, both conditions guarantee that~$M\in\mathbb{F}^{n\times m}.$
%

\begin{Remark}\label{Rem:IsomorphismBlockMatrix}
The first step to talk about genericity of sets of matrices is to identify any matrix space of the form~$
\left(\mathbb{F}^{n_{1,1}\times m_{1,1}}\right)^{g_{1,1}}\times\cdots\times\left(\mathbb{F}^{n_{p,q}\times m_{p,q}}\right)^{g_{p,q}}
$
with~$\mathbb{F}^G$, where~$G$ is given in (3.0c). We define the operator
\begin{align*}
& T: \left(\mathbb{F}^{n_{1,1}\times m_{1,1}}\right)^{g_{1,1}}\times\cdots\times\left(\mathbb{F}^{n_{p,q}\times m_{p,q}}\right)^{g_{p,q}}\to\mathbb{F}^G\\
& \Big(\big(P_0^{1,1},\ldots,P^{1,1}_{g_{1,1}}\big),\ldots,\big(P_0^{p,q},\ldots,P^{p,q}_{g_{p,q}}\big)\Big)\mapsto \Big(\big(P_0^{1,1}\big)_{1,1},\ldots, \big(P_0^{1,1}\big)_{n_{1,1},m_{1,1}},\ldots,\big(P_{g_{p,q}}^{p,q}\big)_{n_{p,q},m_{p,q}}\Big).
\end{align*}
Clearly~$T$ is a well-defined isomorphism. Any polynomial~$p\in\mathbb{F}[x_1,\ldots,x_G]$ can be interpreted as a mapping from~$\left(\mathbb{F}^{n_{1,1}\times m_{1,1}}\right)^{g_{1,1}}\times\cdots\times\left(\mathbb{F}^{n_{p,q}\times m_{p,q}}\right)^{g_{p,q}}$ to~$\mathbb{F}$, if we consider~$p\circ T.$ Hereafter we will write~$p$ instead of~$p\circ T$ for shortness.
\end{Remark}

An important part of the study of the rank of some matrix is the concept of a minor.

\begin{Definition}[Submatrix, minor]~\label{def:submatrix}
Let~$s,t\in\mathbb{N}^*$ fulfill the inequalities
\[
s\leq n,~t\leq m\qquad\mathrm{and}\qquad d\leq \min\set{s,t},
\]
$\sigma: \underline {s}\to\underline n$ and~$\pi: \underline {t}\to\underline m$ be injective. Then we define the \textit{submatrix induced by~$\sigma$ and~$\pi$} as
\[
   m_{\sigma,\pi}:\mathbb{F}^{n\times m}\to \mathbb{F}^{s\times t},\quad A\mapsto \left[A_{\sigma(i),\pi(j)}\right]_{i\in\underline{n},j\in\underline{m}}
\]
If~$d\geq 1$,~$\sigma$ and~$\pi$ are monotonically increasing\footnote{A minor could also be defined using both monotonically increasing and decreasing mappings, in which case any minor would appear up to four times since determinants do not change, if the rows and columns are interchanged. The use of monotonically increasing mappings is an agreement to simplify matters.} and~$s = t = d$, then the mapping
\[
M_{\sigma,\pi}: \mathbb{F}^{n\times m}\to \mathbb{F},\quad A\mapsto \det m_{\sigma,\pi}(A)
\]
is called \textit{minor of degree~$d$} (w.r.t.~$\mathbb{F}^{n\times m}$). A \textit{minoe of degree~$0$} (w.r.t.~$\mathbb{F}^{n\times m}$) is the constant mapping
\begin{align*}
M_0: \mathbb{F}^{n\times m}\to\mathbb{F},\quad A\mapsto 1.
\end{align*}
\end{Definition}

We collect the following facts.

\begin{Remark}\label{rem:submatrix_rank}
\begin{enumerate}[(i)]
\item If~$\sigma$ and~$\pi$ in Definition~\ref{def:submatrix} are injective, then 
\begin{align*}
\forall\, A\in{\mathbb{F}}^{n\times m}:\rk m_{\sigma,\pi}(A)\leq\rk A.
\end{align*}
\item Let~$M_{\sigma,\pi}$ be a minor. The Leibniz formula implies
\begin{align*}
\forall\, A\in \mathbb{F}^{n\times m}: M_{\sigma,\pi}(A) = \sum_{\tau\in S_d}\sign\,\tau\prod_{i = 1}^dA_{\sigma(i),\pi(\tau(i))},
\end{align*}
and hence~$M_{\sigma,\pi}$ is a polynomial in the entries of the matrix and we can write
\[
M_{\sigma,\pi}\in \mathbb{F}[x_1,\ldots,x_{nm}].
\]
\end{enumerate}
\end{Remark}

We investigate how some matrices and some transformations interact in the sense of the submatrix.

\begin{Lemma}\label{lem:TransformingMinors}
Let~$\pi:\underline{t}\to\underline{m}$ for~$1\leq t\leq m$ be injective and~$T\in \mathbb{F}^{n\times n}$ be an arbitrary matrix. Then
\begin{align*}
\forall\, M\in {\mathbb{F}}^{n\times m}: m_{\id,\pi}(TM) = Tm_{\id,\pi}(M).
\end{align*}
\end{Lemma}
\begin{proof}
Let~$M\in{\mathbb{F}}^{n\times m}$ be arbitrary. Then we get
\begin{align*}
\forall\, i\in\underline{m}: (TM)_{\cdot,i} = T M_{\cdot,i},
\end{align*}
which proves the lemma.
\end{proof}

The next proposition shows that the set of matrices with full rank is generic. This implies especially that a random quadratic matrix is almost surely invertible. For~$m = n = 1$ this is clear since~$\set{0}$ is a Lebesgue nullset, but for higher dimensions it is not that evident.

\begin{Proposition}[Generic full rank property for block matrices]\label{prop:full_rank_property_block}
Let
\begin{align*}
\mathfrak{F} := \mathbb{F}^{n_{1,1}\times m_{1,1}}\times\cdots\times\mathbb{F}^{n_{p,q}\times m_{p,q}}.
\end{align*}
The set
\begin{align*}
S^c = \set{ \big(A_{1,1},\ldots,A_{p,q}\big)\in \mathfrak{F}\,\left\vert\,\rk\underbrace{\begin{bmatrix}A_{1,1} & \cdots & A_{1,q}\\\vdots & \ddots & \vdots\\ A_{p,1} & \cdots & A_{p,q}\end{bmatrix}}_{\in \mathbb{F}^{n\times m}}<d \right.}
\end{align*}
is a proper algebraic variety if, and only if,~$d\leq\min\set{n,m}.$ Hence~$S$ is a generic set if, and only if,~$d\leq\min\set{n,m}.$
\end{Proposition}
\begin{proof}
It is evident that~$S^c\neq\emptyset$, since~$\big(0_{m_{1,1}\times m_{1,1}},\ldots,0_{m_{p,q}\times m_{p,q}}\big)\in S^c.$ Hence it suffices to investigate whether~$S^c$ is a proper algebraic variety.
\begin{enumerate}
\item[$\implies$] If~$d>\min\set{n,m}$ then~$S^c = \mathfrak{F}$,
which is an algebraic variety but not proper.
\item[$\impliedby$] Let~$\widetilde{M}_1(\cdot),\ldots,\widetilde{M}_r(\cdot)$ be all minors of order~$d$ w.r.t.~$\mathbb{F}^{n\times m}$ and define
\begin{align*}
\forall\, i\in\underline{r}: M_i: \mathfrak{F}\to\mathbb{F},\quad A\mapsto \widetilde{M}_i\begin{bmatrix}A_{1,1} & \cdots & A_{1,q}\\\vdots & \ddots & \vdots\\ A_{p,1} & \cdots & A_{p,q}\end{bmatrix}.
\end{align*}
Then
\begin{align*}
S^c = \bigcap_{i=1}^r \left(M_i\right)^{-1}(\set{0})
\end{align*}{}
is a algebraic variety by Lemma~\ref{Rem:Zariski}\,(ii). It remains to prove that $S^c$ is proper. Since $d\leq\min\set{n,m}$, there exists some $M\in\mathbb{F}^{n\times m}$ with $\rk M = \min\set{n,m}\geq d$. Define, for any $(i,j)\in\underline{p}\times\underline{q}$ the mappings
\begin{align*}
\sigma_{i,j}: \underline{n_{i,j}}\to\underline{n},\quad k\mapsto \sum_{\ell = 1}^{i-1} n_{i,1}+k\\
\pi_{i,j}: \underline{m_{i,j}}\to\underline{m},\quad k\mapsto \sum_{\ell = 1}^{j-1} m_{1,j}+k
\end{align*}
and with the mapping $m_{\sigma,\pi}$ from Definition~\ref{def:submatrix} the matrices
\begin{align*}
A_{i,j} := m_{\sigma_{i,j},\pi_{i,j}}(M),\quad i\in\underline{p},j\in\underline{q}.
\end{align*}
Then we have $(A_{1,1},\ldots,A_{p,q})\in\mathfrak{F}$ and 
\begin{align*}
\rk\begin{bmatrix}A_{1,1} & \cdots & A_{1,q}\\\vdots & \ddots & \vdots\\ A_{p,1} & \cdots & A_{p,q}\end{bmatrix} = \rk M\geq d.
\end{align*}
Hence $(A_{1,1},\ldots,A_{p,q})\in S$ and $S^c$ is a proper algebraic variety. This completes the proof of the proposition.
\end{enumerate}
\end{proof}

For the remainder of the section we turn our attention to polynomial block matrices.

\begin{Remark}[Polynomial matrices]\label{Rem:IsomorphismPolBlockMatrix}
The linear operator
\begin{align*}
\widehat{T}_g:\left(\mathbb{F}^{n\times m}\right)^{g+1}\to\mathbb{F}[x]^{n\times m},\quad (P_0,\ldots P_g)\mapsto \sum_{i = 0}^g P_ix^i
\end{align*}
satisfies
\begin{align*}
\im\widehat{T}_g = \set{P(x)\in\mathbb{F}[x]^{n\times m}\Bigg\vert \max_{(i,j)\in\underline{n}\times\underline{m}}\deg\,\big(P(x)\big)_{i,j}\leq g}.
\end{align*}
Hereafter we will write~$P(x)$ instead of~$\widehat{T}_g(P_0,\ldots,P_g).$ Define the linear operator
\begin{align}\label{eq:HutOperator}
\widehat{\cdot}: &~\left(\mathbb{F}^{n_{1,1}\times m_{1,1}}\right)^{g_{1,1}+1}\times\cdots\times\left(\mathbb{F}^{n_{p,q}\times m_{p,q}}\right)^{g_{p,q}+1} \to\mathbb{F}[x]^{n\times m},\\
& P = \Big(\big(P_0^{1,1},\ldots,P^{1,1}_{g_{1,1}}\big),\ldots,\big(P_0^{p,q},\ldots,P^{p,q}_{g_{p,q}}\big)\Big)\mapsto \widehat{P} := \begin{bmatrix}P^{1,1}(x) & \cdots & P^{1,q}(x)\\\vdots & \ddots & \vdots\\ P^{p,1}(x) & \cdots & P^{p,q}(x)\end{bmatrix}.
\end{align}
As already discussed, the conditions (3.0a) and (3.0b) gurantee that the operators~$\widehat{T}_g$ and $\widehat{\cdot}$ are well-defined.
\end{Remark}

\begin{Remark}
The Remarks~\ref{Rem:IsomorphismPolBlockMatrix} and~\ref{Rem:IsomorphismBlockMatrix} imply the isomorphy
\begin{align*}
\left(\mathbb{F}^{n_{1,1}\times m_{1,1}}\right)^{g_{1,1}+1}\times\cdots\times\left(\mathbb{F}^{n_{p,q}\times m_{p,q}}\right)^{g_{p,q}+1} & \simeq\mathbb{F}^G
\end{align*}
and
\begin{align*}
 \mathbb{F}^G\simeq\set{\left.\begin{bmatrix}P^{1,1}(x) & \cdots & P^{1,q}(x)\\\vdots & \ddots & \vdots\\ P^{p,1}(x) & \cdots & P^{p,q}(x)\end{bmatrix}\,\right\vert\,\forall\, (i,j)\in\underline{p}\times\underline{q}:\max_{k\in n_{i,j},\ell\in m_{i,j}}\deg \big(P^{i,j}(x)\big)_{k,\ell}\leq g_{i,j}}.
\end{align*}
\end{Remark}

Throughout the following observations we will make use of the notation $\widehat{P}(x)$ as introduced in Remark~\ref{Rem:IsomorphismPolBlockMatrix}.

It is well-known that the coefficients of the determinant of a polynomial matrix~$P(x)\in\mathbb{F}[x]^{n\times n}$ are polynomials in the coefficients of the entries of~$P.$ This implies that the coefficients of any minor of a polynomial matrix are polynomials in their entries. 

\begin{Lemma}\label{lem:minors_are_polynomials}
Let
\begin{align*}
\mathfrak{F} := \left(\mathbb{F}^{n_{1,1}\times m_{1,1}}\right)^{g_{1,1}+1}\times\cdots\times\left(\mathbb{F}^{n_{p,q}\times m_{p,q}}\right)^{g_{p,q}+1},
\end{align*}
$\sigma:\underline{d}\to\underline{n}$ and~$\pi:\underline{d}\to\underline{m}$ be injective mappings and
\begin{align*}
M: &~\mathfrak{F} \to\mathbb{F}[x],\quad P \mapsto M_{\sigma,\pi}\widehat{P}(x).
\end{align*}
Then
\begin{align*}
\exists\,h\in\mathbb{N}~\exists\,m_0,\ldots, m_h\in\mathbb{F}\left[x_1,\ldots,x_G\right]~\forall\,P\in\mathfrak{F}: M(P) = \sum_{i = 0}^hm_i(P)x^i.
\end{align*}
\end{Lemma}
\begin{proof}
Let~$P=\Big(\big(P_0^{1,1},\ldots,P^{1,1}_{g_{1,1}}\big),\ldots,\big(P_0^{p,q},\ldots,P^{p,q}_{g_{p,q}}\big)\Big)\in\mathfrak{F}$ be arbitrary.
By the Leibniz formula we find
\begin{align*}
M(P) = \sum_{\tau\in S_d}\sign(\tau)\prod_{i = 1}^d\widehat{P}(x)_{\sigma(i),\pi(\tau(i))}.
\end{align*}
Then
\begin{align*}
\forall\, (i,j)\in\underline{n}\times\underline{m}: \deg\widehat{P}(x)_{i,j}\leq h' := \max_{(i,j)\in\underline{p}\times\underline{q}}g_{i,j}
\end{align*}
and
\begin{align*}
\deg M(P)\leq h := \sum_{(i,j)\in\underline{p}\times\underline{q}}\min\set{n_{i,j},m_{i,j}}g_{i,j}.
\end{align*}
Let~$\tau\in S_d$ and~$i\in\underline{d}$ be arbitrary. By definition of~$\widehat{P}(x)$ we find
\begin{align*}
\exists\,m_0^{i,\tau,\sigma,\pi},\ldots,m_h^{i,\tau,\sigma,\pi}\in\mathbb{F}[x_1,\ldots,x_G]~\forall\, P\in\mathfrak{F}: \widehat{P}(x)_{\sigma(i),\pi(\tau(i))} = \sum_{j = 0}^h m_j^{i,\tau,\sigma,\pi}(P)x^i.
\end{align*}
By multiplying we get
\begin{align*}
\exists\, \widehat{m}_0^{\tau,\sigma,\pi},\ldots,\widehat{m}_h^{\tau,\sigma,\pi}\in\mathbb{F}[x_1,\ldots,x_G]~\forall\,P\in\mathfrak{F}:\prod_{i = 1}^d\widehat{P}(x)_{\sigma(i),\pi(\tau(i))} = \sum_{j = 0}^h \widehat{m}_j^{\tau,\sigma,\pi}(P)x^j
\end{align*}
and thus the statement is proven.
\end{proof}

With this we show that Proposition~\ref{prop:full_rank_property_block} can be generalized to polynomial block matrices.

\begin{Proposition}[Generic full rank property for polynomial block matrices]\label{Prop:gen_full_rk_pol_block}
\textcolor{white}{Leer}
\newline Let
\begin{align*}
\mathfrak{F} := \left(\mathbb{F}^{n_{1,1}\times m_{1,1}}\right)^{g_{1,1}+1}\times\cdots\times\left(\mathbb{F}^{n_{p,q}\times m_{p,q}}\right)^{g_{p,q}+1}.
\end{align*}
The set
\begin{align*}
S^c = \set{P\in \mathfrak{F}\,\left|\,\rk_{\mathbb{F}(x)}\widehat{P}(x)<d\right.}
\end{align*}
is a proper algebraic variety if, and only if,~$d\leq\min\set{n,m}.$ Hence~$S$ is generic w.r.t.~$\mathbb{V} = S^c$ if, and only if,~$d\leq\min\set{n,m}.$
\end{Proposition}
\begin{proof}
\begin{enumerate}
\item[$\implies$] If~$S^c$ is a proper algebraic variety, then~$S\neq\emptyset$ or, equivalently, there exists some~$\Big(\big(P_0^{1,1},\ldots,P^{1,1}_{g_{1,1}}\big),\ldots,\big(P_0^{p,q},\ldots,P^{p,q}_{g_{p,q}}\big)\Big)\in\mathfrak{F}$ such that
\begin{align*}
d\leq \rk_{\mathbb{F}(x)}\widehat{P}(x)\leq \min\set{n,m}.
\end{align*}
\item[$\impliedby$] Let~$\widetilde{M}_1,\ldots,\widetilde{M}_r$ be all minors of order~$d$ w.r.t.~$\mathbb{F}[x]^{n\times m}$ and put
\begin{align*}
\forall\, i\in\underline{r}: M_i: \mathfrak{F}\to\mathbb{F}[x],\quad P\mapsto \widetilde{M}_i\widehat{P}(x).
\end{align*}
Lemma~\ref{lem:minors_are_polynomials} yields
\begin{align*}
\exists\,h\in\mathbb{N}~\forall\,P\in \mathfrak{F}~\forall\,i\in\underline{r}: \deg M_i(P)\leq h.
\end{align*}
Introducing the mapping
\begin{align*}
\varphi: \set{p\in\mathbb{F}[x]: \deg p\leq h}\to\mathbb{F}^{h+1},\quad p = \sum_{j=0}^{h}p_jx^j\mapsto (p_0,\ldots,p_{h}),
\end{align*}
the functions
\begin{align*}
h_i := \varphi\circ M_i: \mathfrak{F}\to\mathbb{F}^{h+1},\quad i\in\underline{r}
\end{align*}
are well defined polynomial vectors. Since by the Leibniz formula we find
\begin{align*}
\forall\,i\in\underline{r}~\forall\,P\in R: \big(M_i(P)\big)_0 = \det\widetilde{M}_i\begin{bmatrix}P^{1,1}_0 & \cdots & P^{1,q}_0\\\vdots & \ddots & \vdots\\ P^{p,1}_0 & \cdots & P^{p,q}_0\end{bmatrix},
\end{align*}
Proposition~\ref{prop:full_rank_property_block} yields that the entries of each~$h_i$,~$i\in\underline{r}$, are not all identical zero. Hence for any~$i\in\underline{r}$ the preimage of zero under~$h_i$ is a proper algebraic variety by Corollary~\ref{Cor:some_more_varieties} and Proposition~\ref{prop:full_rank_property_block} and thus
\begin{align*}
S = \bigcap_{i=1}^rh_i^{-1}(\set{0})
\end{align*}
is a proper algebraic variety by Lemma~\ref{Rem:Zariski}\,(ii).
\end{enumerate}
\end{proof}

We investigate the properties of the set of polynomial matrices which have full rank on the whole complex plane. For this we need a tool to characterize when polynomials are coprime. This tool is the resultant.

\begin{Definition}[Resultant,~{\cite[p.\,61]{Algebra}}]\label{Def:Resultant}
The \textit{resultant} of two polynomials~$p(x),q(x)\in\mathbb{F}[x]\setminus\set{0_{\mathbb{F}[x]}}$ with~$\deg p = n\geq 0$ and~$\deg q = m\geq 0$ is defined as
\begin{align*}
\Res(p,q) = \det\underbrace{\left[\begin{array}{ccccc|ccccc}
p_0 & & & & & q_0 & & &\\
p_1 & p_0 & & & & q_1 & \cdot &\\
\cdot & \cdot & & & & \cdot & \cdot & \cdot & \\
\cdot & \cdot & \cdot & & & \cdot & \cdot & \cdot & q_0 \\
p_n & p_{n-1} & \cdot & \cdot & & q_n & \cdot & \cdot & q_1 \\
 & p_n & \cdot & \cdot & & \cdot & \cdot & \cdot & \cdot \\
 & & \cdot & \cdot & p_0 & q_{m-1} & \cdot & \cdot & \cdot\\
& & \cdot & \cdot & \cdot & q_{m} & \cdot & \cdot & \cdot & \\
& & \cdot & \cdot & \cdot & & \cdot & \cdot & \cdot \\
& & & \cdot & \cdot & & & \cdot & \cdot\\
& & & & p_n & & & &  q_m
\end{array}\right]}_{\in\mathbb{F}^{(n+m)\times (m+n)}}
\end{align*}
The matrix above is called the \textit{Sylvester matrix} of~$p(x)$ and~$q(x).$ The Sylvester matrix contains~$m$ columns with the coefficients of~$p$ and~$n$ columns with the coefficients of~$q$, so that the Sylvester matrix is in~$\mathbb{F}^{(n+m)\times (m+n)}.$ The entries, which we did not wrote down, are zero. Note that the diagram shows the case~$n<m$.
\end{Definition}

We give some examples to illustrate the resultant of some polynomials.

\begin{Example}
\begin{enumerate}[(i)]
\item Consider the real polynomials
\begin{align*}
p(x) & = x^3 + 5x^2 - 6 & \text{with}~n=3~\text{and}\\
q(x) & = -2x^5 -2x^4 + 6x^3 + 9x\qquad & \text{with}~m=5.
\end{align*}
The Sylvester matrix of~$p(x)$ and~$q(x)$ is
\begin{align*}
\left[\begin{array}{ccccc|ccc}
-6 & 0 & 0 & 0 & 0 & 0 & 0 & 0\\ 
0 & -6 & 0 & 0 & 0 & 9 & 0 & 0\\
5 & 0 & -6 & 0 & 0 & 0 & 9 & 0\\
1 & 5 & 0 & -6 & 0 & 6 & 0 & 9\\
0 & 1 & 5 & 0 & -6 & -2 & 6 & 0\\
0 & 0 & 1 & 5 & 0 & -2 & -2 & 6\\
0 & 0 & 0 & 1 & 5 & 0 & -2 & -2\\
0 & 0 & 0 & 0 & 1 & 0 & 0 & -2
\end{array}\right]
\end{align*}
and the resultant of~$p(x)$ and~$q(x)$ is 
\begin{align*}
\text{Res}(p,q) = 750222.
\end{align*}
\item The Sylvester matrix of~$\widetilde{p}(x) = a\in\mathbb{R}\setminus\set{0}$ and~$q(x)$ as in (i)  is
\begin{align*}
a I_{5} = \begin{bmatrix}
a & 0 & 0 & 0 & 0\\
0 & a & 0 & 0 & 0\\
0 & 0 & a & 0 & 0\\
0 & 0 & 0 & a & 0\\
0 & 0 & 0 & 0 & a
\end{bmatrix}
\end{align*}
and thus~$\text{Res}(\widehat{p},q) = a^5.$
\item The Sylvester matrix of two constant nonzero polynomials~$\widehat{p}(x) = a\in\mathbb{R}\setminus\set{0}$ and~$\widehat{q}(x) = b\in\mathbb{R}\setminus\set{0}$ is the empty matrix. Since the determinant of the emtpy matrix is the empty product,~$\text{Res}(\widehat{p},\widehat{q}) = 1.$
\end{enumerate}
\end{Example}

\begin{Lemma}[Common zeros of polynomials]\label{Lem:commom_zeros}
Let~$p(x),q(x)\in\mathbb{F}[x]\setminus\set{0_{\mathbb{F}[x]}}\subseteq\mathbb{C}[x]\setminus\set{0_{\mathbb{C}[x]}}$ with~$\deg p = n$ and~$\deg q = m$. Then~$p(\cdot),q(\cdot)$ have a common zero, i.e. there is some~$z\in\mathbb{C}$ such that~$p(z) = q(z) = 0$, if, and only if,~$\Res(p,q) = 0.$
\end{Lemma}
\begin{proof}
See e.g. \cite[Theorem 9.3, p.\,24]{AlgCurv} or \cite[Theorem 3.3.1, pp.\,60]{Algebra}.
\end{proof}

\begin{Remark}
If two polynomials~$p,q\in\mathbb{F}[x]$ do not have common zeros, then we call~$p$ and~$q$~\textit{coprime}. The resultant gives a tool to check whether two nonzero polynomials~$p,q$ are coprime.
\end{Remark}


The following proposition is the main result of this section. It is evident that we could only proove this proposition and then immediately obtain Proposition~\ref{prop:full_rank_property_block} and Proposition~\ref{Prop:gen_full_rk_pol_block}. Indeed, if all $g_i$ are zero, then Proposition~\ref{prop:full_rank_property_block} and Proposition~\ref{prop:complex_full_rank_property_pol_mat} coincide. If the set of polynomial block matrices whose rank is ``full'' is generic, then the rank w.r.t. the corresponding field of rational functions is also ``full''. It remains in the case $d = n = m$ to prove that the set of matrices whose determinant does not disappear is generic, which is simple.

{}However, since Proposition~\ref{prop:full_rank_property_block} and Proposition~\ref{Prop:gen_full_rk_pol_block} are much simpler to prove, we showed them first.

\begin{Proposition}\label{prop:complex_full_rank_property_pol_mat}
Let
\begin{align*}
\mathfrak{F} := \left(\mathbb{F}^{n_{1,1}\times m_{1,1}}\right)^{g_{1,1}+1}\times\cdots\times\left(\mathbb{F}^{n_{p,q}\times m_{p,q}}\right)^{g_{p,q}+1}.
\end{align*}
The set
\begin{align*}
S = \set{P\in \mathfrak{F}\left|\forall\,\lambda\in\mathbb{C}:\rk_{\mathbb{C}}\widehat{P}(\lambda)\geq d\right.}
\end{align*}
is generic if, and only if,
\begin{align*}
d\leq\min\set{n,m}\quad\wedge\quad\Big[{}^\neg \big(d = n = m\big)\,\vee \big(\forall\, i\in\underline{p}~\forall\, j\in\underline{q}: g_{i,j} = 0\big)\Big]
\end{align*}
\end{Proposition}
\begin{proof}
Recall that
\begin{align*}
\forall\,P\in\mathfrak{F}~\forall\,\lambda\in\mathbb{C}: \widehat{P}(\lambda) := \left(\widehat{P}(x)\right)(\lambda)\in\mathbb{C}^{n\times m}.
\end{align*}{}
\begin{enumerate}
\item[$\implies$]
Assume first~$d>\min\set{n,m}.$ Then~$S=\emptyset$ and hence not generic. Assume next that
$$d = n = m \quad\wedge\quad \exists\, i\in\underline{p}~\exists\, j\in\underline{q}: g_{i,j} \neq 0.$$
From Lemma~\ref{lem:minors_are_polynomials} we know that
\begin{align*}
\exists\,\alpha\geq 0~\exists\, p_0,\ldots,p_\alpha\in\mathbb{F}[x_1,\ldots,x_{G}]~\forall\, P\in \mathfrak{F}: \det \widehat{P}(x) = \sum_{i = 0}^\alpha p_i(P)x^i
\end{align*}
with~$p_\alpha\not\equiv 0.$ The Leibniz formula and the assumption 
\begin{align*}
\exists\, i\in\underline{p}~\exists\, j\in\underline{q}: g_{i,j} \neq 0
\end{align*}
yield that~$\alpha\geq 1.$ Since~$p_\alpha\not\equiv 0$ we obtain
\begin{align*}
p_\alpha^{-1}(\set{0})\in\mathcal{V}_n^{\text{prop}}(\mathbb{F})
\end{align*}
is a proper algebraic variety. With the fundamental theorem of algebra and Remark~\ref{Rem:InclusionOfGenericSets} we conclude that
\begin{align*}
S^c\supseteq \set{P\in \mathfrak{F}\left\vert \deg\det\widehat{P}(x)\geq 1\right.}\supseteq p_\alpha^{-1}(\set{0})^c
\end{align*}
is a generic set. Corollary~\ref{Cor:PartitionInGenericSets} implies that~$S$ is not generic.
\item[$\impliedby$]
If for any~$i\in\underline{p}$ and~$j\in\underline{q}$ the numbers~$g_{i,j}$ are zero, then genericity of~$S$ follows from Proposition~\ref{prop:full_rank_property_block}. Hence it remains to consider the case~$$\exists\, i\in\underline{p}~\exists\, j\in\underline{q}: g_{i,j} \neq 0\quad\wedge\quad{}^\neg \big(d = n = m\big)\quad\wedge\quad d\leq\min\set{n,m}.$$ We proceed in steps.
\begin{enumerate}[(A)]
\item  Let~$\widetilde{M}_1,\ldots,\widetilde{M}_r$ be all minors of order~$d$ w.r.t.~$\mathbb{F}[x]^{n\times m}$ and define
\begin{align*}
\forall i\in\underline{r}:M_i: \mathfrak{F}\to\mathbb{F}[x],\quad P\mapsto \widetilde{M}_i\left(\widehat{P}(x)\right).
\end{align*}
Note that~$r\geq 2$ since~$d<\max\set{n,m}.$
It is evident that, for any~$P\in\mathfrak{F}$ and any~$\lambda\in\mathbb{C}$, the following implications are true
\begin{align}
\label{property:no_commom_zeros}&~\exists\, i,j\in\underline r: M_i(P),M_j(P)~\text{are~coprime}\\
\implies & \exists\, i,j\in\underline r: \lambda~\text{is~not~a~common~zero~of}~M_i(P),M_j(P)\notag\\
\iff & \exists\, i\in\underline r: M_i(P)(\lambda)\neq 0\notag\\
\iff & \rk_\mathbb{C}\widehat{P}(\lambda)\geq d\notag
\end{align}
and therefore
\begin{align}\label{eq:some_never_again_used_equation}
\widetilde S:= \set{P\in \mathfrak{F}\,\big\vert\,\eqref{property:no_commom_zeros}~\text{holds}}\subseteq S.
\end{align}
\item Define
\begin{align*}
\forall\, i\in\underline{r}:\quad\alpha_i:=\max \set{\deg M_i(P)\,\left|P\in \mathfrak{F}\right.}.
\end{align*}
Then Lemma~\ref{lem:minors_are_polynomials} yields
\begin{align*}
\forall\, i\in\underline{r}~\exists\, M_i^0,\ldots,M_i^{\alpha_i}\in\mathbb{F}[x_1,\ldots,x_G]~\forall\, P\in \mathfrak{F}:\quad M_i(P) = \sum_{j=0}^{\alpha_i} M_i^j(P)x^j
\end{align*}
and, in view of the definition of~$\alpha_i$, we find that
\begin{align*}
\forall\, i\in\underline{r}:M_i^{\alpha_i}\not = 0.
\end{align*}
Since there is some~$(i,j)\in\underline{p}\times\underline{q}$ so that~$g_{i,j}> 0$, there exists some~$k\in\underline{r}$ so that
\begin{align*}
\alpha_k = \max\set{\alpha_i\,\big\vert\,i\in\underline{q}}>0.	
\end{align*}
\item Setting
\begin{align*}
\widehat{S} := \big(M_k^{\alpha_k}\big)^{-1}(\set{0})
\end{align*}
we show that~$\big(\widehat{S}\cap\widetilde{S}\big)^c\subseteq\mathbb{V}$ for some~$\mathbb{V}\in\mathcal{V}_n(\mathbb{F}).$ Define, for~$i\in\underline {r}\setminus\set{k}$, the functional~$q_i$ as the unique continuous extension of
\begin{align*}
\widehat{q}_i: \bigcap_{i = 1}^r \big(M_i^{\alpha_i}\big)^{-1}(\set{0})\to\mathbb{F},\qquad P\mapsto \text{Res}(M_i(P),M_k(P))
\end{align*}
on~$\mathfrak{F}$.
Then the~$q_i(\cdot)$ have the form
\begin{align*}
   q_i(\cdot) := \det\underbrace{\left[\begin{array}{ccccc|ccccc}
M_i^0 & & & & & M_k^0 & & &\\
M_i^1 & M_i^0 & & & & M_k^1 & \cdot &\\
\cdot & \cdot & & & & \cdot & \cdot & \cdot & \\
\cdot & \cdot & \cdot & & & \cdot & \cdot & \cdot & M_k^0 \\
M_i^{\alpha_i} & M_i^{{\alpha_i}-1} & \cdot & \cdot & & \cdot & \cdot & \cdot & M_k^1 \\\hline
 & M_i^{\alpha_i} & \cdot & \cdot & & \cdot & \cdot & \cdot & \cdot \\
 & & \cdot & \cdot & M_i^0 & M_k^{\alpha_k-1} & \cdot & \cdot & \cdot\\
& & \cdot & \cdot & \cdot & M_k^{\alpha_k} & \cdot & \cdot & \cdot & \\
& & \cdot & \cdot & \cdot & & \cdot & \cdot & \cdot \\
& & & \cdot & \cdot & & & \cdot & \cdot\\
& & & & M_i^{\alpha_i} & & & &  M_k^{\alpha_k}
\end{array}\right]}_{\in\mathbb{F}^{(\alpha_i+\alpha_k)\times (\alpha_i+\alpha_k)}}.
\end{align*}
If we expand the determinant with Laplace's formula by the last row, then we find
\begin{align*}
\forall P\in\mathfrak{F}~\text{with}~\deg M_i(P) = \alpha_i: q_i(P) = \pm \left(M_k^{\alpha_k}(P)\right)^{\alpha_k-\deg M_k(P)}\text{Res}(M_i(P),M_k(P))
\end{align*}
and
\begin{align*}
\forall P\in\widehat{S}: q_i(P) = \pm \left(M_i^{\alpha_i}(P)\right)^{\alpha_i-\deg M_i(P)}\text{Res}(M_i(P),M_k(P)).
\end{align*}
Hence we conclude with Lemma~\ref{Lem:commom_zeros}
\begin{align*}
\forall P\in\mathfrak{F}~\forall i\in\underline{r}\setminus\set{k}: q_i(P) =  0\iff\begin{array}{l}
M_i^{\alpha_i}(P) = M_k^{\alpha_k}(P) = 0\\
\vee M_i(P),M_k(P)~\text{not~coprime}.
\end{array}
\end{align*}
This shows that the inclusion
\begin{align}\label{eq:proof_gen_rank_complex}
\big(\widehat{S}\cap\widetilde{S}\big)^c\subseteq \bigcap_{i = 1, i\neq k}^r q_{i}^{-1}(\set{0})\in\mathcal{V}_n(\mathbb{F})
\end{align}
holds true.
\item We show that
\begin{align*}
\bigcap_{i = 1, i\neq k}^r q_{i}^{-1}(\set{0})\in\mathcal{V}^{\text{prop}}_n(\mathbb{F})
\end{align*}
or equivalentely that at least one~$q_{i}\not\equiv 0.$ To do so we construct a matrix~$\widehat H\in\mathbb{F}[x]^{n\times m}$ so that~$\widehat H = \widehat{P}(x)$ with~$q_{i}(P)\neq 0.$ First we note that there are transformations~$S\in\mathcal{GL}_n(\mathbb{F})$ and~$T\in\mathcal{GL}_m(\mathbb{F})$ so that
\begin{align}\label{eq:eine_wichtige_Trafo}
\exists\kappa\in\set{0,1}\forall P\in\mathfrak{F}: M_k(P) = (-1)^{\kappa}\det \big((S\widehat{P}(x)T)_{i,j}\big)_{i,j\in\underline{d}}.
\end{align}
We can choose~$S$ and~$T$ so that~$S$ is the product of finitely many column-switching matrices and~$T$ is the product of finitely many row-switching matrices (for a defintion see~\cite[pp.155]{LinAlg}).
We distinguish two cases.
\begin{enumerate}
\item[$n>m$:] Let, for any~$(i,j)\in\set{(i,j)\in\underline{n}\times\underline{m}\,\big\vert\, i = j\vee i = j+1}$,
\begin{align*}
a_{i,j}\in\mathbb{R}
\end{align*}
be pairwise different real numbers (e.g.~$a_{i,j} = i+j$). Further define
\begin{align*}
\widetilde{g}_{i,j} := \max\set{g\in\mathbb{N}\,\Bigg\vert\,\exists\, P\in\mathfrak{F}: \Big(S\widehat{P}(x)T\Big)_{i,j} = x^g}\in\mathbb{N},
\end{align*}
which is well defined since~$0\in\set{g\in\mathbb{N}\,\Bigg\vert\,\exists\, P\in\mathfrak{F}: \Big(S\widehat{P}(x)T\Big)_{i,j} = x^g}$ and this set is bounded by definition of~$\mathfrak{F}.$\\
Define~$H(x)\in\mathbb{F}[x]^{n\times m}$ via
\begin{align*}
{H(x)}_{i,j} = \begin{cases}
(x-a_{i,j})^{\widetilde{g}_{i,j}}, & i = j\vee i = j+1,\\
0, & \text{else}.
\end{cases}
\end{align*}
Then there exists some~$P'\in\mathfrak{F}$ so that~$H(x) = (S\widehat{P'}(x)T)$ and~$H(x)$ has the form
\begin{align*}
H(x) = \begin{bmatrix}
\star & 0 & 0 & \cdots & 0 & 0\\
\star & \star & 0 & \cdots & 0 & 0\\
0 & \star & \star & \cdots & 0 & 0\\
\vdots & \vdots & \ddots & \ddots & \ddots  & \vdots\\
0 & 0 & 0 & \cdots & \star & 0\\
0 & 0 & 0 & \cdots & \star & \star\\
0 & 0 & 0 & \cdots & 0 & \star\\
0 & 0 & 0 & \cdots & 0 & 0\\
\vdots & \vdots & \ddots & \ddots & \ddots & \vdots\\
0 & 0 & 0 & \cdots & 0 & 0\\
\end{bmatrix}
\end{align*}
We choose the mappings
\begin{align*}
\sigma: \underline{d} & \to\underline{n},\quad i\mapsto i+1~\text{and}\\
\pi: \underline{d} & \to\underline{m},\quad i\mapsto i.\\
\end{align*}
W.l.o.g. we may assume that~$k\neq 1$ and 
\begin{align*}
\widetilde{M}_1: \mathfrak{F}\to\mathbb{F},\quad P\mapsto M_{\sigma,\pi}(S\widehat{P}(x)T).
\end{align*}
Then we find with~\eqref{eq:eine_wichtige_Trafo}
\begin{align*}
M_k(P') = (-1)^{\kappa}\prod_{\ell = 1}^d (x-a_{\ell,\ell})^{\widetilde{g}_{\ell,\ell}}\quad\text{and}\quad M_1(P') = \prod_{\ell = 1}^d (x-a_{\ell+1,\ell})^{\widetilde{g}_{\ell+1,\ell}}
\end{align*}
and~$P'\in\widehat{S}$ by definition of the~$\widetilde{g}_{\ell,\ell}$. Hence~$q_{1}(P')\neq 0.$ 
\item[$n\leq m$:] Let, for any~$(i,j)\in\set{(i,j)\in\underline{n}\times\underline{m}\,\big\vert\, i = j\vee i+1 = j}$,
\begin{align*}
a_{i,j}\in\mathbb{R}
\end{align*}
be pairwise different real numbers (e.g.~$a_{i,j} = i+j$). Further define
\begin{align*}
\widetilde{g}_{i,j} := \max\set{g\in\mathbb{N}\,\Bigg\vert\,\exists\, P\in\mathfrak{F}: \Big(S\widehat{P}(x)T\Big)_{i,j} = x^g}\in\mathbb{N},
\end{align*}
which is well defined since~$0\in\set{g\in\mathbb{N}\,\Bigg\vert\,\exists\, P\in\mathfrak{F}: \Big(S\widehat{P}(x)T\Big)_{i,j} = x^g}$ and this set is bounded by definition of~$\mathfrak{F}.$\\
Define~$H'(x)\in\mathbb{F}[x]^{n\times m}$ via
\begin{align*}
{H'(x)}_{i,j} = \begin{cases}
(x-a_{i,j})^{\widetilde{g}_{i,j}}, & j = i\vee j = i+1,\\
0, & \text{else}.
\end{cases}
\end{align*}
It is evident that there is some~$P^*\in\mathfrak{F}$ so that~$H'(x) = S\widehat{P^*}(x)T$ and that~$H'(x)$ has the form
\begin{align*}
H'(x) = \begin{bmatrix}
\star & \star & 0 & \cdots & 0 & 0 & 0 & \cdots & 0\\
0 & \star & \star & \cdots & 0 & 0 & 0 & \cdots & 0\\
0 & 0 & \star & \cdots & 0 & 0 & 0 & \cdots & 0\\
\vdots & \vdots & \ddots & \ddots & \vdots  & \vdots & \vdots & \ddots & \vdots\\
0 & 0 & 0 & \cdots & \star & \star & 0 & \cdots & 0\\
0 & 0 & 0 & \cdots & 0 & \star & \star & \cdots & 0\\
\end{bmatrix}
\end{align*}
We choose the mappings
\begin{align*}
\sigma: \underline{d} & \to\underline{n},\quad i\mapsto i~\text{and}\\
\pi': \underline{d} & \to\underline{m},\quad i\mapsto i+1.
\end{align*}
W.l.o.g. we may assume that~$k\neq 1$ and 
\begin{align*}
M_1: \mathfrak{F}\to\mathbb{F},\quad P\mapsto M_{\sigma',\pi'}(S\widehat{P}(x)T).
\end{align*} 
We obtain
\begin{align*}
M_k(P^*) = (-1)^{\kappa}\prod_{\ell = 1}^d (x-a_{\ell,\ell})^{\widetilde{g}_{\ell,\ell}}\quad\text{and}\quad M_1(P^*) = \prod_{\ell = 1}^d (x-a_{\ell,\ell+1})^{\widetilde{g}_{\ell,\ell+1}}
\end{align*}
and~$P^*\in\widehat{S}$ by definition of the~$\widetilde{g}_{\ell,\ell}$. Hence~$q_{1}(P^*)\neq 0.$
\end{enumerate}
\item Since not all~$q_{i}$ are zero, the set
\begin{align*}
\bigcap_{i = 1, i\neq k}^r q_{i}^{-1}(\set{0})
\end{align*}
is a proper algebraic variety and hence its complement is a generic sets. By~\eqref{eq:proof_gen_rank_complex} the inclusions
\begin{align*}
S\supseteq\widetilde S\supseteq \widehat S\cap\widetilde S \supseteq \bigcup_{i = 1, i\neq k}^r q_{i}^{-1}(\set{0})^c
\end{align*}
hold true. Remark~\ref{Rem:InclusionOfGenericSets} implies that the set~$S$ is generic. This completes the proof of the proposition.
\end{enumerate}
\end{enumerate}
\end{proof}

We prove that for $\mathbb{F} = \mathbb{R}$ the set of polynomial matrices with full rank on the closed right halfplane is generic. For this we need the well-known Hurwitz criterion.

\begin{Lemma}[{Hurwitz criterion, see~\cite[pp.\,274]{Hurwitz}}]\label{lem:Hurwitz_criterion}
Let~$p(x) = \sum_{i = 0}^n p_ix^i\in\mathbb{R}[x]$ be some real polynomial of degree~$n\in\mathbb{N}^*.$ Then~$p^{-1}(\set{0})\subseteq\overset{\circ}{\mathbb{C}}_{-}$ if, and only if,
\begin{align*}
\forall\, i,j\in\set{0,\ldots,n}: \sign p_i = \sign p_j
\end{align*}
and
\begin{align*}
\forall\, M~\text{leading~minor~of~degree}~d\leq n: \sign M\begin{bmatrix}p_1 & p_3 & p_5 & p_7 & \ldots\\
p_0 & p_2 & p_4 & p_6 & \ldots\\ 0 & p_1 & p_3 & p_5 & \ldots\\ 0 & p_0 & p_2 & p_4 & \ldots\\0 & 0 & p_1 & p_3 & \ldots\\0 & 0 & p_0 & p_2 &  \ldots\\\vdots & \vdots & \vdots & \vdots & \ddots\end{bmatrix} = \sign p_0.
\end{align*}
\end{Lemma}
\begin{proof}
See~\cite[pp.\,274]{Hurwitz}.
\end{proof}

\begin{Proposition}\label{prop:positive_complex_full_rank_property}
Let
\begin{align*}
\mathfrak{R} := \left(\mathbb{R}^{n_{1,1}\times m_{1,1}}\right)^{g_{1,1}+1}\times\cdots\times\left(\mathbb{R}^{n_{p,q}\times m_{p,q}}\right)^{g_{p,q}+1},
\end{align*}
The set
\begin{align*}
\set{P\in \mathfrak{R}\,\Big\vert\,\forall\,\lambda\in\overline{\mathbb{C}}_{+}:\rk_{\mathbb{C}}\widehat{P}(\lambda)\geq d}
\end{align*}
is generic if, and only if,
\begin{align*}
d\leq\min\set{n,m}\quad\wedge\quad\Big[{}^\neg \big(d = n = m\big)\,\vee \big(\forall\, i\in\underline{p}~\forall\, j\in\underline{q}: g_{i,j} = 0\big)\Big].
\end{align*}
\end{Proposition}
\begin{proof}
\begin{enumerate}
\item[$\impliedby$] This follows from Proposition~\ref{prop:complex_full_rank_property_pol_mat} and Remark~\ref{Rem:InclusionOfGenericSets}.
\item[$\implies$] We will seek a contradiction. If~$d>\min\set{n,m}$, then~$\set{P\in \mathfrak{R}\left|\forall\,\lambda\in\overline{\mathbb{C}}_{+}:\rk_{\mathbb{C}}\widehat{P}(\lambda)\geq d\right.} = \emptyset.$ It remains to consider the case
\begin{align*}
d = n = m\quad\wedge\quad\exists\, i\in\underline{p}~\exists\, j\in\underline{q}: g_{i,j} \neq 0.
\end{align*}
From Lemma~\ref{lem:minors_are_polynomials} we know that
\begin{align*}
\exists\, \alpha\in\mathbb{N}~\exists\, p_1,\ldots,p_\alpha\in\mathbb{R}[x_1,\ldots,x_G]~\forall\, P\in \mathfrak{R}: \det\widehat{P}(x) = \sum_{i = 0}^\alpha p_i(P)x^i
\end{align*}
with~$\alpha\geq 0$ w.l.o.g. minimal. The Leibniz formula and the assumption 
\begin{align*}
\exists\, i\in\underline{p}~\exists\, j\in\underline{q}: g_{i,j} \neq 0
\end{align*}
yield that~$\alpha\geq 1.$

We construct some~$P=\Big(\big(P_0^{1,1},\ldots,P^{1,1}_{g_{1,1}}\big),\ldots,\big(P_0^{p,q},\ldots,P^{p,q}_{g_{p,q}}\big)\Big)\in\mathfrak{R}$ so that~$\sign p_0(P) = \sign \big(-p_\ell(P)\big)\neq 0$ for some~$\ell\in\underline{\alpha}.$ Recall that~$P$ is associated to~$\widehat{P}(x)$ in~\eqref{eq:HutOperator}.

Let~$(i,j)\in\underline{p}\times\underline{q}$ such that~$g_{i,j}>0$ and set
\begin{align*}
\forall\,(a,b)\in\underline{p}\times\underline{q}\setminus\set{(i,j)}~\forall\,c\in\underline{(g_{a,b}+1)}: P_c^{a,b} = 0_{n_{a,b}\times m_{a,b}}.
\end{align*}
This means that~$x$ only appears in the block~$(i,j)$ in~$\widehat{P}(x).$
Let~$(\iota,\kappa)\in\underline{\big(n_{i,j}\big)}\times\underline{\big(m_{i,j}\big)}$ be arbitrary and choose
\begin{align*}
\Big(P_{g_{i,j}}^{i,j}\Big)_{r,s} = \begin{cases}
-1, & (r,s) = (\iota,\kappa)\\
0, & else
\end{cases}
\end{align*}
and
\begin{align*}
P_1^{i,j} = \cdots = P_{g_{i,j}-1}^{i,j} = 0_{n_{i,j}\times m_{i,j}}.
\end{align*}
Further we choose
\begin{align*}
\forall\,(a,b)\in\underline{p}\times\underline{q}\setminus\set{(i,j)}: P_0^{a,b}\in\mathbb{R}^{n_{a,b}\times m_{a,b}}
\end{align*}
and~$P_0^{i,j}\in\mathbb{R}^{n_{i,j}\times m_{i,j}}$ such that
\begin{align*}
& \widehat{P}(x)_{k,m} = \begin{cases} 1, & k = m\quad\wedge\quad k\neq \sum\limits_{y = 1}^{i-1}n_{i,1} + \iota\quad\wedge\quad m\neq \sum\limits_{y = 1}^{j-1}m_{1,y} + \kappa,\\
2, & m = \sum\limits_{y = 1}^{i-1}n_{y,1} + \iota\quad\wedge\quad k = \sum\limits_{y = 1}^{j-1}m_{1,y} + \kappa \quad\wedge\quad k\neq m,\\
1-x^{g_{i,j}}, & k = \sum\limits_{y = 1}^{i-1}n_{y,1} + \iota\quad\wedge\quad m = \sum\limits_{y = 1}^{j-1}m_{1,y} + \kappa,\\
0, & \text{else}.\end{cases}
\end{align*}
Then~$\widehat{P}(x)$ has the structure
\begin{align*}
\widehat{P}(x) = \begin{bmatrix}
I\\
& 0 & & & &  1-x^{g_{i,j}}\\
& & 1 & &\\
& & & \ddots\\
& & & & 1\\
& 2 & & & & 0\\
& & & & & & I\\
\end{bmatrix}.
\end{align*}
With this struture it is evident that
\begin{align*}
\exists k\in\set{1,0}: \frac{1}{2}\det\widehat{P}(x) = (-1)^k(1-x^{g_{i,j}}) \implies p_0(P) = (-1)^{k}\cdot 2 = -p_{g_{i,j}}(P).
\end{align*}
Continuity of all~$p_i$ yields that
\begin{align*}
\exists\,\varepsilon>0~\forall\,P'\in\mathbb{B}(P,\varepsilon):p_0(P')\neq 0\quad\wedge\quad p_{g_{i,j}}(P')\neq 0\quad\wedge\quad \sign p_0(P') = \sign \big(-p_{g_{i,j}}(P')\big)
\end{align*}
where the ball is w.r.t. some norm on~$\mathfrak{R}$. By Lemma~\ref{lem:Hurwitz_criterion}, the complement of the set
\begin{align*}
S := \set{P\in \mathfrak{R}\left|\forall\,\lambda\in\overline{\mathbb{C}}_{+}:\rk_{\mathbb{C}}\widehat{P}(\lambda)\geq d\right.} = \set{P\in \mathfrak{R}\,\left|\,\left(\det\widehat{P}(x)\right)^{-1}(\set{0})\subseteq\overset{\circ}{\mathbb{C}}_-\right.}
\end{align*}
contains an inner point and hence $S$ is not generic.
\end{enumerate}
\end{proof}


\section{Controllability of ordinary differential systems}

As a motivation for the study of differential-algebraic systems we first study controllability of linear systems described by ordinary differential equations of the form
\begin{align}\label{Def_eq:ODE}
\dot{x} = Ax+Bu,\quad (A,B)\in\Sigma_{n,m},~u\in\mathcal{L}^1_{\text{loc}}(\mathbb{R},\mathbb{R}^m).
\end{align}
First we recall the definition and the existence and uniqueness of a solution.

\begin{Definition-Proposition}[Solution of~{\eqref{Def_eq:ODE}, see \cite[p.\,40]{Ryan}}]
For any control $u\in\mathcal{L}^1_{\text{loc}}(\mathbb{R},\mathbb{R}^m)$ and any initial value $x^0\in\mathbb{R}^n$ there exists a unique function $x(\cdot) = x(\cdot;x^0,u(\cdot))\in AC(\mathbb{R},\mathbb{R}^n)$ so that
\begin{align*}
\dot{x}(t) = Ax(t) + Bu(t)~\text{a.e.}
\end{align*}
This solution~$x(\cdot;x^0,u(\cdot))$ is given by the variation of constants formula (see~\cite[Theorem VII.1.17, p.\,137]{AE_II})
\begin{align}
x(\cdot;x^0,u(\cdot)): [0,\infty)\to\mathbb{F}^n, t\mapsto \exp{At}x^0 + \int_0^t\!\exp{A(t-\tau)}Bu(\tau)\,\mathrm{d}\tau.
\end{align}
\end{Definition-Proposition}

Now that we know what a solution is, we recall the definition of controllability of linear ODEs.

\begin{Definition}[Controllability~{\cite[Definition 3.3, p.\,67]{Ryan}}]
Let~$\widehat{x},x_0\in\mathbb{R}^n$ be arbitrary.~$\widehat{x}$ is called \textit{reachable from~$x_0$}, if
\begin{align*}
\exists\, T>0~\exists\, u\in\mathcal{L}^1_{\text{loc}}(\mathbb{R},\mathbb{R}^m): x(T;x_0,u) = \widehat{x}.
\end{align*}
The system~\eqref{Def_eq:ODE} is called \textit{controllable}, if
\begin{align*}
\forall\, \widehat{x},x_0\in\mathbb{R}^n~\exists\, T>0~\exists\, u\in\mathcal{L}^1_{\text{loc}}(\mathbb{R},\mathbb{R}^m): x(T;x_0,u) = \widehat{x}.
\end{align*}
The set of all controllable matrix pairs is called 
\begin{align*}
\Sigma_{n,m}^{\text{cont}}:=\set{(A,B)\in\Sigma_{n,m}\,\big\vert\,\eqref{Def_eq:ODE}~\text{is~controllable}}.
\end{align*}
\end{Definition}

There are two widely known criteria for $(A,B)$ to be in $\Sigma_{n,m}^{\text{cont}},$ the \textit{Kalman criterion} and the \textit{Hautus criterion}. We will make use of the Kalman criterion.

\begin{Lemma}[Kalman criterion,\,{\cite[Corollary 3.4,\,pp.\,40]{ContTheorLinSyst}}]~\label{lem:Kalmam_rank_condition}
The system~\eqref{Def_eq:ODE} is controllable if, and only if,
\begin{align*}
\rk[B,AB,A^2B,\ldots,A^{n-1}B] = n.
\end{align*}
\end{Lemma}

The following proposition is well known and is e.g. proven in~\cite[Theorem 1.3, p.\,44]{Wonham}. As an example we repeat the proof.

\begin{Proposition}\label{Prop:controllability}
$\Sigma_{n,m}^{\text{cont}}$ is the complement of a proper algebraic variety.
\end{Proposition}
\begin{proof}
From the definition of matrix multiplication we obtain
\begin{align}\label{eq:important_polys}
\forall\,i\in\underline{n}~\forall\,j\in\underline{nm}~\exists\, p_{i,j}\in\mathbb{R}[x_1,\ldots,x_{(n^2+nm)}]: [B,AB,A^2B,\ldots,A^{n-1}B]_{i,j} = p_{i,j}(A,B).
\end{align}
Let~$\widetilde{M}_1,\ldots,\widetilde{M}_q$ be all minors of order~$n$ w.r.t.~$\mathbb{R}^{n\times nm}$ and define the mappings
\begin{align*}
\forall i\in\underline{q}: M_i: \Sigma_{n,m}\to\mathbb{F},\quad (A,B)\mapsto \widetilde{M}_i[B,Ab,\ldots,A^{n-1}B]
\end{align*}
From Lemma~\ref{lem:minors_are_polynomials} and~\eqref{eq:important_polys} we obtain that each~$M_i$ is a polynomial in the entries of~$A$ and~$B$ and by Lemma~\ref{lem:Kalmam_rank_condition} the set
\begin{align*}
\left(\Sigma_{n,m}^{\text{cont}}\right)^c = \bigcap_{i=1}^q M_i^{-1}(\set{0})
\end{align*}
is an algebraic variety; it is proper if not all the mappings~$M_i$ are zero. Consider the fixed matrices
\begin{align*}
A := \begin{bmatrix}
0 & 0 & \ldots & 0 & 1\\
1 & 0 & \ldots & 0 & 0\\
0 & \ddots & \ddots & \vdots & \vdots\\
0 & \ldots & 1 & 0 & 0\\
0 & \ldots & 0 & 1 & 0
\end{bmatrix}\in\mathbb{R}^{n\times n},\quad B := [e_1,0_{n\times (m-1)}].
\end{align*}
Then~$[B,AB,\ldots,A^{n-1}B] = [e_1,0,e_2,0,\ldots,e_n,0]$ and
$$\exists\, i\in\underline{q}:M_i([B,AB,\ldots,A^{n-1}B]) = 1.$$
Hence not all the mappings~$M_i$ are zero and the constructed algebraic variety is proper. This completes the proof of the proposition.
\end{proof}

\begin{Remark}
\begin{enumerate}[(i)]
\item Proposition~\ref{Prop:controllability} and Proposition~\ref{alg_var_nullset} imply the well known fact that~$\Sigma_{n,m}^{\text{cont}}$ is open and dense (see~\cite[Theorem 2.6, p.\,37]{IntOptCont})
\item A similar result was shown by~\cite[Theorem 9.5.2, p.\,324]{Polderman} for the pole placement of controllable matrix pairs.
\item Proposition~\ref{Prop:controllability} can be proved by applying the proof of Proposition~\ref{prop:complex_full_rank_property_pol_mat} to the Hautus criterion (see \cite[Theorem 3.11, p.\,79]{Ryan}). However, since this method is more technical than the usage of the Kalman criterion, we don't write it down.
\item It is notable that the Kalman criterion can be applied to systems with complex coefficients (see\,\cite[Corollary 3.4,\,pp.\,40]{ContTheorLinSyst}). This implies that the restriction to systems with real coefficients was artificial and the set 
\begin{align*}
\set{(A,B)\in\mathbb{C}^{n\times n}\times\mathbb{C}^{n\times m}\,\big\vert\,\eqref{Def_eq:ODE}~\text{is~controllable}}\subseteq\mathbb{C}^{n\times n}\times\mathbb{C}^{n\times m}
\end{align*}
is also generic.
\end{enumerate}
\end{Remark}

\newpage
\section{Controllability of differential-algebraic systems}

Controllability and observabillity of a DAE system of the form
\begin{align}\label{Def_eq:DAE}
\tfrac{\mathrm{d}}{\mathrm{d}t}(Ex) = Ax + Bu,\quad (E,A,B)\in\Sigma_{\ell,n,m}.
\end{align}
can be characterised by algebraic properties of the matrix triple~$(E,A,B)\in\Sigma_{\ell,n,m}.$ We investigate whether these properties are generic and start with controllability. We first define what a solution of the system~\eqref{Def_eq:DAE} is:

\begin{Definition}[Solution of~\eqref{Def_eq:DAE}]
A mapping~$(x,u): \mathbb{R}\to\mathbb{R}^n\times\mathbb{R}^m$ is called a \textit{solution}, if
\begin{align*}
& x \in\mathcal{L}^1_{\text{loc}}(\mathbb{R},\mathbb{R}^\ell ),\,Ex\in\mathcal{AC}(\mathbb{R},\mathbb{R}^n),\,u \in\mathcal{L}^1_{\text{loc}}(\mathbb{R},\mathbb{R}^m),\\
& \exists\, N\subseteq\mathbb{R}~\text{nullset}~\forall\, t\in\mathbb{R}\setminus N: \diff{}{t}(Ex)(t) = Ax(t) + Bu(t).
\end{align*}
The set of all solutions is called the \textit{behaviour} of~\eqref{Def_eq:DAE} and denoted by~$\mathfrak{B}_{[E,A,B]}.$
\end{Definition}

If we turn our attention to linear DAEs, we find several controllability concepts:
\begin{itemize}
\item freely initializability,
\item impulse controllability,
\item completely controllability,
\item strongly controllability and
\item controllability in the behavioural sense.
\end{itemize}
We want to investigate whether and in which cases regarding~$\ell,n,m$ these controllability concepts are generic on~$\Sigma_{\ell,n,m}.$ Since we are solely interested in DAEs with real coefficients, let~$\mathbb{F} = \mathbb{R}$ for the remainder of this chapter.

Our restriction to $\mathbb{F} = \mathbb{R}$ origins in the restriciton of Berger and Reis to such systems in~\cite{DAEs}. In this survey they collect algebraic criteria for all of the controllability concepts which we consider. 

\subsection{Freely initializability}

The first controllability concept is the concept of freely initializable systems. First we recall the definition.

\begin{Definition}[Freely initializability~{\cite[Definition 2.1(a)]{DAEs}}]\label{Def:freely_initializable}
$(E,A,B)\in\Sigma_{\ell,n,m}$ is \textit{freely initializable}, if
\begin{align*}
\forall\, x_0\in\mathbb{R}^n~\exists\, (x,u)\in\mathfrak{B}_{[E,A,B]}: x(0) = x_0.
\end{align*}
\end{Definition}

We briefly discuss the terminology.

\begin{Remark}
In many books freely initializable systems are called \textit{controllable at infinity} (see e.g.~\cite{DAEs}). However, the terminology controllable at infinity is confusing. Some authors (e.g.~\cite{MatPenc}) suggest the more intuitive term \textit{freely initializable}. We adopt the latter.
\end{Remark}

To prove that the set of freely initializable systems is generic, we need an algebraic criterion.

\begin{Lemma}[Algebraic criterion for freely initializability~{\cite[p.\,32]{DAEs}}]\label{lem:alg_crit_free_init}
The system~\eqref{Def_eq:DAE} is freely initializable if, and only if,
\begin{align*}
\rk[E,B] = \rk[E,A,B].
\end{align*}
\end{Lemma}

We can use this criterion and the results of Section 3 to prove that the set of freely initializable systems is generic.

\begin{Proposition}\label{prop:freely_initializable}
The set
\begin{align*}
S = \set{(E,A,B)\in\Sigma_{\ell,n,m}\,\big\vert\,\eqref{Def_eq:DAE}~\text{freely~initializable}}
\end{align*}
is generic if, and only if,~$\ell\leq n+m.$
\end{Proposition}
\begin{proof}
Lemma~\ref{lem:alg_crit_free_init} yields 
\begin{align*}
S = \set{(E,A,B)\in\Sigma_{\ell,n,m}\,\big\vert\,\rk[E,A,B] = \rk [E,B]}
\end{align*}
\begin{itemize}
\item[$\implies$] Let~$\ell>n+m.$ We show that~$S$ is a Lebesgue nullset and thus by Corollary~\ref{Cor:GenSetsInfiniteMeasure} not generic.

Define the sets
\begin{align*}
S_1 := \set{(E,A,B)\in\Sigma_{\ell,n,m}\,\big\vert\,\rk[E,B] = n+m}
\end{align*}
and
\begin{align*}
S_2 := \set{(E,A,B)\in\Sigma_{\ell,n,m}\,\big\vert\,\rk[E,A,B] \geq \min\set{\ell, 2n+m}}.
\end{align*}
With Proposition~\ref{prop:full_rank_property_block} it is clear that~$S_2$ is generic. With~$\ell>n+m$ we find the isomorphy 
\begin{align*}
S_1 \cong \set{(E,B)\in\mathbb{R}^{\ell\times n}\times\mathbb{R}^{\ell\times m}\,\big\vert\,\rk[E,B] \geq n+m}\times\mathbb{R}^{\ell\times n}
\end{align*}
and Proposition~\ref{prop:full_rank_property_block}, Lemma~\ref{lem:product_of_generic_sets} and Lemma~\ref{Lem:transformgenericsets} yield that~$S_1$ is also a generic set. Hence~$S_1\cap S_2$ is a generic set by Lemma~\ref{cor:intersection_union}. Since~$\ell>n+m$ we conclude
\begin{align*}
\forall (E,A,B)\in S_1\cap S_2: \rk[E,A,B]\neq\rk[E,B].
\end{align*}
Corollary~\ref{cor:intersection_union1} yields that~$(S_1\cap S_2)^c$ is a Lebesgue nullset. Since the Lebesgue measure is complete,~$S\subseteq(S_1\cap S_2)^c$ implies that~$S$ is also a Lebesgue nullset.
\item[$\impliedby$] Let~$\ell\leq n+m.$ Then for any~$(E,A,B)\in\Sigma_{\ell,n,m}$ we have~$\rk[E,A,B]\leq\ell$ as well as~$\rk[E,B]\leq\ell$ and equality can be achieved. By Proposition~\ref{prop:full_rank_property_block}, the sets~$$\widetilde S_1 := \set{(E,A,B)\in\Sigma_{\ell,n,m}\,\big\vert\,\rk[E,B] = \ell}$$ and~$$\widetilde S_2 := \set{(E,A,B)\in\Sigma_{\ell,n,m}\,\big\vert\,\rk[E,A,B] = \ell}$$ are generic sets. By Lemma~\ref{cor:intersection_union},
\begin{align*}
\widetilde{S}_1\cap\widetilde{S}_2
\end{align*}
is a generic set. We conclude with Remark~\ref{Rem:InclusionOfGenericSets} that
\begin{align*}
\set{(E,A,B)\in\Sigma_{\ell,n,m}\,\big\vert\,\rk[E,A,B] = \rk[E,B]}\supset \widetilde{S}_1\cap\widetilde{S}_2
\end{align*}
is a generic set.
\end{itemize}
\end{proof}

\subsection{Impulse controllability}

The second controllability concept that we investigate is the concept of impulse controllable systems. This was also investigated by Belur and Shankar in~\cite{Belur}. They consider the class of polynomial DAEs
\begin{align}\label{eq:intro_2}
M(\tfrac{\mathrm{d}}{\mathrm{d}t}) y = 0\qquad\text{for}~M(\tfrac{\mathrm{d}}{\mathrm{d}t})\in\mathbb{C}\left[\tfrac{\mathrm{d}}{\mathrm{d}t}\right]^{\ell\times q}~\text{and}~\ell\leq q.
\end{align}
Note that~\eqref{Def_eq:DAE} is equivalent to
\begin{align*}
[\tfrac{\mathrm{d}}{\mathrm{d}t}E-A,-B]\begin{bmatrix}
x\\u
\end{bmatrix} = 0
\end{align*}
and hence the systems of the form~\eqref{Def_eq:DAE} form a subclass of~\eqref{eq:intro_2}. If we restrict our attention to the special case
\begin{align}\label{eq:intro_6}
M(\tfrac{\mathrm{d}}{\mathrm{d}t}) = \left[\tfrac{\mathrm{d}}{\mathrm{d}t}E-A,-B\right],\quad (E,A,B)\in\Sigma_{\ell,n,m}, \ell\leq n+m,
\end{align}
then we can compare our results with those of Belur and Shankar. First we give the definition of impulse controllable systems. 

\begin{Definition}[Impulse controllability~{\cite[Definition 2.1(b), p.\,9]{DAEs}}]\label{Def:Impulse_controllable}
The system~\eqref{Def_eq:DAE} is \textit{impulse controllable} if, and only if,
\begin{align*}
\forall\, x_0\in\mathbb{R}^n~\exists\, (x,u)\in\mathfrak{B}_{[E,A,B]}: Ex_0 = Ex(0).
\end{align*}
\end{Definition}

For our next investigations we need an algebraic criterion for impulse controllability.

\begin{Lemma}[Algebraic criterion for impulse controllability~{\cite[Corollary 4.3, p.\,32]{DAEs}}]\label{lem:impulse_controllability}
The system~\eqref{Def_eq:DAE} is impulse controllability if, and only if,
\begin{align}\label{Def:impulse_controllability}
\forall\, Z\in\mathbb{R}^{n\times n-\rk E}~\text{with}~\im_{\mathbb{R}} Z = \ker_{\mathbb{R}}E: \rk[E,A,B] = \rk[E,AZ,B].
\end{align}
\end{Lemma}

Belur and Shankar use a different criterion for impulse controllability (see~\cite[Theorem 2.1, p.\,492]{Belur}) than we do. Their criterion can be applied to real systems (see~\cite[Theorem 5.5, p.\,2468]{Belur2}) and hence their results regarding genericity hold true, if we consider systems with real coefficients instead of complex coefficients. 

In~\cite[Theorem 4.1, p.\,496]{Belur} Belur and Shankar proved that the set of impulse controllable systems of the form~\eqref{eq:intro_2} is generic. This result can not be reached by studying the systems of the form~\eqref{eq:intro_1} and their corresponding polynomial representation~\eqref{eq:intro_6}.

With the help of Lemma~\ref{lem:impulse_controllability} we formulate our main result regarding the case of systems of the form~\eqref{Def_eq:DAE}.

\begin{Proposition}\label{prop:impulse_controllable}
The set
$S = \set{(E,A,B)\in\Sigma_{\ell,n,m}\,\big\vert\,\eqref{Def_eq:DAE}~\text{impulse~controllable}}$
is generic if, and only if,~$\ell\leq n+m.$
\end{Proposition}
\begin{proof}
We consider the two cases~$\ell\geq n$ and~$\ell<n.$
\begin{enumerate}
\item[$\ell\geq n$:] Let~$d:=\min\set{\ell,2n+m}.$ By Proposition~\ref{prop:full_rank_property_block} the set 
\begin{align*}
S_1 := \set{(E,A,B)\in\Sigma_{\ell,n,m}\,\big\vert\,\rk [E,A,B] = d}
\end{align*}
is generic. Applying Proposition~\ref{prop:full_rank_property_block} again yields that~$\set{E\in\mathbb{R}^{\ell\times n}\,\big\vert\,\rk E = n}$ is generic and therefore Lemma~\ref{lem:product_of_generic_sets} gives that the set
\begin{align*}
S_2 := \set{(E,A,B)\in\Sigma_{\ell,n,m}\,\big\vert\,\rk E = n} = \set{E\in\mathbb{R}^{\ell\times n}\,\big\vert\,\rk E\geq n}\times \Sigma_{n,m}
\end{align*}
is generic. With Corollary~\ref{Cor:PartitionInGenericSets}, Lemma~\ref{cor:intersection_union} and Remark~\ref{Rem:InclusionOfGenericSets} we find that~$S$ is a generic set if, and only if,~$S\cap S_1\cap S_2$ is generic. For any~$(E,A,B)\in S_2$ we find by the rank-nullity Theorem that~$\ker E = \set{0}.$ Hence the condition~\eqref{Def:impulse_controllability} is, for any~$(E,A,B)\in S_2$, equivalent to 
\begin{align*}
\rk[E,A,B] = \rk[E,B]
\end{align*} 
and by Lemma~\ref{lem:impulse_controllability}, we find
\begin{align*}
S\cap S_1\cap S_2 = \set{(E,A,B)\in\Sigma_{n,n,m}\,\big\vert\,\rk [E,B] = d}.
\end{align*}
It is evident that~$S\cap S_1\cap S_2 = \emptyset$, if~$d = 2n+m.$ If~$d = \ell$, then Proposition~\ref{prop:full_rank_property_block} and Lemma~\ref{lem:product_of_generic_sets} imply that~$S\cap S_1\cap S_2$ is generic if, and only if,~$\ell\leq n+m.$ This proves the first case.
\item[$\ell < n$:] Put~$\widetilde{S} := \set{(E,A,B)\in\Sigma_{n,n,m}\,\big\vert\,\rk[E,B] = \ell = \min\set{\ell,n+m}}.$ With Proposition~\ref{prop:full_rank_property_block}, Lemma~\ref{lem:product_of_generic_sets}, Lemma~\ref{Lem:transformgenericsets} and the congruence 
\begin{align*}
\widetilde{S}\cong \set{(E,B)\in\Sigma_{n,m}\,\big\vert\,\rk[E,B] = \ell}\times \mathbb{R}^{n\times n}
\end{align*}
we find that~$\widetilde{S}$ is a generic set. It is evident that the properties
\begin{align*}
\forall\,(E,A,B)\in\Sigma_{\ell,n,m}~\forall\, Z\in\mathbb{R}^{n\times n-\rk E}~\text{with}~\im_{\mathbb{R}} Z = \ker_{\mathbb{R}}E: \rk[E,AZ,B]\leq \ell
\end{align*}
and
\begin{align*}
\forall\,(E,A,B)\in\widetilde{S}~\forall\, Z\in\mathbb{R}^{n\times n-\rk E}~\text{with}~\im_{\mathbb{R}} Z = \ker_{\mathbb{R}}E: \rk[E,AZ,B]=\ell
\end{align*}
hold true. By Lemma~\ref{lem:impulse_controllability}, we find that~$\widetilde{S}\subseteq S$ and hence, by Remark~\ref{Rem:InclusionOfGenericSets},~$S$ is generic.
\end{enumerate}
\end{proof}


The class of systems of the form~\eqref{eq:intro_6} is treated in~\cite[Theorem 4.2, p.\,499]{Belur}. Belur and Shankar proved that the set of impulse controllable systems~\eqref{eq:intro_6} is generic. Moreover, they restrict the degree of the polynomials or put certain entries of~$E,\,A$ and~$B$ to zero resp., and find a criterion on the degrees of the entries so that the set of impulse controllable restricted systems is generic. We do not consider such a restriction of the entries of~$E,\,A$ and~$B$ and hence the result of Belur and Shankar is more general. Instead we consider systems of the form~\eqref{eq:intro_6} with real coefficients and~$\ell>n+m$ and their matrix representation~$(E,A,B)\in\Sigma_{\ell,n,m}$. We find that the set of such impulse controllable systems is not generic. This is a corollary of~\cite[Theorem 4.2]{Belur}, if we add~$\ell-(n+m)$ rows of zeros to~$E,\,A$ and~$B$.

It is noteable that Proposition~\ref{prop:impulse_controllable} can be derived from~\cite{Belur} while a generalization to the case of systems with complex coefficients might require some more work.

Belur and Shankar also consider a generalization of~\cite[Theorem 4.2]{Belur} to systems of the form~\eqref{eq:intro_2} with
\begin{align*}
\forall\, i\in\underline{\ell}~\forall\,j\in\underline{q}: \deg M(\tfrac{\mathrm{d}}{\mathrm{d}t})\leq c_{i,j},\quad c_{i,j}\in\mathbb{N}
\end{align*}
in~\cite[Theorem 4.3, pp.\,499]{Belur}. We do not consider such systems.

\subsection{Completely controllable}

The third controllability concept is the concept of completely controllable systems. This concept is defined as follows.

\begin{Definition}[Completely controllable~{\cite[Definition 2.1(g), p.\,9]{DAEs}}]\label{Def:Completely_controllable}
The system~\eqref{Def_eq:DAE} is \textit{completely controllable} if, and only if,
\begin{align*}
\exists\, T>0~\forall\, x_0,x_T\in\mathbb{R}^n~\exists\, (x,u)\in\mathfrak{B}_{[E,A,B]}:x(0) = x_0\wedge x(T) = x_T
\end{align*}
\end{Definition}

As for all of the controllability concepts considered in this section, there is an algebraic criterion for completely controllable systems.

\begin{Lemma}[Algebraic criterion for completely controllability~{\cite[Corollary 4.3, p.\,32]{DAEs}}]\label{lem:compStab}
The system~\eqref{Def_eq:DAE} is completely controllable if, and only if,
\begin{align*}
\forall\,\lambda\in\mathbb{C}: \rk [E,A,B] = \rk [E,B] = \rk[\lambda E-A,B].
\end{align*}
\end{Lemma}

We can use this criterion to prove genericity.

\begin{Proposition}\label{prop:comp_cont} The set
$S = \set{(E,A,B)\in\Sigma_{\ell,n,m}\,\big\vert\, \eqref{Def_eq:DAE}~\text{completely~controllable}}$ is generic if, and only if,~$\ell<n+m.$
\end{Proposition}
\begin{proof}
\begin{enumerate}
\item[~$\impliedby$] By Proposition~\ref{prop:full_rank_property_block} and Corollary~\ref{cor:intersection_union1}, the set
\begin{align*}
S_1 := \set{(E,A,B)\in\Sigma_{\ell,n,m}\,\big\vert\,\rk[E,A,B] = \rk[E,B] = \ell}
\end{align*}
is generic. From Proposition~\ref{prop:complex_full_rank_property_pol_mat} we know that
\begin{align*}
S_2 := \set{(E,A,B)\in\Sigma_{\ell,n,m}\,\big\vert\,\forall\lambda\in\mathbb{C}:\rk[\lambda E-A,B] = \ell}
\end{align*}
is generic. Corollary~\ref{cor:intersection_union1} yields that~$S_1\cap S_2$ is generic and with~$S_1\cap S_2\subseteq S$ we find that~$S$ is generic.
\item[$\implies$] 
By Lemma~\ref{lem:compStab} and Lemma~\ref{lem:alg_crit_free_init}, the inclusion
\begin{align*}
S\subseteq\set{(E,A,B)\in\Sigma_{\ell,n,m}\,\big\vert\, (E,A,B)~\text{freely~initializable}}=:S'
\end{align*}
holds true. Hence genericity of~$S$ implies with Remark~\ref{Rem:InclusionOfGenericSets} that~$S'	$ is generic. Thus Proposition~\ref{prop:freely_initializable} yields that~$\ell\leq n+m.$

Let~$\ell = n+m.$ We show that~$S$ is not generic. In Proposition~\ref{prop:freely_initializable} we have shown that the set 
\begin{align*}
\set{(E,A,B)\in\Sigma_{\ell,n,m}\,\big\vert\,\rk[E,A,B] = \rk[E,B] = n+m}
\end{align*}
is generic. Thus Corollary~\ref{Cor:PartitionInGenericSets} yields that a necessary condition for genericity of~$S$ is genericity of
\begin{align*}
\set{(E,A,B)\in\Sigma_{\ell,n,m}\,\big\vert\,\forall\lambda\in\mathbb{C}:\rk[\lambda E-A,B] = n+m}.
\end{align*}
By Proposition~\ref{prop:complex_full_rank_property_pol_mat}, this is not the case. Hence~$S$ is not generic if~$\ell = n+m.$
\end{enumerate}
\end{proof}

\subsection{Controllability in the behavioural sense}

The fourth controllability concept is the concept of controllability in the behavioural sense. 

\begin{Definition}[Controllability in the behavioural sense~{\cite[Definition 2.1(b), p.\,9]{DAEs}}]\label{Def:cont_in_the_beh_sense}
The system~\eqref{Def_eq:DAE} is called \textit{controllable in the behavioural sense} if, and only if,
\begin{align*}
\forall\, (x_1,u_1),(x_2,u_2)\in\mathfrak{B}_{[E,A,B]}~\exists\, T>0~\exists\, (x,u)\in\mathfrak{B}_{[E,A,B]}: (x,u)(t) = \begin{cases}(x_1,u_1)(t), & t<0\\(x_2,u_2)(t), & t>T\end{cases}.
\end{align*}
\end{Definition}

We need an algebraic criterion to proceed.

\begin{Lemma}[Algebraic criterion for controllability in the behavioural sense~{\cite[Corollary 4.3, p.\,32]{DAEs}}]\label{lem:AlgCritContBehavSense}
The system~\eqref{Def_eq:DAE} is controllable in the behavioural sense if, and only if,
\begin{align*}
\forall\, \lambda\in\mathbb{C}: \rk_{\mathbb{R}(x)}[xE-A,B] = \rk_{\mathbb{C}}[\lambda E-A,B].
\end{align*}
\end{Lemma}

With this criterion we can prove genericity of the set of in the behavioural sense controllable systems with real coefficients.

\begin{Proposition}\label{Prop:cont_im_the_beh_sense_is_generic}
The set~$S = \set{(E,A,B)\in\Sigma_{\ell,n,m}\,\big\vert\,\eqref{Def_eq:DAE}~\text{controllable~in~the~behavioural~sense}}$
is generic if, and only if,~$\ell\neq n+m.$
\end{Proposition}
\begin{proof}
Set~$d:=\min\set{\ell,n+m}.$ From Proposition~\ref{Prop:gen_full_rk_pol_block} we find that the set
\begin{align*}
S_1 := \set{(E,A,B)\in\Sigma_{\ell,n,m}\,\big\vert\,\rk_{\mathbb{R}(x)}[xE-A,B] = d}
\end{align*}
is a generic set. From Lemma~\ref{lem:AlgCritContBehavSense} we find the equality
\begin{align*}
S\cap S_1 = \set{(E,A,B)\in\Sigma_{\ell,n,m}\,\big\vert\,\forall\lambda\in\mathbb{C}:\rk_{\mathbb{C}}[\lambda E-A,B] = d}.
\end{align*}
By Lemma~\ref{lem:generic_simplification}, genericity of~$S\cap S_1$ is necessary and sufficient for~$S$ being generic. As shown in Proposition~\ref{prop:complex_full_rank_property_pol_mat},~$S\cap S_1$ is generic if, and only if,~$\ell\neq n+m.$ This completes the proof of the proposition.
\end{proof}

\subsection{Strongly controllable}

At last, we consider strongly controllable systems which are defined as follows.

\begin{Definition}[Strongly controllable~{\cite[Definition 2.1(j), p.\,10]{DAEs}}]\label{Def:strong_cont}
The system~\eqref{Def_eq:DAE} is \textit{strongly controllable} if, and only if,
\begin{align*}
\exists\, T>0~\forall\, x_0,x_T\in\mathbb{R}^n~\exists\, (x,u)\in\mathfrak{B}_{[E,A,B]}:Ex(0) = Ex_0\quad\wedge\quad Ex(T) = Ex_T.
\end{align*}
\end{Definition}

There is an algebraic criterion which we give next.

\begin{Lemma}[Algebraic criterion for strongly controllability~{\cite[Corollary 4.3, p.\,32]{DAEs}}]\label{lem:AlgCritStrCont}
The system~\eqref{Def_eq:DAE} is strongly controllable if, and only if,
\begin{align*}
\forall\,\lambda\in\mathbb{C}~\forall\, Z\in\mathbb{R}^{n\times n-\rk E}~\text{with}~\im_{\mathbb{R}} Z = \ker_{\mathbb{R}}E: \rk[E,A,B] = \rk[AZ,B] = \rk [\lambda E-A,B].
\end{align*}
\end{Lemma}



From this lemma we see that the results of Section 3 are not sufficient to prove that the set of strongly controllable systems is generic. We need some access to the kernel of $E$. The best way to do this is the well-known Gaussian elimination. We will use this method in the proof of the following proposition.

\begin{Proposition}\label{prop:StrongContGen} The set~$S = \set {(E,A,B)\in\Sigma_{\ell,n,m}\left\vert (E,A,B)~\text{strongly~controllable}\right.}$ is generic if, and only if,
\begin{align}\label{eq:useless}
\Big[n\leq\ell\leq m\quad\vee\quad \big(\ell<n\quad\wedge\quad 2\ell\leq n+m\big)\Big]\quad\wedge\quad \ell\neq n+m
\end{align}
\end{Proposition}
\begin{proof}
\begin{enumerate}[(A)]
\item We prove in a first step that the set
\begin{align*}
S' := \set{(E,A,B)\in\Sigma_{\ell,n,m}\,\big\vert\,\forall\, Z\in\mathbb{R}^{n\times n-\rk E}~\text{with}~\im_{\mathbb{R}} Z = \ker_{\mathbb{R}}E: \rk[E,A,B] = \rk[AZ,B]}
\end{align*}
is generic if, and only if,
\begin{align*}
\big(n\leq\ell~\wedge~\ell\leq m\big)\quad\vee\quad \big(\ell<n~\wedge~2\ell\leq n+m\big).
\end{align*}
With truth tables we find
\begin{align*}
A\longleftrightarrow \Big[\big(B\wedge C\big)\vee\big(\neg B\wedge D\big)\Big] \equiv \Big[B\longrightarrow\big(A\longleftrightarrow C\big)\Big]\wedge \Big[\neg B\longrightarrow\big(A\longleftrightarrow D\big)\Big].
\end{align*}
Hence we can equivalentely prove the statement: If~$n\leq\ell$, then~$S'$ is generic if, and only if,~$\ell\leq n$; and if~$\ell <n$, then~$S'$ is generic if, and only if,~$2\ell\leq n+m$.
We proceed in steps.
\item Let~$n\leq\ell.$ Then Proposition~\ref{prop:full_rank_property_block} implies that~$\set{(E,A,B)\in\Sigma_{\ell,n,m}\,\big\vert\,\rk E=n}$ is a generic set. The rank-nullity Theorem implies that
\begin{align*}
\forall\, (E,A,B)\in\Sigma_{\ell,n,m}: \ker E = \set{0} \iff \rk E = n.
\end{align*}
Thus~$S'$ is by Remark~\ref{Rem:InclusionOfGenericSets} and Corollary~\ref{cor:intersection_union1} generic if, and only if, the set
\begin{align*}
S_A := S'\cap \set{(E,A,B)\in\Sigma_{\ell,n,m}\,\big\vert\, \rk E = n}
\end{align*}
is generic. By Proposition~\ref{prop:full_rank_property_block}, the set
\begin{align*}
\set{(E,A,B)\in\Sigma_{\ell,n,m}\,\big\vert\,\rk[E,A,B] = \min\set{\ell,2n+m}}
\end{align*}
is generic. This yields with Remark~\ref{Rem:InclusionOfGenericSets} and Corollary~\ref{cor:intersection_union1} that~$S_A$ is generic if, and only if,
\begin{align*}
\set{(E,A,B)\in\Sigma_{\ell,n,m}\,\big\vert\,\rk E = n~\wedge~\rk B = \min\set{\ell,2n+m}} 
\end{align*}
is generic. The necessary and sufficient condition for this is genericity of 
\begin{align*}
\set{(E,A,B)\in\Sigma_{\ell,n,m}\,\big\vert\,\rk B = \min\set{\ell,2n+m}},
\end{align*}
which is, in view of Proposition~\ref{prop:full_rank_property_block}, the case if, and only if,
\begin{align*}
\min\set{\ell,2n+m} = \min\set{\ell,m}.
\end{align*}
This equality holds if, and only if,~$\ell\leq m$. This chain of equivalences proves the first part.
\item Let~$\ell < n.$ Then for any~$E\in\mathbb{R}^{\ell\times n}$ the kernel~$\ker E\neq\set{0}.$ In order to get some informations about~$\ker E$ we use Gaussian elimination.
\begin{enumerate}[(C1)]
\item We construct an operator~$T\in\mathbb{R}(x_1,\ldots,x_{\ell n})^{\ell\times n}$ for Gaussian elimination. Define 
\begin{align*}
T_k: \set{E\in\mathbb{R}^{\ell\times n}\,\big\vert\, E_{ii}\neq 0}\to\mathbb{R}^{\ell\times n},\quad \forall\,{i\in\underline \ell}~\forall\,{j\in\underline n}: (T_kE)_{i,j} = \begin{cases} E_{i,j}, & i\leq k,\\
E_{i,j}-\frac{E_{i,k}}{E_{k,k}}E_{k,j}, & i>k \end{cases}
\end{align*}
for~$k \in\underline {\ell-1}$. The operator~$T_k$ represents the~$k$-th step of Gaussian elimination without row-switching and fullfills for the canonical unit vectors~$e_1,\ldots,e_\ell\in\mathbb{R}^\ell$
\begin{align*}
\forall E\in\dom T_k: \big(T_kE)_{\cdot,k} \in\text{span}\set{e_1,\ldots,e_k}.
\end{align*} 
and
\begin{align}\label{eq:row_preserving}
\forall\, E\in\dom T_k~\forall\,i\in\underline{k}: \big(T_kE)_{i,\cdot} = E_{i,\cdot}
\end{align}
The operator~$T_k$ is an ``elimination operator'' and may be written as matrix multiplication of an~$E$-dependend transformation matrix and~$E$, namely
\begin{align*}
& T_k: \set{E\in\mathbb{R}^{\ell\times n}\,\big\vert\, E_{k,k}\neq 0}\to\mathbb{R}^{\ell\times n}, \\
& T\mapsto T_kE = \begin{bmatrix}
1\\
& \ddots\\
& & 1\\
& & -\frac{E_{k+1,k}}{E_{k,k}} & 1\\
& & -\frac{E_{k+2,k}}{E_{k,k}} & & 1\\
& & \vdots & & & \ddots\\
& & -\frac{E_{k+1,k}}{E_{k,k}} & & & & 1\\
\end{bmatrix}\cdot E = \begin{bmatrix}
E_{1,1} & & \cdot & \cdot & \cdot & & E_{1,n}\\
\vdots & & & & & & \vdots\\
E_{k,1} & & \cdot & \cdot & \cdot & & E_{k,n}\\
\star & \cdot & \star & 0 & \star & \cdot & \star\\
\vdots & & \vdots & \vdots & \vdots & & \vdots\\
\star & \cdot & \star & 0 & \star & \cdot & \star\\
\end{bmatrix}
\end{align*}
In passing, note that~$T_k\in\mathbb{R}(x_1.\ldots.x_{n^2})^{\ell\times n}$, i.e. for any~$i\in\underline{\ell},\,j\in\underline{n}$ and mapping 
\begin{align*}
\psi_{i,j,k}:\dom T_k\to\mathbb{R},\qquad E\mapsto (T_kE)_{i,j}
\end{align*}
there exists some~$\varphi_{i,j,k}(\cdot)\in\mathbb{R}(x_1,\ldots,x_{n^2})$ so that 
\begin{align*}
\psi_{i,j,k}(\cdot) = \varphi_{i,j,k}(\cdot)|_{\hspace*{-1.33ex}~_{| \dom T_k}}.
\end{align*}
Let $T_0$ be the identity operator and define
\begin{align}\label{eq:TheTOperator}
T := T_{\ell-1}\circ\ldots\circ T_0\in\mathbb{R}(x_1.\ldots.x_{n^2})^{\ell\times n}.
\end{align}
The domain of~$T$ is
\begin{align*}
\dom T = \set{E\in\mathbb{R}^{\ell\times n}\,\big\vert\,E_{1,1}\neq 0,\forall\,{i\in\set{1,\ldots \ell-2}}:\big((T_{i}\circ T_{i-1}\ldots\circ T_0)E\big)_{i,i}\neq 0}
\end{align*}
and for any~$E\in\dom T$ the matrix~$TE$ has the structure
\begin{align}\label{eq:structure_T_mapping}
TE = \begin{bmatrix}
\varepsilon_1 & \star & \star & \cdots & \star & \star & \cdots & \star\\
0 & \varepsilon_2 & \star & \cdots & \star & \star & \cdots & \star\\
0 & 0 & \varepsilon_3 & \cdots & \star & \star & \cdots & \star\\
\vdots & \vdots & \vdots & \ddots & \vdots & \vdots & \ddots & \vdots\\
0 & 0 & 0 & \cdots & \varepsilon_{\ell -1} & \star & \cdots & \star\\
0 & 0 & 0 & \cdots & 0 & \star & \cdots & \star
\end{bmatrix}\quad\text{for~some}~\varepsilon_1,\ldots,\varepsilon_{\ell-1}\in\mathbb{R}\setminus\set{0}.
\end{align}
Note that~$\dom T$ is the complement of an algebraic variety and this variety is proper since~$E_I := [I_\ell,0_{\ell\times(n-\ell)}]\in\dom T.$ Moreover we have
\begin{align}\label{eq:Eigenvalue_of_T}
TE_I = E_I.
\end{align}
\item We restrict the domain of~$T$ and get the operator
\begin{align*}
T': \dom T' := \set{E\in\dom T\,\big\vert\,(TE)_{\ell,\ell} \neq 0}\to\mathbb{R}^{\ell\times n},\quad E\mapsto
 TE.
\end{align*}
Then~$\dom T'$ is still the complement of a proper algebraic variety, since~$E_I\in\dom T'.$ By~\eqref{eq:row_preserving} and~\eqref{eq:structure_T_mapping}, the inclusions
\begin{align}\label{eq:inclusion_1}
\dom T'\subseteq\set{E\in\mathbb{R}^{\ell\times n}\,\big\vert\,\rk E = n}
\end{align}
holds true.
\item Let~$E\in\dom T'$. We construct a basis of~$\ker E.$ From Gaussian elimination we know
\begin{align*}
\ker E = \set{z\in\mathbb{R}^{n}\,\big\vert\,(T'E)z = 0}.
\end{align*}
Let~$i\in\underline{n-\ell}$ be arbitrary and define
\begin{align*}
z^i_{\ell+i}: \dom T'\to\mathbb{R},\quad E\mapsto 1
\end{align*}
and, for~$j\in\underline{n-\ell}\setminus\set{i}$
\begin{align*}
z^i_{\ell+j}: \dom T'\to\mathbb{R},\quad E\mapsto 0.
\end{align*}
Define now, for~$E\in\dom T',$ recursively~$z^i_{k}(E)$ for~$k = \ell,\ell-1,\ldots,1$ as the unique solution of
\begin{align}\label{eq:the_z_mappings}
(T'E)_{\cdot,k}^\top \begin{bmatrix}
0_{1\times (k-1)}\\z^i_k(E)\\\vdots\\z^i_n(E)
\end{bmatrix} = 0.
\end{align}
This leads, for any~$i\in\underline {n-\ell}$ and~$j\in\underline n,$ to the rational functions~$z_j^i\in\mathbb{R}(x_1,\ldots,x_{n^2})$ with
\begin{align*}
z^i_j: \dom T'\to\mathbb{R},\quad E\mapsto\begin{cases}
-\frac{1}{(T'E)_{j,j}}\Big((T'E)(j,\ell+i)+\sum_{k=j+1}^{\ell}(T'E)_{j,k}z^i_k(E)\Big), & j\leq \ell,\\
1, & j = \ell + i,\\
0, & \text{else}.
\end{cases}
\end{align*}
Define the mappings~$z^i(\cdot) := (z^i_1(\cdot),\ldots,z^i_n(\cdot))\in\mathbb{R}(x_1,\ldots,x_{n^2})^n$. From the structure~\eqref{eq:structure_T_mapping} and the definition of~$z^i(\cdot)$ in~\eqref{eq:the_z_mappings} we find that
\begin{align*}
\forall E\in\dom T'~\forall\,i\in\underline{n-\ell}: (T'E) z^i(E) = 0.
\end{align*}
The set~$\set{z^1(E),\ldots,z^{n-\ell}(E)}$ is linear independent and has cardinality~$n-\ell$. Since~$\rk E = \ell$ for any~$E\in\dom T'$, we have~$\dim (\ker E) = n-\ell$ and thus~$\set{z^1(E),\ldots,z^{n-\ell}(E)}$ is a basis of~$\ker E$. Define the mapping
\begin{align*}
Z: \dom T'\to\mathbb{R}^{n\times (n-\ell)},\quad E\mapsto [z^1(E),\ldots,z^{n-\ell}(E)].
\end{align*}
Then we find
\begin{align*}
\forall E\in\dom T': \im Z(E) = \ker E.
\end{align*}
\item 
Let~$\widehat{M}_1,\ldots,\widehat{M}_q$ be all minors of order~$\min\set{\ell,n+m-\ell}$ w.r.t.~$\mathbb{R}^{n\times (n+m-\ell)}$, define the mappings
\begin{align*}
\mathbb{R}(x_1,\ldots,x_{\ell(2n+m)})\ni M_i: \dom T'\times\mathbb{R}^{\ell\times n}\times\mathbb{R}^{\ell\times m},\quad (E,A,B)\mapsto\widehat{M}_i[AZ(E),B].
\end{align*}
and consider
\begin{align*}
S_1 & := \left(\bigcap_{i=1}^q M_i^{-1}(\set{0})\right)^c\cap\Big(\dom T'\times\mathbb{R}^{\ell\times n}\times\mathbb{R}^{\ell\times m}\Big).
\end{align*}
By Remark~\ref{rem:submatrix_rank}\,(i) and
\begin{align*}
\forall\, (E,A,B)\in \dom T'\times\mathbb{R}^{\ell\times n}\times\mathbb{R}^{\ell\times m}: \rk[AZ(E),B] \leq\set{\ell,n+m-\ell},
\end{align*}
the equality
\begin{align*}
S_1 = \set{(E,A,B)\in\dom T'\times\mathbb{R}^{\ell\times n}\times\mathbb{R}^{\ell\times m}\,\big\vert\,\rk[AZ(E),B] = \min\set{\ell,n+m-\ell}}
\end{align*}
holds true. Corollary~\ref{cor:rational_functions} yields that~$\dom T'$ is the complement of a proper algebraic variety. Thus Lemma~\ref{lem:product_of_generic_sets} yields that~$\dom T'\times\mathbb{R}^{\ell\times n}\times\mathbb{R}^{\ell\times m}$ is the complement of a proper algebraic variety. By Corollary~\ref{cor:rational_functions}, the set
\begin{align*}
\bigcap_{i=1}^q M_i^{-1}(\set{0})
\end{align*}
is an algebraic variety and thus Corollary~\ref{cor:intersection_union1} yields that~$S_1$ is the complement of an algebraic variety. Equation~\eqref{eq:Eigenvalue_of_T} yields
\begin{align*}
Z(E_I) = \begin{bmatrix} 0_{\ell\times(n-\ell)}\\I_{n-\ell} \end{bmatrix}~\text{for}~E_I = [I_\ell,0_{\ell\times(n-\ell)}].
\end{align*}
Choose~$A'_I\in\mathbb{R}^{\ell\times (n-\ell)}$ and~$B_I\in\mathbb{R}^{\ell\times m}$ such that~$\rk [A'_I,B_I] = \min\set{\ell,n+m-\ell}$ and set~$A_I := \left[0_{\ell\times \ell},A'_I\right]\in\mathbb{R}^{\ell\times n}.$ Then we get
\begin{align*}
\rk[A_IZ(E_I),B_I] = \rk[A_I',B_I] = \min\set{\ell,n+m-\ell}
\end{align*}
and hence~$(E_I,A_I,B_I)\in S_1$. This shows that~$S_1$ is the complement of a proper algebraic variety and thus generic.
\item In view of Proposition~\ref{prop:full_rank_property_block}, the set
\begin{align*}
S_2 := \set{(E,A,B)\in\Sigma_{\ell,n,m}\,\big\vert\,\rk[E,A,B] = \min\set{\ell,2n+m}}
\end{align*}
is generic and in step (C4) we have proven that~$S_1$ is generic. Thus Corollary~\ref{Cor:PartitionInGenericSets} implies that genericity of~$(S_1\cap S')$ and~$(S_2\cap S')$ is necessary for genericity of~$S'.$ By 
\begin{align*}
S_1\cap S_2\cap S' = (S_1\cap S')\cap (S_2\cap S')
\end{align*}
genericity of~$S_1\cap S_2\cap S'$ is necessary for genericity of~$S'$. By Remark~\ref{Rem:InclusionOfGenericSets}, genericity of~$S_1\cap S_2\cap S'$ is sufficient. It is evident that the equivalences
\begin{align*}
S_1\cap S_2\cap S'\neq\emptyset & \iff \min\set{\ell,2n+m} = \min\set{\ell,n-\ell+m}\\
& \iff 2\ell\leq n+m
\end{align*}
hold true. Thus we get
\begin{align*}
2\ell\leq n+m & \iff \min\set{\ell,2n+m} = \min\set{\ell,n-\ell+m}\\
& \iff S_1\cap S_2\subseteq S\\
& \implies S'~\text{is~generic}
\end{align*}
and
\begin{align*}
S'~\text{is~generic} & \implies S_1\cap S_2\cap S'\neq\emptyset\iff 2\ell\leq n+m.
\end{align*}
This proves that~$S'$ is generic if, and only if,~$2\ell\leq n+m$.
\end{enumerate}
\item From Proposition~\ref{Prop:cont_im_the_beh_sense_is_generic} we obtain that 
\begin{align*}
S'' := \set{(E,A,B)\in\Sigma_{\ell,n,m}\,\big\vert\,\forall\,\lambda\in\mathbb{C}:\rk[E,A,B] = \rk[\lambda E-A,B]}
\end{align*}
is a generic set if, and only if,~$\ell\neq n+m$. The equality~$S = S'\cap S''$ holds true by Lemma~\ref{lem:AlgCritStrCont} and Remark~\ref{Rem:InclusionOfGenericSets} yields that genericity of~$S$ is sufficient for genericity of~$S'$ and~$S''$. Conversely, by Corollary~\ref{cor:intersection_union1} genericity of~$S'$ and~$S''$ is sufficient for genericity of~$S$. This completes the proof of the Proposition.
\end{enumerate}
\end{proof}

The operators $T_k$ which we used to investigate gaussian elimination have some nice properties. Although these are not important for the proof (except the fact that $T_k$ is a matrix of rational functions), we write them down.

\begin{Remark}\label{rem:stuff}
Consider the operator~$T_k$ for some~$k\in\underline{\ell-1}$ from step (C1) in the proof of Proposition~\ref{prop:StrongContGen}. It is evident that~$T_k$ has a matrix representation which depends on~$E$. If we define the operator
\begin{align*}
M_k: \dom T_k\to\mathbb{R}^{\ell\times \ell},\qquad \forall i,j\in\ell: (M_kE)_{i,j} = \begin{cases} 1, & i = j,\\
-\frac{E_{i,j}}{E_{j,j}}, & i>k~~\wedge~~j = k,\\
0, & \text{else}, \end{cases}
\end{align*}
then we find that
\begin{align*}
\forall E\in\dom T_k: T_k E = (M_kE)\cdot E.
\end{align*}
Define for an arbitrary operator norm on~$\mathbb{R}^{\ell\times \ell}$ the mapping
\begin{align*}
c_k: \dom T_k\to\mathbb{R}>0,\qquad E\mapsto\norm{M_kE}.
\end{align*}
Then we find
\begin{align*}
\forall E\in\dom T_k: \norm{T_kE}\leq c_k(E)\norm{E}.
\end{align*}
For any~$E\in\dom T_k$, the matrix~$M_k E$ is invertible since it has determinant 1. Defining the mapping
\begin{align*}
C_k: \dom T_k\to\mathbb{R}>0,\qquad E\mapsto\norm{(M_kE)^{-1}}.
\end{align*}
 we find
\begin{align*}
\forall E\in\dom T: \norm{E} = \norm{(M_kE)^{-1}T_kE}\leq C_k(E)\norm{T_kE}.
\end{align*}
\end{Remark}

\newpage
\section{Stabilizability of differential-algebraic equations}
In this section we investigate stabilizability of DAEs. We distingush the following concepts:
\begin{itemize}
\item completely stabilizability,
\item strongly stabilizability and
\item stabilizable in the behavioural sense.
\end{itemize}
As for controllability, there are algebraic characterizations for these stabilizability concepts. Thus we are able to observe whether and in which cases they are generic.

\subsection{Completely stabilizability}

First we consider completely stabilizable systems.

\begin{Definition}[Completely stabilizable~{\cite[Definition 2.1(h), p.\,9]{DAEs}}]
The system~\eqref{Def_eq:DAE} is \textit{completely stabilizable} if, and only if:
\begin{align*}
\forall\, x_0\in\mathbb{R}^n~\exists\, (x,u)\in\mathfrak{B}_{[E,A,B]}: x(0) = x_0\quad\wedge\quad \lim_{t\to\infty}\text{ess\,sup}\, x|_{\hspace*{-1.33ex}~_{| (t,\infty)}} = 0
\end{align*}
\end{Definition}

This Definition and the Definitions~\ref{Def:freely_initializable} and~\ref{Def:Completely_controllable} yield immideately the following inclusions.

\begin{Lemma}\label{Cor:criteria_for_compl_stab}
\begin{align*}
\set{(E,A,B)\in\Sigma_{\ell,n,m}\,\big\vert\,\eqref{Def_eq:DAE}~\text{compl. cont.}}& \subseteq\set{(E,A,B)\in\Sigma_{\ell,n,m}\,\big\vert\,\eqref{Def_eq:DAE}~\text{compl. stab.}}\\&\subseteq\set{(E,A,B)\in\Sigma_{\ell,n,m}\,\big\vert\,\eqref{Def_eq:DAE}~\text{freely init.}}.
\end{align*}
\end{Lemma}

Berger and Reis give an algebraic criterion for completely stabilizable systems.

\begin{Lemma}[{Algebraic criterion for completely stabilizablility~\cite[Corollary 4.3, p.\,32]{DAEs}}]
The system~\eqref{Def_eq:DAE} is completely stabilizable if, and only if,
\begin{align*}
\forall\,\lambda\in\overline{\mathbb{C}}_{+}: \rk[E,A,B] = \rk[E,B] =\rk[\lambda E-A,B].
\end{align*}
\end{Lemma}

We can now formulate a result regarding genericity.

\begin{Proposition}\label{prop:comp_stab}
$S = \set{(E,A,B)\in\Sigma_{\ell,n,m}\,\big\vert\,\eqref{Def_eq:DAE}~\text{completely stabilizable}}$ is generic if, and only if~$\ell < n + m.$
\end{Proposition}
\begin{proof}
\begin{enumerate}
\item[$\impliedby$] If~$\ell<n+m$, then by Proposition~\ref{prop:comp_cont}~$\set{(E,A,B)\in\Sigma_{\ell,n,m}\,\big\vert\,(E,A,B)~\text{compl. cont.}}$ is generic and hence the statement holds by Remark~\ref{Rem:InclusionOfGenericSets} and Lemma~\ref{Cor:criteria_for_compl_stab}.
\item[$\implies$]
By Lemma~\ref{Cor:criteria_for_compl_stab} and Proposition~\ref{prop:freely_initializable}, we find~$\ell\leq n+m$. In Proposition~\ref{prop:freely_initializable} we have shown that
\begin{align*}
S' = \set{(E,A,B)\in\Sigma_{\ell,n,m}\,\big\vert\,\rk[E,A,B] = \rk[E,B] = \ell}
\end{align*}
is generic. By Corollary~\ref{cor:intersection_union1}, this yields that
\begin{align*}
S\cap S' = \set{(E,A,B)\in\Sigma_{\ell,n,m}\,\big\vert\,\forall\lambda\in\overline{\mathbb{C}}_+: \rk[\lambda E-A,B] = \ell}
\end{align*}
is generic. By Proposition~\ref{prop:complex_full_rank_property_pol_mat}, this yields~$\ell\neq n+m$ and hence we get~$\ell <n+m$.
\end{enumerate}
\end{proof}

\subsection{Strong stabilizability}

Next we consider strongly stabilizable systems

\begin{Definition}[Strongly stabilizable~{\cite[Definition 2.1(k), p.\,10]{DAEs}}]
The system~\eqref{Def_eq:DAE} is \textit{strongly stabilizable} if, and only if,
\begin{align*}
\forall\, x_0\in\mathbb{R}^n~\exists\, (x,u)\in\mathfrak{B}_{[E,A,B]}: Ex(0) = Ex_0\quad\wedge\quad \lim_{t\to\infty}Ex(t) = 0
\end{align*}
\end{Definition}

With this Definition we find the following inclusions.

\begin{Lemma}[{see~\cite[Proposition 2.4, p.\,14]{DAEs}}]\label{Cor:criteria_for_str_stab}
\begin{align*}
\set{(E,A,B)\in\Sigma_{\ell,n,m}\,\big\vert\,\eqref{Def_eq:DAE}~\text{compl. stab.}}& \subseteq\set{(E,A,B)\in\Sigma_{\ell,n,m}\,\big\vert\,\eqref{Def_eq:DAE}~\text{strong. stab.}}\\&\subseteq\set{(E,A,B)\in\Sigma_{\ell,n,m}\,\big\vert\,\eqref{Def_eq:DAE}~\text{impulse cont.}}.
\end{align*}
\end{Lemma}

There is an algebraic criterion, namely the following.

\begin{Lemma}[{Algebraic criterion for strongly stabilizable~\cite[]{DAEs}}]
The system~\eqref{Def_eq:DAE} is strongly stabilizable if, and only if,
\begin{align*}
\forall\,\lambda\in\overline{\mathbb{C}}_+~\forall\,Z~\text{with}~\im Z = \ker E:\quad\rk[E,A,B] = \rk[E,AZ,B] = \rk[\lambda E-A,B].
\end{align*}
\end{Lemma}

We can use both Lemma 6.6 and Corollary 6.7 to prove our result regarding genericity.

\begin{Proposition}
$S = \set{(E,A,B)\in\Sigma_{\ell,n,m}\,\big\vert\,\eqref{Def_eq:DAE}~\text{strongly stabilizable}}$ is generic if, and only if~$\ell < n + m.$
\end{Proposition}
\begin{proof}
\begin{enumerate}
\item[$\impliedby$] If~$\ell<n+m$, then by Proposition~\ref{prop:comp_stab} the set~$\set{(E,A,B)\in\Sigma_{\ell,n,m}\,\big\vert\,\eqref{Def_eq:DAE}~\text{compl. stab.}}$ is generic. The statement holds by Remark~\ref{Rem:InclusionOfGenericSets} and Lemma~\ref{Cor:criteria_for_str_stab}.
\item[$\implies$] If~$\ell = n+m$, then by Corollary~\ref{prop:positive_complex_full_rank_property} we find that
\begin{align*}
\set{(E,A,B)\in\Sigma_{\ell,n,m}\,\big\vert\,\exists\,\lambda\in\overline{\mathbb{C}}_+: \rk[\lambda E-A,B]<n}
\end{align*}
is not a nullset. The set
\begin{align*}
\set{(E,A,B)\in\Sigma_{\ell,n,m}\,\big\vert\,\rk[E,A,B] = n}
\end{align*}
is by Proposition~\ref{prop:full_rank_property_block} generic. This yields that~$S^c$ is not a nullset and Corollary~\ref{cor:intersection_union1} yields that~$S$ is not generic.

If~$\ell>n+m$, then we find~$\ell>n$ and thus the set
\begin{align*}
\set{(E,A,B)\in\Sigma_{\ell,n,m}\,\big\vert\,\ker E = \set{0}}
\end{align*}
is generic by Proposition~\ref{prop:full_rank_property_block} and Lemma~\ref{lem:product_of_generic_sets}. Furthermore the sets
\begin{align*}
\set{(E,A,B)\in\Sigma_{\ell,n,m}\,\big\vert\,\rk[E,A,B] > n+m}
\end{align*}
and
\begin{align*}
\set{(E,A,B)\in\Sigma_{\ell,n,m}\,\big\vert\,\rk[E,B] = n+m}
\end{align*}
are generic. By Corollary~\ref{cor:intersection_union1}, we find that
\begin{align*}
S' := \set{(E,A,B)\in\Sigma_{\ell,n,m}\,\big\vert\,\forall\,Z~\text{with}~\im Z = \ker E:\rk[E,A,B]\neq \rk[E,AZ,B]}
\end{align*}
is generic. Remark~\ref{Rem:InclusionOfGenericSets} and the inclusion~$S'\subseteq S^c$ yield that~$S^c$ is generic. By Corollary~\ref{Cor:PartitionInGenericSets},~$S$ is not generic.
\end{enumerate}
\end{proof}

\subsection{Stabilizability in the behavioural sense}

At last we consider stabilizability in the behavioural sense. Our observations in this case are as straight forward as for the other stabilizability concepts.

\begin{Definition}[Stabilizable in the behavioural sense~{\cite[Definition 2.1(k), p.\,10]{DAEs}}]
The system~\eqref{Def_eq:DAE} is \textit{stabilizable in the behavioural sense} if, and only if,
\begin{align*}
\forall\, (x,u)\in\mathfrak{B}_{[E,A,B]}~\exists\, (x_0,u_0)\in\mathfrak{B}_{[E,A,B]}\cap\left(\mathcal{W}^{1,1}_{\text{loc}}(\mathbb{R},\mathbb{R}^n)\times\mathcal{W}^{1,1}_{\text{loc}}(\mathbb{R},\mathbb{R}^m)\right):\\
\left(\forall\, t<0:(x(t),u(t)) = (x_0(t),u_0(t))\right)\quad\wedge\quad\lim_{t\to\infty}(x_0(t),u_0(t)) = 0.
\end{align*}
\end{Definition}

We find immideately the following inclusions.

\begin{Lemma}\label{Cor:Criteria_behaviouralSense}
\begin{align*}
\set{(E,A,B)\in\Sigma_{\ell,n,m}\,\big\vert\,\eqref{Def_eq:DAE}~\text{cont. in beh. s.}}& \subseteq\set{(E,A,B)\in\Sigma_{\ell,n,m}\,\big\vert\,\eqref{Def_eq:DAE}~\text{stab. in beh. s.}}
\end{align*}
\end{Lemma}

\begin{Lemma}[{Algebraic criterion for stabilizable in the behavioural sense~\cite[Corollary 4.3, p.\,32]{DAEs}}]
The system~\eqref{Def_eq:DAE} is stabilizable in the behavioural sense if, and only if, 
\begin{align*}
\forall\,\lambda\in\overline{\mathbb{C}}_+: \rk_{\mathbb{R}(x)}[xE-A,B] = \rk_{\mathbb{C}}[\lambda E-A,B].
\end{align*}
\end{Lemma}

\begin{Proposition}
$S = \set{(E,A,B)\in\Sigma_{\ell,n,m}\,\big\vert\,\eqref{Def_eq:DAE}~\text{stabilizable in the behavioural sense}}$ is generic if, and only if~$\ell \neq n + m.$
\end{Proposition}
\begin{proof}
Proposition~\ref{Prop:gen_full_rk_pol_block} yields that the set
\begin{align*}
S_1 := \set{(E,A,B)\in\Sigma_{\ell,n,m}\,\big\vert\,\rk_{\mathbb{R}(x)}[xE-A,B] = \min\set{\ell,n+m}}
\end{align*}
is generic. By Lemma~\ref{lem:generic_simplification}, genericity of
\begin{align*}
S_2 := \set{(E,A,B)\in\Sigma_{\ell,n,m}\,\big\vert\,\forall\,\lambda\in\overline{\mathbb{C}}_+: \rk_{\mathbb{C}}[\lambda E-A,B] = \min\set{\ell,n+m}}
\end{align*}
is a necessary and sufficient condition for genericity of~$S.$ By Proposition~\ref{prop:positive_complex_full_rank_property},~$S_2$ is generic if, and only if,~$\ell\neq n+m$.
\end{proof}

\newpage

\section{Outlook}

In the case~$\ell = n$ we find with Proposition~\ref{prop:full_rank_property_block} and Lemma~\ref{lem:product_of_generic_sets} that
\begin{align*}
S = \mathcal{GL}_n(\mathbb{R})\times\mathbb{R}^{n\times n}\times\mathbb{R}^{n\times m}
\end{align*}
is generic. For any~$(E,A,B)\in S$ the DAE
\begin{align}\label{eq:eine_DAE}
\tfrac{\mathrm{d}}{\mathrm{d}t}(Ex) = Ax+Bu
\end{align}
is equivalent to the ODE
\begin{align}\label{eq:eine_ODE}
\tfrac{\mathrm{d}}{\mathrm{d}t}x = E^{-1}Ax + E^{-1}Bu
\end{align}
and hence~\eqref{eq:eine_DAE} is controllable in any sense, if~\eqref{eq:eine_ODE} is controllable. From the Kalman-criterion (Lemma~\ref{lem:Kalmam_rank_condition}) we obtain that the latter is the case if, and only if
\begin{align*}
n = \rk[E^{-1}B,E^{-1}AE^{-1}B,(E^{-1}A)^2E^{-1}B,\ldots,(E^{-1}A)^{n-1}E^{-1}B].
\end{align*}
Let~$\widehat{M}$ be a minor of order~$n$ w.r.t.~$\mathbb{R}^{n\times nm}$ and
\begin{align*}
M: S\to\mathbb{R},\quad (E,A,B)\mapsto \widehat{M}([E^{-1}B,E^{-1}AE^{-1}B,(E^{-1}A)^2E^{-1}B,\ldots,(E^{-1}A)^{n-1}E^{-1}B]).
\end{align*}
Then we find that~$M^{-1}(\set{0})$ is a algebraic variety, which is proper since~$\set{I_n}\times\Sigma_{n,m}^{\text{cont}}\subseteq\left(M^{-1}(\set{0})\right)^c$. Since 
\begin{align*}
\left(M^{-1}(\set{0})\right)^c\subseteq\set{(E,A,B)\in\Sigma_{n,n,m}\,\big\vert\,~\eqref{eq:eine_DAE}~\text{``controllable''}},
\end{align*}
the latter set is generic for any controllability concept from Section 5 and we have no information about the DAEs with singular~$E$. Hence we would rather consider the smaller set~$S^c$ and investigate the properties of
\begin{align}\label{eq:eine_Menge}
\set{(E,A,B)\in S^c\,\big\vert\,~\eqref{eq:eine_DAE}~\text{``controllable''}}.
\end{align}
Since~$S^c\not\cong\mathbb{R}^{k}$ for any~$k\in\mathbb{N}$, the concept of generic sets can not be applied. In Remark~\ref{rem:Zariski_topology} we have given an alternative definition of generic sets as sets which contain some non-empty Zariski-open set. We could introduce the notion of ``relative generic'' sets of some set~$V\subseteq\mathbb{F}^n$ as sets which contain some nonempty set that is open in the relative topology on~$V$ induced by the Zariski-topology (i.e. some relative Zariski-open set). But this concept would contain no information for many sets, e.g. if~$V$ is a countable set. Then any nonempty subset of~$V$ is some nonempty relative Zariski-open set. So we should add one more condition for a set in order to be relative generic in~$V$. In Lemma~\ref{lem:some_uninteresting_lemma} we have seen that Zariski open sets are dense w.r.t. the Euclidean topology. Thus we could say a set~$S'\subseteq V$ is \textit{relative generic} in~$V$, if~$S'$ contains some relative Zariski-open set which is dense w.r.t.~the Euclidean topology.

The next step is to verify that this concept of relative genericity is useful to describe the relation of~$S'\subseteq\mathbb{F}^n$ and~$V\subseteq\mathbb{F}^n$, if~$V\not\cong\mathbb{F}^k$ for any~$k\in\mathbb{N}$. Then we check whether the set from~\eqref{eq:eine_Menge} is relative generic in~$S^c$, if we replace the term ``controllable'' with the controllability concepts from Section 5 and the stabilizability concepts from Section 6.

Unfortunately, this is beyond the scope of this thesis.

\newpage

\bibliography{Literatur}

\begin{thebibliography}{BIRT20}

\bibitem[AE06a]{AE_I}
Herbert Amann and Joachim Escher.
\newblock {\em Analysis I}.
\newblock Birkhäuser Verlag, Basel, Boston, Berlin, 2nd edition, 2006.

\bibitem[AE06b]{AE_II}
Herbert Amann and Joachim Escher.
\newblock {\em Analysis II}.
\newblock Birkhäuser Verlag, Basel, Boston, Berlin, 3rd edition, 2006.

\bibitem[Bal56]{AlgVar}
M.~Baldassari.
\newblock {\em {Algebraic Varieties}}.
\newblock Springer-Verlag, Berlin, Göttingen, Heidelberg, 1956.

\bibitem[Bau92]{MassIntTheorie}
Heinz Bauer.
\newblock {\em Mass- und Integrationstheorie}.
\newblock De Gruyter, Berlin, New York, 2nd edition, 1992.

\bibitem[BIRT20]{MatPenc}
Thomas Berger, Achim Ilchmann, Timo Reis, and Stephan Trenn.
\newblock {\em Matrix Pencils}.
\newblock 2020.
\newblock unpublished notes.

\bibitem[BKP16]{Belur2}
Madhu~N. Belur, Rachel~Kalpana Kalaimani, and C.~Praagman.
\newblock {Impulse Controllability: From Descriptor Systems to Higher Order
  DAEs}.
\newblock {\em IEEE Transaction on Automatic Control}, 61:2463--2472, 2016.

\bibitem[BR13]{DAEs}
Thomas Berger and Timo Reis.
\newblock Controllability of linear differential-algebraic systems -- a survey.
\newblock In Achim Ilchmann and Timo Reis, editors, {\em Surveys in
  Differential-Algebraic Equations I}. Springer Verlag, 2013.

\bibitem[BS19]{Belur}
Madhu~N. Belur and Shiva Shankar.
\newblock The persistence of impulse controllability.
\newblock {\em Mathe\-matics of Control, Signals, and Systems}, 31:487--501,
  2019.

\bibitem[Fed69]{GeomMeasTheory}
Herbert Federer.
\newblock {\em Geometric Measure Theory}.
\newblock Springer-Verlag, Berlin, Heidelberg, New York, 1969.

\bibitem[Fis10]{LinAlg}
Gerd Fischer.
\newblock {\em Lineare Algebra}.
\newblock Vieweg+Teubner, Wiesbaden, 17th edition, 2010.

\bibitem[Fuh96]{Algebra}
Paul~A. Fuhrmann.
\newblock {\em A Polynomial Approach to Linear Algebra}.
\newblock Springer Verlag, New York, 1996.

\bibitem[Hot19]{Hot}
Thomas Hotz.
\newblock {\em Ma\ss - und Wahrscheinlichkeitstheorie}.
\newblock Ilmenau, winter semester 2018/19.
\newblock unpublished lecture notes.

\bibitem[Hur95]{Hurwitz}
Adolf Hurwitz.
\newblock {Ueber die Bedingungen, unter welchen eine Gleichung nur Wurzeln mit
  negativen reellen Theilen besitzt}.
\newblock {\em Mathematische Annalen}, 46:273--284, 1895.

\bibitem[Lan02]{Alg}
Serge Lang.
\newblock {\em Algebra}.
\newblock Springer Verlag, New York, 3rd edition, 2002.

\bibitem[LR14]{Ryan}
Hartmut Logemann and Eugene~P. Ryan.
\newblock {\em Ordinary Differential Equations}.
\newblock Springer Verlag, London, Heidelberg, New York, Dordrecht, 2014.

\bibitem[MS95]{IntOptCont}
Jack Macki and Aaron Strauss.
\newblock {\em Introduction to Optimal Control Theory}.
\newblock Springer-Verlag, New York et al., 2nd edition, 1995.

\bibitem[PW97]{Polderman}
Jan~Willem Polderman and Jan~C. Willems.
\newblock {\em Introduction to Mathematical Systems Theory}.
\newblock Springer Verlag, London, Heidelberg, New York, 1997.

\bibitem[Rei98]{UndAlgGeom}
Miles Reid.
\newblock {\em Undergraduate Algebraic Geometry}.
\newblock Cambridge University Press, Cambridge, 1998.

\bibitem[TSH01]{ContTheorLinSyst}
Harry~L. Trentelman, Anton~A. Stoorvogel, and Malo Hautus.
\newblock {\em Control Theory for Linear Systems}.
\newblock Springer, London, 2001.

\bibitem[Wal78]{AlgCurv}
Robert~J. Walker.
\newblock {\em Algebraic Curves}.
\newblock Springer Verlag, New York et al., 2nd edition, 1978.

\bibitem[Wer18]{FunkAny}
Dirk Werner.
\newblock {\em Funktionalanalysis}.
\newblock Springer-Verlag, Berlin, 8th edition, 2018.

\bibitem[Won85]{Wonham}
W.~Murray Wonham.
\newblock {\em Linear Multivariable Control}.
\newblock Springer Verlag, New York, Berlin, Heidelberg, London, Paris, Tokyo,
  Hong Kong, Barcelona, 2nd edition, 1985.

\end{thebibliography}

\bibliographystyle{alpha}

\end{document}